\documentclass[preprint,3p,11pt,authoryear]{elsarticle}

\makeatletter
\def\ps@pprintTitle{%
 \let\@oddhead\@empty
 \let\@evenhead\@empty
 \def\@oddfoot{\centerline{\thepage}}%
 \let\@evenfoot\@oddfoot}
\makeatother

\usepackage[OT1]{fontenc}
\usepackage{amsthm,amsmath,amsfonts,amssymb,natbib}
\usepackage{mathtools}
\mathtoolsset{showonlyrefs=true} 
\usepackage[colorlinks]{hyperref}

\usepackage{multirow}

\usepackage{geometry}
\usepackage[tableposition=t]{caption} 
\usepackage{longtable} 
\usepackage{xcolor} 
\usepackage{pdflscape} 
\usepackage{graphicx} 
\usepackage{dsfont} 
\usepackage{appendix} 
\usepackage{etoolbox} 
\usepackage{dsfont} 
\usepackage{slashbox} 
\usepackage{mathrsfs} 
\usepackage{cancel} 
\usepackage{listings} 
\usepackage{subcaption} 
\patchcmd{\thebibliography}{\clubpenalty4000}{\clubpenalty10000}{}{}
\patchcmd{\thebibliography}{\widowpenalty4000}{\clubpenalty10000}{}{}

\usepackage{algpseudocode,algorithm}

\newtheorem{theorem}{Theorem}[section]
\newtheorem{proposition}[theorem]{Proposition}
\newtheorem{lemma}[theorem]{Lemma}
\newtheorem{corollary}[theorem]{Corollary}
\newtheorem{definition}[theorem]{Definition}
\newtheorem{remark}[theorem]{Remark}
\newtheorem*{notation*}{Notation}

\newcommand{\N}{\mathbb{N}}
\newcommand{\R}{\mathbb{R}}
\newcommand{\PP}{\mathbb{P}}
\newcommand{\EE}{\mathbb{E}}
\newcommand{\VV}{\mathbb{V}\mathrm{ar}}
\newcommand{\Xn}{\mathbf{X}_n}
\newcommand{\nvert}[0]{\, \vert \, }
\newcommand{\bb}[1]{\boldsymbol{#1}}
\newcommand{\ind}[1]{\mathds{#1}}
\newcommand{\rd}{{\rm d}}
\newcommand{\leqdef}{\vcentcolon=}
\newcommand{\reqdef}{=\vcentcolon}
\newcommand{\e}{\varepsilon}

\newcommand{\cone}{c_{1\hspace{-0.3mm},\hspace{-0.3mm}\lambda}}
\newcommand{\aone}{\alpha_{\hspace{-0.3mm}1\hspace{-0.3mm},\hspace{-0.3mm}\lambda}}
\newcommand{\ctwo}{c_{2\hspace{-0.3mm},\hspace{-0.3mm}\lambda}}
\newcommand{\atwo}{\alpha_{\hspace{-0.3mm}2\hspace{-0.3mm},\hspace{-0.3mm}\lambda}}
\newcommand{\cthree}{c_{3\hspace{-0.3mm},\hspace{-0.3mm}\lambda}}
\newcommand{\athree}{\alpha_{\hspace{-0.3mm}3\hspace{-0.3mm},\hspace{-0.3mm}\lambda}}
\newcommand{\cfour}{c_{4\hspace{-0.3mm},\hspace{-0.3mm}\lambda}}
\newcommand{\afour}{\alpha_{\hspace{-0.3mm}4\hspace{-0.3mm},\hspace{-0.3mm}\lambda}}
\newcommand{\Vone}{V_{1\hspace{-0.2mm},\hspace{-0.2mm}\lambda}}
\newcommand{\Vtwo}{V_{2\hspace{-0.2mm},\hspace{-0.2mm}\lambda}}

\usepackage{xspace}

\allowdisplaybreaks

\begin{document}

\vspace{-5cm}
    \noindent
    \fbox{
    \begin{minipage}{41em}
        \small
        This manuscript was accepted for publication in Statistics (Taylor \& Francis).
        This version {\it may differ} from the published version (\href{https://doi.org/10.1080/02331888.2022.2144859}{doi:10.1080/02331888.2022.2144859}) in typographic details.
    \end{minipage}
    }

    \vspace{7mm}

\begin{frontmatter}

\title{Goodness-of-fit tests for Laplace, Gaussian and exponential power distributions based on $\lambda$-th power skewness and kurtosis}

\author[a1]{Alain Desgagn\'e\corref{cor1}}
\cortext[cor1]{Corresponding author.}
\ead{desgagne.alain@uqam.ca}
\address[a1]{D\'epartement de Math\'ematiques, Universit\'e du Qu\'ebec \`a Montr\'eal, Montr\'eal, Canada.}

\author[a2,a3,a4,a5]{Pierre Lafaye de Micheaux}
\address[a2]{AMIS, Universit\'{e} Paul-\!Val\'{e}ry Montpellier 3, Montpellier, France.}
\address[a3]{PreMeDICaL - Precision Medicine by Data Integration and Causal Learning, Inria Sophia Antipolis, France.}
\address[a4]{Desbrest Institute of Epidemiology and Public Health, Universit\'e de Montpellier, Montpellier, France.}
\address[a5]{School of Mathematics and Statistics, UNSW Sydney, NSW, Australia.}

\author[a6,a7,a8]{Fr\'ed\'eric Ouimet}
\address[a6]{Division of Physics, Mathematics and Astronomy, California Institute of Technology, Pasadena, USA.}
\address[a7]{Department of Mathematics and Statistics, McGill University, Montreal, Canada.}
\address[a8]{Centre de recherches math\'ematiques, Universit\'e de Montr\'eal, Montr\'eal, Canada.}

\begin{abstract}
    Temperature data, like many other measurements in quantitative fields, are usually modeled using a normal distribution. However, some distributions can offer a better fit while avoiding underestimation of tail event probabilities. To this point, we extend Pearson's notions of skewness and kurtosis to build a powerful family of goodness-of-fit tests based on Rao's score for the exponential power distribution $\mathrm{EPD}_{\lambda}(\mu,\sigma)$, including tests for normality and Laplacity when $\lambda$ is set to 1 or 2. We find the asymptotic distribution of our test statistic, which is the sum of the squares of two $Z$-scores,
     under the null and under local alternatives. We also develop an innovative regression strategy to obtain $Z$-scores
     that are nearly independent and distributed as standard Gaussians, resulting in a $\chi_2^2$ distribution valid for any sample size (up to very high precision for $n\geq 20$).
     The case $\lambda=1$ leads to a powerful test of fit for the Laplace($\mu,\sigma$) distribution, whose empirical power is superior to all $39$ competitors in the literature, over a wide range of $400$ alternatives. Theoretical proofs in this case are particularly challenging and substantial. We applied our tests to three temperature datasets. The new tests are implemented in the \texttt{R} package \texttt{PoweR}.
\end{abstract}


\end{frontmatter}

\section{Introduction}\label{sec:intro}

In many fields of application, the Gaussian distribution is the primary choice to model a symmetric dataset. However, further analysis using for example a QQ-plot often reveals more (or sometimes fewer) extreme values than expected in a Gaussian model. This is problematic because it can lead to a poor estimation of the probability of extreme events. The exponential power distribution $\mathrm{EPD}_{\lambda}(\mu,\sigma)$, whose density function is proportional to $\exp(-|(x-\mu)/\sigma|^\lambda/\lambda)$, is then an interesting alternative thanks to its wide range of tails' behaviour. In particular, the Laplace distribution ($\lambda=1$), with its heavier tails, is a popular alternative to the Gaussian distribution ($\lambda=2$) for modelling observations that show a higher rate of extreme values. Goodness-of-fit tests for normality, Laplacity or in general for the $\mathrm{EPD}_{\lambda}(\mu,\sigma)$ with a fixed value of $\lambda$, are then necessary to assess the validity of the chosen model.

Our contribution in this paper is to propose a unified testing approach with a family of goodness-of-fit tests for the exponential power distribution $\mathrm{EPD}_{\lambda}(\mu,\sigma)$ with $\lambda\geq 1$ and unknown location and scale parameters $\mu$ and $\sigma$, including tests for the Laplace and Gaussian distributions if $\lambda$ is set to 1 or 2, respectively. Specifically, we use Rao's score test \citep{Rao1948} -- also known as Lagrange multiplier test -- on the asymmetric power distribution (APD) introduced by \cite{Komunjer2007}. Asymptotically, we obtain tests that are equivalent to the Wald and likelihood ratio tests, which are the most powerful for small deviations. Since the APD family combines the wide range of exponential tail behaviours provided by the EPD family with different levels of asymmetry, the resulting tests are based on two asymptotically independent measures that we named `$\lambda$-th-power skewness' and `$\lambda$-th-power kurtosis', similar to the normality test of \cite{Jarque_Bera_1987} or the moment-based tests of \cite{Natarajan2004} for the inverse Gaussian.

We go further by adapting our asymptotic tests for all sample sizes and $\lambda\geq 1$, using simulation and regression techniques. It is well known that skewed distributions are often associated with heavy tails for small samples. We therefore designed the `$\lambda$-th-power net kurtosis' to replace the $\lambda$-th-power kurtosis, the former being virtually independent of the $\lambda$-th-power skewness for all sample sizes. This is a crucial step in building powerful tests. These two measures are then transformed into $Z$-scores, which are very close to being normally distributed under the null hypothesis and could be used as specific test statistics against symmetric and asymmetric alternatives, respectively.  We obtain omnibus test statistics by adding the square of the two $Z$-scores, which is very close to being $\chi_2^2$ distributed under the null hypothesis. Theoretical results are also provided, showing that the asymptotic distribution of the omnibus test statistics is a chi-square with two degrees of freedom ($\chi_2^2$) under the null hypothesis and a noncentral $\chi_2^2$ under local alternatives. The proofs, given in the Supplementary Material B, are particularly challenging and substantial.

Our family of tests has characteristics that stand out in several respects. First, since our tests are built from two nearly independent components for all sample sizes, the amount of information in the test statistics is maximized. This results in powerful omnibus tests. In particular, in \cite{Laplace_tests_compilation}, we performed a comprehensive empirical power comparison of 40 goodness-of-fit tests for the univariate Laplace distribution against 400 alternatives and our Laplace omnibus test performed best. Second, the distribution of our test statistics under the null hypothesis can be approximated very closely, for all sample sizes (up to very high precision for $n\geq 20$), either by a $\chi_2^2$ for our omnibus tests or by the $\mathcal{N}(0,1)$ for the $Z$-scores. As a result, we obtain turnkey goodness-of-fit tests that allow accurate and easy calculation of critical values and \textit{p}-values without the need to rely on simulated quantiles or tables, which can certainly facilitate their acceptance and implementation. Note that this is a rare feature in the Laplace test literature. Third, thanks to the $\lambda$-th-power skewness and net kurtosis -- which extend Pearson's notions of skewness and kurtosis --, the rejection of the null hypothesis is accompanied with a direct interpretation: the distribution can be skewed to the right or to the left, exhibit heavy or light tails. Fourth, when we know that the distribution of the random variable is symmetric, we can take advantage of this information to increase power. To achieve this, we propose to use the $Z$-score based on the $\lambda$-th-power net kurtosis as a test statistic against symmetric alternatives.

In this paper, we focus on applications to temperature data. The Gaussian distribution is used for example by \cite{Allen_1996} for human body temperatures, by \cite{Chamberlain_et_al_1995} for ear temperatures measured using an infrared emission detection thermometer, by \cite{Issautier_et_al_1998} for core electron temperatures, and by \cite{Pardo_et_al_2017} for sea surface temperatures, while the Gaussian model is questioned by \cite{DeWitt_Friedman_1979} in the context of body temperatures of ectotherms and by \cite{Schoenau_Kehrig_1990} for heating or cooling degree days in buildings. In Section~\ref{sec:example}, we analyse a temperature dataset containing errors defined as the difference between observed ocean surface temperatures and forecasted values given by a geostatistical model \citep{Gel_et_al_2007}. We show that the Laplace distribution, and more so the $\mathrm{EPD}_{1.5}(\mu,\sigma)$, is a better model than the Gaussian distribution. We also study the Gaussian and Laplace modelling of a large dataset ($n=2,037$) of domestic refrigerator temperatures collected as part of a study to help consumers reduce bacterial growth and ensure the quality and safety of food products stored at home for U.S.\ households \citep{Kosa_et_al_2007}. We finally investigate a London temperature time series using a moving-average model with Gaussian and Laplace innovations \citep{Piggott_1980,Shea_1987}.

The remainder of the paper is structured as follows. In Section~\ref{sec:preliminaries}, we define reparametrized versions of the EPD and APD families, the null and alternative hypotheses of our tests, $\lambda$-th-power skewness, kurtosis and net kurtosis. In Section~\ref{sec:asymp.dist.H0}, we present our family of goodness-of-fit tests for the EPD. We establish the asymptotic test statistics and then adapt them for all sample sizes using simulation and regression techniques. Their asymptotic distributions under the null hypothesis are also given. In Section~\ref{sec:asymp.dist.H1}, we present the asymptotic distributions of the test statistics under local alternatives, which allows us to obtain asymptotic power curves. Specific cases of goodness-of-fit tests for Laplace and Gaussian distributions are presented in Section~\ref{sec:Laplace.Gaussian}. In Section~\ref{sec:example}, as described above, we investigate Laplace, Gaussian and EPD modelling of three temperature datasets using our family of tests. Section~\ref{sec:empirical.power.analysis} gathers the results of an extensive empirical power comparison of goodness-of-fit tests for the Laplace distributions. The conclusion follows in Section~\ref{sec:conclusion}. All the detailed programming codes using the \texttt{R} software and all the proofs are given in the Supplementary Material.

\section{Background and preliminaries}\label{sec:preliminaries}

\subsection{Family of asymmetric power distributions}\label{sec:APD}

The asymmetric power distribution (APD) introduced by \cite{Komunjer2007} is a generalization of the exponential power distribution (EPD) \cite[p.~271]{Kotz_2001}, the latter also being known as the generalized error distribution or the generalized normal distribution \citep{MR2119411}. The APD family is broader in that it combines the wide range of exponential tail behaviours provided by the symmetric EPD with different levels of asymmetry. Below we propose a modified version of the original APD density function $f(u)$ defined in \citet[Section~2]{Komunjer2007} by modifying its original scaling through the change of variable $u = \lambda^{-1/\theta_2} x$ and by adding location and scale parameters. In our context, $\lambda$ is a fixed parameter chosen by the user.

\begin{definition}\label{def:APD}
A random variable $X$ is said to be $\mathrm{APD}_{\lambda}(\theta_1,\theta_2,\mu,\sigma)$ distributed if its density function is given by
\begin{equation}\label{eq:APD.density}
    f_{\lambda}(x \nvert \theta_1,\theta_2,\mu,\sigma) \leqdef
    \frac{(\delta_{\hspace{-0.2mm}\theta_1\hspace{-0.5mm},\theta_2}/\lambda)^{1/\theta_2}}{\sigma\Gamma(1 + 1/\theta_2)}
    \exp\left(-\frac{1}{\lambda}\frac{\delta_{\hspace{-0.2mm}\theta_1\hspace{-0.5mm},\theta_2}}
    {A_{\hspace{-0.1mm}\theta_1\hspace{-0.5mm},\hspace{-0.2mm}\theta_2}\hspace{-0.6mm}(y)}|y|^{\theta_2}\right),\quad x\in\R,
\end{equation}
where $y\leqdef \sigma^{-1}(x-\mu)\in \R$, $\theta_1\in(0,1)$ is the asymmetry parameter, $\theta_2\in(0,\infty)$ is the tail decay parameter, $\mu\in\R$ is the location parameter, $\sigma>0$ is the scale parameter, $\lambda\in(0,\infty)$,
\begin{equation}
    \delta_{\hspace{-0.2mm}\theta_1\hspace{-0.5mm},\theta_2} \leqdef \frac{2 \theta_1^{\theta_2}
    (1 - \theta_1)^{\theta_2}}{\theta_1^{\theta_2} + (1 - \theta_1)^{\theta_2}}\in(0,1)\quad \text{and}
    \quad A_{\hspace{-0.1mm}\theta_1\hspace{-0.5mm},\hspace{-0.2mm}\theta_2}\hspace{-0.6mm}(y) \leqdef \big[1/2 + \mathrm{sign}(y) (1/2 - \theta_1)\big]^{\theta_2}=
    \begin{cases}
        \theta_1^{\theta_2},& \mbox{if } y<0, \\
        (1-\theta_1)^{\theta_2},& \mbox{if } y>0.
    \end{cases}
\end{equation}
\end{definition}

By integrating \eqref{eq:APD.density}, one can easily verify that the cumulative distribution function (c.d.f.) is given by
\begin{equation}\label{eq:cdf.APD}
    F_{\lambda}(x \nvert \theta_1,\theta_2,\mu,\sigma) \leqdef \theta_1
    \left[1 - F_W\bigg(\frac{\delta_{\hspace{-0.2mm}\theta_1\hspace{-0.5mm},\theta_2}}{\lambda}
    \Big(\frac{\max(-y,0)}{\theta_1}\Big)^{\theta_2}\bigg)\right] + (1 - \theta_1) \,
    F_W\bigg(\frac{\delta_{\hspace{-0.2mm}\theta_1\hspace{-0.5mm},\theta_2}}{\lambda}
    \Big(\frac{\max(y,0)}{1 - \theta_1}\Big)^{\theta_2}\bigg),
\end{equation}
which makes it possible to generate observations $X\sim \mathrm{APD}_{\lambda}(\theta_1,\theta_2,\mu,\sigma)$ (see the proof in the Supplementary Material, Section~B.1) by taking
\begin{equation}\label{eq:how.to.simul.APD}
    X = \mu + \sigma \big(\delta_{\hspace{-0.2mm}\theta_1\hspace{-0.5mm},\theta_2}^{-1}\lambda W\big)^{1/\theta_2}
    \big((1 - \theta_1)(1-V) - \theta_1  V\big),
\end{equation}
where $W\sim \text{Gamma}\hspace{0.2mm}(1/\theta_2,1)$ and $V\sim\text{Bernoulli}\hspace{0.2mm}(\theta_1)$ in \eqref{eq:cdf.APD} and \eqref{eq:how.to.simul.APD} are independent.
\begin{figure}[H]
\begin{center}
\begin{tabular}{c}
    \includegraphics[width=7cm]{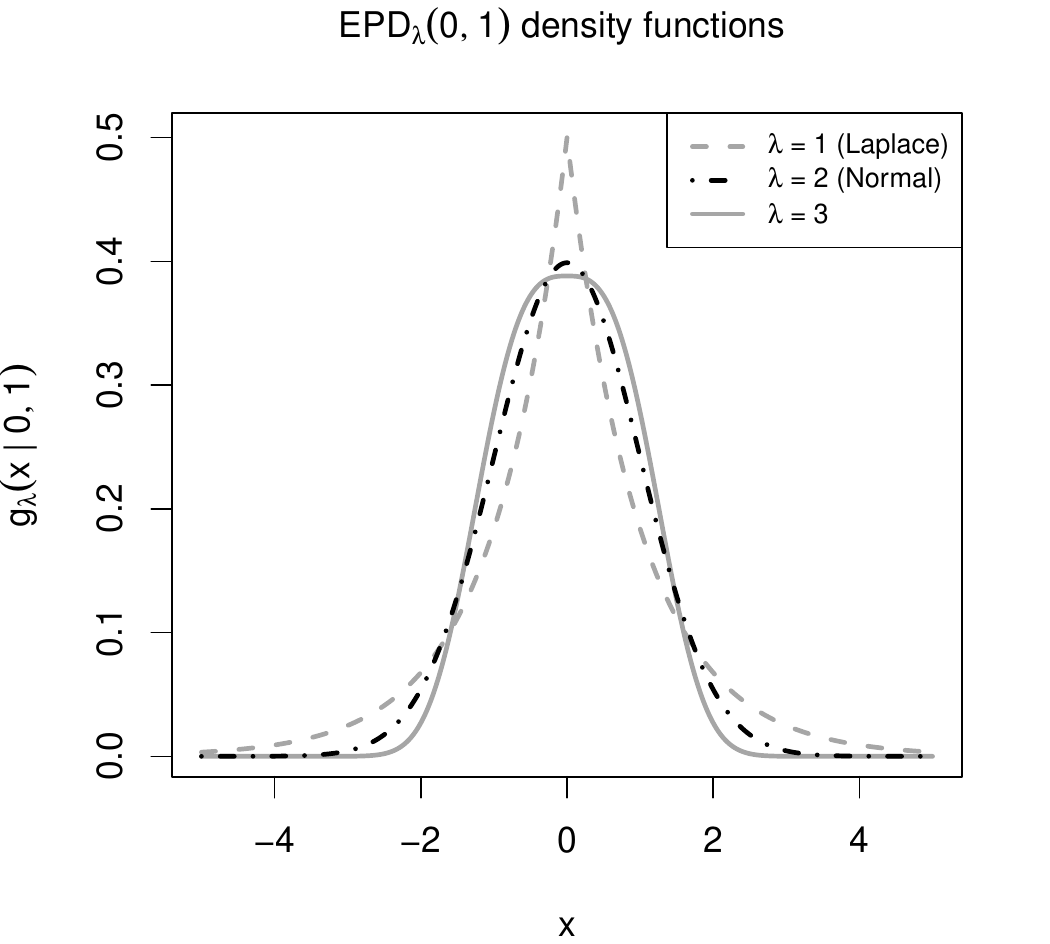}
\end{tabular}
\end{center}
\captionsetup{width=0.8\linewidth}
\vspace{-6mm}
\caption{The $\mathrm{EPD}_{\lambda}(\mu=0,\sigma=1)$ density functions for values of $\lambda\in\{1, 2, 3\}$.}\label{fig:EPD}
\end{figure}

\subsection{The null and alternative hypotheses}\label{sec:hypotheses}

We now consider particular cases of the $\mathrm{APD}_{\lambda}(\theta_1,\theta_2,\mu,\sigma)$ which will form the null and alternative hypotheses of the tests. First, we define the distribution to be tested in the null hypothesis.
\begin{definition}\label{def:EPD}
A random variable $X$ is said to be $\mathrm{EPD}_{\lambda}(\mu,\sigma)\leqdef \mathrm{APD}_{\lambda}(1/2,\lambda,\mu,\sigma)$ distributed if its density function is given by
\begin{equation}\label{eq:EPD.density}
    g_{\lambda}(x \nvert \mu,\sigma)\leqdef f_{\lambda}(x \nvert 1/2,\lambda,\mu,\sigma)= \frac{1}{2\sigma  \lambda^{1/\lambda} \Gamma(1 + 1/\lambda)}
    \exp\left(-\frac{1}{\lambda} |y|^{\lambda}\right),\quad x\in\R,
\end{equation}
where $y\leqdef \sigma^{-1}(x-\mu)\in \R$, $\lambda\in(0,\infty)$ is the tail decay parameter, $\mu\in\R$ is the location parameter and $\sigma>0$ is the scale parameter.
\end{definition}

In particular, we obtain the $\mathrm{EPD}_{1}(\mu,\sigma)=\text{Laplace}(\mu,\sigma)$ distribution if $\lambda=1$ and the $\mathrm{EPD}_{2}(\mu,\sigma)=\mathcal{N}(\mu,\sigma^2)$ distribution if $\lambda=2$, as illustrated in Figure~\ref{fig:EPD}.

We want to construct goodness-of-fit tests for specific $\mathrm{EPD}_{\lambda}(\mu,\sigma)$, with a fixed value of $\lambda\ge 1$. In particular, the values of $\lambda=1$ and $\lambda=2$ lead to tests for Laplace and Gaussian distributions. If the random sample is denoted by
\begin{equation}\label{not:sample}
    \Xn \leqdef X_1,X_2,\ldots,X_n,
\end{equation}
then, for a fixed value of $\lambda\ge 1$, the composite null hypothesis for all our tests is given by
\begin{equation}\label{eq:H0}
    H_0 : \Xn \sim\mathrm{EPD}_{\lambda}(\mu,\sigma), \text{ where } \mu \text{ and } \sigma \text{ are unknown.}
\end{equation}

The alternative hypothesis will depend on the type of test. Consider an omnibus test for the $\mathrm{EPD}_{\lambda}(\mu,\sigma)$ designed to detect all types of alternatives, whether asymmetrical or symmetrical with heavy or light tails. The alternative hypothesis is then given by
\begin{equation}\label{eq:H1.omnibus}
    H_1 : \Xn \not\sim \mathrm{EPD}_{\lambda}(\mu,\sigma),\text{ where } \mu \text{ and } \sigma \text{ are unknown.}
\end{equation}
However, our strategy is to approximate as many existing distributions as possible by the large APD family. Therefore, we use this alternative hypothesis instead:
\begin{equation}\label{eq:H1.quasi.omnibus}
    H_1 : \Xn\sim\mathrm{APD}_{\lambda}(\theta_1,\theta_2,\mu,\sigma),\text{ where }(\theta_1,\theta_2)\neq (1/2,\lambda), \text{ and }
    \mu \text{ and } \sigma \text{ are unknown.}
\end{equation}

When specific information about the distribution of the random variable is known or assumed, we can take advantage of it to construct a more powerful directional test by narrowing the universe of the alternative hypothesis. Consider a directional test for the $\mathrm{EPD}_{\lambda}(\mu,\sigma)$ designed to detect asymmetric alternatives. For a fixed value of $\lambda\ge 1$, the alternative hypothesis is then given by
\begin{equation}\label{eq:H1.asym}
    H_1 : \Xn \sim \mathrm{APD}_{\lambda}(\theta_1,\lambda,\mu,\sigma),\text{ where }\theta_1\neq 1/2, \text{ and } \mu \text{ and } \sigma \text{ are unknown.}
\end{equation}
Consider finally a directional test for the $\mathrm{EPD}_{\lambda}(\mu,\sigma)$ designed to detect symmetric alternatives. For a fixed value of $\lambda\ge 1$, the alternative hypothesis is then given by
\begin{equation}\label{eq:H1.sym}
    H_1 : \Xn \sim\mathrm{APD}_{\lambda}(1/2,\theta_2,\mu,\sigma),\text{ where }\theta_2\neq \lambda, \text{ and } \mu \text{ and } \sigma \text{ are unknown,}
\end{equation}
and the density function of the $\mathrm{APD}_{\lambda}(1/2,\theta_2,\mu,\sigma)$ is given by
\begin{equation}\label{eq:APD.sym.density}
    f_{\lambda}(x \nvert 1/2,\theta_2,\mu,\sigma)= \frac{1}{2\sigma  \lambda^{1/\theta_2} \Gamma(1 + 1/\theta_2)} \exp\left(-\frac{1}{\lambda}
    |y|^{\theta_2}\right), \quad y\leqdef \sigma^{-1}(x-\mu)\in \R,
\end{equation}
since $\delta_{\hspace{-0.1mm}1/2\hspace{-0.3mm},\hspace{-0.2mm}\theta_2} =A_{\hspace{-0.1mm}1/2\hspace{-0.3mm},\hspace{-0.2mm}\theta_2}\hspace{-0.6mm}(y)=2^{-\theta_2}$ (see Definition~\ref{def:APD}).

\subsection{Other definitions and notation}\label{sec:other}

\begin{proposition}\label{prop:MLE.estimators}
For an i.i.d.\ sample $\Xn$ and for a fixed value of $\lambda\ge 1$, the maximum likelihood estimators of $\mu$ and $\sigma$ for the $\mathrm{EPD}_{\lambda}(\mu,\sigma)$ (see Definition~\ref{def:EPD}) are given by
\begin{equation}
    \hat{\sigma}_{\lambda} = \bigg(\frac{1}{n} \sum_{i=1}^n |X_i - \hat{\mu}_{\lambda}|^{\lambda}\bigg)^{1/\lambda}
\end{equation}
and
\begin{equation}
    \hat{\mu}_{\lambda} =
        \begin{cases}
                \mathrm{median}(\Xn), ~&\mbox{if } \lambda = 1, \\[2mm]
                \bar{X}\leqdef \frac{1}{n} \sum_{i=1}^n X_i, ~&\mbox{if } \lambda = 2,\\[2mm]
                \text{the unique numerical solution to } & \\[0mm]
                \sum_{i=1}^n |X_i - \hat{\mu}_{\lambda}|^{\lambda - 1} \mathrm{sign}(X_i - \hat{\mu}_{\lambda}) = 0, ~&\mbox{if } \lambda > 1.
        \end{cases}
\end{equation}
\end{proposition}

\begin{remark}\label{rem:median}
Following the usual convention, we define the median of an ordered sample as the central value if $n$ is odd and as the arithmetic mean of the two central values if $n$ is even. Also, when $\lambda \not\in \{1,2\}$, $\hat{\mu}_{\lambda}$ has no explicit expression, however, the numerical calculation is simple.
\end{remark}
\begin{remark}\label{rem:strong.consistency}
We prove in the Supplementary Material Section~B.2 that $\hat{\mu}_{\lambda}$ and $\hat{\sigma}_{\lambda}$ are strongly consistent, both under $H_0$ and $H_1$. 
\end{remark}

\begin{definition}\label{def:S.K.lambda}
For an i.i.d.\ sample $\Xn$ and for a fixed value of $\lambda\ge 1$, the `$\lambda$-th-power skewness' and the `$\lambda$-th-power kurtosis' are, respectively, given by
\begin{equation}\label{eq:S.K.lambda}
    S_{\lambda}(\Xn) \leqdef \frac{1}{n} \sum_{i=1}^n |Y_i|^{\lambda} \mathrm{sign}(Y_i) \quad \text{and} \quad
    K_{\lambda}(\Xn) \leqdef \frac{1}{n} \sum_{i=1}^n  |Y_i|^{\lambda} \log |Y_i|,
\end{equation}
where $Y_i \leqdef \hat{\sigma}_{\lambda}^{-1} (X_i - \hat{\mu}_{\lambda})$, $\hat{\mu}_{\lambda}$ and $\hat{\sigma}_{\lambda}$ are the maximum likelihood estimators of $\mu$ and $\sigma$ for the $\mathrm{EPD}_{\lambda}(\mu,\sigma)$ as given in Proposition~\ref{prop:MLE.estimators}, and we define $(|y|^{\lambda} \log|y|)|_{y=0}\leqdef 0$.
\end{definition}
\begin{remark}\label{prop:Kgeq0}
    For any sample $\Xn$, note that $K_{\lambda}(\Xn)\geq 0$ as a direct consequence of Jensen's inequality and the convexity of $\tau(x) \leqdef x \log x$:
    \begin{equation}
        \lambda K_{\lambda}(\Xn) = \frac{1}{n} \sum_{i=1}^n  |Y_i|^{\lambda} \log |Y_i|^{\lambda}=\sum_{i=1}^n \frac{1}{n}\tau\big(|Y_i|^{\lambda}\big) \geq \tau\Bigg(\frac{1}{n} \sum_{i=1}^n |Y_i|^{\lambda}\Bigg) = \tau(1) = 0.
    \end{equation}
\end{remark}
As their names suggest, the $\lambda$-th-power skewness and kurtosis are new measures of asymmetry and tail thickness for the data distribution. Our test statistics will be based on these two quantities.  A positive (resp.\ negative) value of $S_{\lambda}(\Xn)$ suggests that the distribution is right-skewed (resp.\ left-skewed), while a value close to 0 suggests that the distribution is symmetric. A large (resp.\ small) value of $K_{\lambda}(\Xn)$ corresponds to a heavy-tailed (resp.\ light-tailed) distribution.

\begin{definition}\label{def:net.S.K.lambda}
For an i.i.d.\ sample $\Xn$ and for a fixed value of $\lambda\ge 1$, the `$\lambda$-th-power net kurtosis' is given by
\begin{equation}\label{eq:knet}
     K^{\mathrm{net}}_{\lambda}(\Xn)\leqdef \max\big(0,K_{\lambda}(\Xn)-(\lambda/2)S^2_{\lambda}(\Xn)\big),
\end{equation}
where $S_{\lambda}(\Xn)$ and $K_{\lambda}(\Xn)$ are given in Definition~\ref{def:S.K.lambda}.
\end{definition}
The $\lambda$-th-power net kurtosis will be used instead of the $\lambda$-th-power kurtosis in our test statistics to correct for dependence with the $\lambda$-th-power skewness for small to moderate sample sizes. A large (resp.\ small) value of $K^{\mathrm{net}}_{\lambda}(\Xn)\ge 0$ corresponds to a heavy-tailed (resp.\ light-tailed) distribution, given the level of skewness. Note that in practice the only cases where we found the maximum function needed to bound a (slightly) negative value of $K_{\lambda}(\Xn)-(\lambda/2)S^2_{\lambda}(\Xn)$ in \eqref{eq:knet} were for nearly symmetric samples with extremely light tails ($K_{\lambda}(\Xn)$ and $S^2_{\lambda}(\Xn)$ close to 0), and thus far from the null hypothesis.

\begin{remark}
Analogously, for a random variable $X$ with an arbitrary distribution function $F$, the $\lambda$-th-power skewness, $\lambda$-th-power kurtosis and $\lambda$-th-power net kurtosis can be defined, respectively, as
\begin{equation*}
  S_{\lambda}(F)\leqdef\EE(|Y|^{\lambda} \mathrm{sign}(Y)),\quad K_{\lambda}(F)\leqdef \EE(|Y|^{\lambda} \log|Y|),
\end{equation*}
and
\begin{equation*}
      K^{\mathrm{net}}_{\lambda}(F)\leqdef \max\big(0,K_{\lambda}(F)-(\lambda/2)S^2_{\lambda}(F)\big),
\end{equation*}
where $Y=\sigma_{\lambda}^{-1}(X-\mu_{\lambda})$, with $\sigma_{\lambda}=(\EE|X-\mu_{\lambda}|^{\lambda})^{1/\lambda}$, $\mu_{1}=\mathrm{median}(X)$, $\mu_2=\EE(X)$ and in general for $\lambda>1$, $\mu_{\lambda}$ is the numerical solution to $\EE(|X - \mu_{\lambda}|^{\lambda - 1} \mathrm{sign}(X - \mu_{\lambda})) = 0$, which is unique for most known distributions.  These measures, which represent the asymptotic version of their sample counterparts, can be useful in comparing the asymmetry and tail thickness of a distribution relative to the $\mathrm{EPD}_{\lambda}(\mu,\sigma)$ tested in~$H_0$. For symmetric distributions $F$, we have $S_1(F)=0$ and $K_1(F)=K^{\mathrm{net}}_{1}(F)$. For example, $K_1(\mathrm{U[a,b]})=0.193$, $K_1(\mathcal{N}(\mu,\sigma^2))=0.284$, $K_1(t_{10})=0.314$, $K_1(t_4)=0.386$, $K_1(\mathrm{Laplace}(\mu,\sigma))=0.423$, $K_1(t_3)=0.452$ and $K_1(t_2)=0.693$. 

\end{remark}

\vspace{3mm}
\noindent\textbf{Notation.} For $z>0$, the gamma, digamma and trigamma functions are denoted respectively by
\begin{equation}\label{eq:gamma.functions}
    \Gamma(z) \leqdef \int_0^{\infty} t^{z-1} e^{-t} \rd t, \quad \psi(z) \leqdef \frac{\rd}{\rd z} \log \Gamma(z)\quad
    \text{ and } \quad \psi_1(z)\leqdef \frac{\rd}{\rd z}\psi(z).
\end{equation}
The chi-square distribution with 2 degrees of freedom is denoted by $\chi_2^2$ and its counterpart with the non-centrality parameter $\xi>0$ by $\chi_2^2(\xi)$. Convergence in distribution, in probability, and almost-surely, are denoted respectively by $\stackrel{\mathcal{D}}{\longrightarrow}$, $\stackrel{\mathcal{P}}{\longrightarrow}$, and $\stackrel{\mathrm{a.s.}}{\longrightarrow}$, as $n\to\infty$. Furthermore, we note `\,$\stackrel{\mathrm{app}}{\sim}$' for `approximately distributed as'.

\section{Goodness-of-fit tests for the \texorpdfstring{$\bb{\mathrm{EPD}_{\lambda}(\mu,\sigma)}$}{EPD lambda (mu, sigma)}}\label{sec:tests.asymp}

In Section~\ref{sec:asymp.dist.H0}, we present our family of goodness-of-fit tests for the EPD and their asymptotic distributions under the null hypothesis. The asymptotic distributions under local alternatives are given in Section~\ref{sec:asymp.dist.H1}. Finally, the Laplace and Gaussian goodness-of-fit tests are presented in Section~\ref{sec:Laplace.Gaussian}.

\subsection{The test statistics and their distributions under the null hypothesis}\label{sec:asymp.dist.H0}

We first use Rao's score test on the APD family specified in $H_1$. For the omnibus test, this consists in taking the $2$-dimensional asymptotic test statistic
\begin{equation}\label{eq:score.rao}
    \bigg(\frac{1}{n}\sum_{i=1}^n \frac{\partial}{\partial \theta_i} \log f_{\lambda}
    (X_i \nvert \theta_1,\theta_2,\mu,\sigma)\big|_{(\theta_1,\theta_2)= (1/2,\lambda)}\bigg)_{i\in\{1,2\}},
\end{equation}
and replacing $\mu$ and $\sigma$ with their maximum likelihood estimators $\hat{\mu}_{\lambda}$ and $\hat{\sigma}_{\lambda}$ under $H_0$. Only the first (resp.\ second) component of the vector in \eqref{eq:score.rao} is needed for the test against asymmetric (resp.\ symmetric) alternatives. The resulting test statistics are given in Definition~\ref{def:asymp.stats}. Their construction is detailed in the Supplementary Material B. These tests are only valid for very large sample sizes, but will be adapted below for all sample sizes. The case $\lambda < 1$ is not studied in this paper because the numerical solution to the equation $\sum_{i=1}^n |X_i - \hat{\mu}_{\lambda}|^{\lambda - 1} \mathrm{sign}(X_i - \hat{\mu}_{\lambda}) = 0$ in Proposition~\ref{prop:MLE.estimators} is not unique, which makes the calculation of the maximum likelihood estimate of $\mu$ unstable due to multiple modes. Moreover, some proofs break down for $\lambda < 1$ (such as the proof of Proposition~B.5, and thus the proof of Theorem~\ref{thm:asymp.conv.H0}) or would have to be significantly extended. Therefore, dealing with the case $\lambda < 1$ is not obvious and would require extensive additional research.

\begin{definition}\label{def:asymp.stats}
Consider the tests of fit for the $\mathrm{EPD}_{\lambda}(\mu,\sigma)$ with a fixed value of $\lambda\ge 1$. The statistic for the asymptotic directional test of fit against asymmetric alternatives is given by
\begin{equation}\label{eq:ZS.asymp}
    Z^{*\!}(S_{\lambda})\leqdef\sqrt{n}\bigg(1 + \lambda - \frac{\lambda^2}{\Gamma(2-1/\lambda)\Gamma(1/\lambda)}\bigg)^{-1/2} S_{\lambda}(\Xn),
\end{equation}
while the statistic for the asymptotic directional test of fit against symmetric alternatives is given by
\begin{equation}\label{eq:ZK.asymp}
    Z^{*\!}(K_{\lambda})\leqdef\sqrt{n} \bigg(\frac{(1+1/\lambda) \psi_1(1+1/\lambda)-1}{\lambda}\bigg)^{-1/2}\bigg(K_{\lambda}(\Xn)-\frac{\lambda+\log \lambda + \psi(1/\lambda)}{\lambda}\bigg),
\end{equation}
where $S_{\lambda}(\Xn)$ and $K_{\lambda}(\Xn)$ are respectively the sample $\lambda$-th-power skewness and $\lambda$-th-power kurtosis given in Definition~\ref{def:S.K.lambda}, and $\Gamma(z),\psi(z),\psi_1(z)$ are defined in \eqref{eq:gamma.functions}.
The statistic for the asymptotic omnibus test is given by
\begin{equation}
    \big(Z^{*\!}(S_{\lambda})\big)^2+\big(Z^{*\!}(K_{\lambda})\big)^2.
\end{equation}
\end{definition}

\begin{remark}
    The superscript $*$ symbol in Definition~\ref{def:asymp.stats} indicates the asymptotic nature of the test statistics.
\end{remark}

\begin{theorem}\label{thm:asymp.conv.H0}
Under the null hypothesis, we have, as $n\to \infty$,
\begin{equation}\label{eq:asymp.conv.Zstar}
    \begin{pmatrix}
        Z^{*\!}(S_{\lambda})\\[1mm]
        Z^{*\!}(K_{\lambda})
    \end{pmatrix}
    \stackrel{\mathcal{D}}{\longrightarrow} \mathcal{N}_2
    \left(
    \begin{pmatrix}
        0 \\[1mm]
        0
    \end{pmatrix},
    \begin{pmatrix}
        1 & 0 \\[1mm]
        0 & 1
    \end{pmatrix}
    \right)
\quad \text{ and } \quad \big(Z^{*\!}(S_{\lambda})\big)^2+\big(Z^{*\!}(K_{\lambda})\big)^2\stackrel{\mathcal{D}}{\longrightarrow} \chi_2^2.
\end{equation}
\end{theorem}


\begin{remark}
Theorem~\ref{thm:asymp.conv.H0} contains our main theoretical result. The proof is presented in the Supplementary Material Section~B.3. When $\mu$ and $\sigma$ are known and fixed, this result follows by a standard application of the central limit theorem. Here the proof is rendered more difficult by the fact that $\mu$ and $\sigma$ are replaced by their maximum likelihood estimators, $\hat{\mu}_{\lambda}$ and $\hat{\sigma}_{\lambda}$, which requires a control of the first-order derivatives of the score function when we expand the score function around $\mu$ and $\sigma$ under $H_0$. Nevertheless, controlling the derivatives is straightforward using ideas that go back to Lucien Le Cam on the uniform law of large numbers, see, e.g., Chapter~16 in \cite{MR1699953}. However, there is one exception, namely $\lambda = 1$. In that case, one of the components of the first-order derivative matrix has misbehaving logarithmic summands. In its essential form, for $\mu = 0$, the problem reduces to showing that
\begin{equation}\label{eq:remark.3.4}
    \bigg|\int_0^1 \frac{1}{n} \sum_{i=1}^n \log|X_i -  v \hat{\mu}_1| \rd v - \frac{1}{n} \sum_{i=1}^n \log|X_i| \bigg| \stackrel{\mathcal{P}}{\longrightarrow} 0, \quad \text{as } n\to \infty.
\end{equation}
Notice that $\log|\cdot|$ blows up when $X_i - v \hat{\mu}_1$ is close to zero for $v\in [0,1]$.
In more technical terms, the envelope of the class of functions $\{x\mapsto \log|x - t \, |\}_{t:|t| < \delta}$ is infinite in any small neighbourhood of $x = 0$, which means that standard uniform laws of large numbers cannot be applied, see, e.g., \cite[Section~2.4]{MR1385671}. The saving grace is that while $\log|\cdot|$ blows up at $0$, it is still integrable locally at $0$, and $\log|\cdot|$ is mild enough at infinity to be dominated by the exponential tail of the Laplace distribution (i.e., the $\mathrm{EPD}_1$ distribution). Under such conditions, it is to be expected that some form of uniform law of large numbers still holds, so that \eqref{eq:remark.3.4} is true (numerical simulations also confirm that intuition) and the aforementioned control on the first-order derivatives remains valid.
In \cite{MR3842623}, a new uniform $L^1$ law of large numbers for summands that blow up was developed for this specific purpose.
Using the main result in that paper, it can be shown that
\begin{equation}
    \sup_{v\in [0,1]} \EE\bigg|\frac{1}{n} \sum_{i=1}^n \log|X_i - v \hat{\mu}_1| -  \EE\big[\log|X_1|\big]\bigg| \longrightarrow 0, \quad \text{as } n\to \infty,
\end{equation}
so that the convergence in \eqref{eq:remark.3.4} holds not only in probability, but in fact in $L^1$.
\end{remark}

Theorem~\ref{thm:asymp.conv.H0} gives us two independent normally distributed $Z$-scores representing asymptotic test statistics designed to detect asymmetric and symmetric alternatives, which can be combined into an omnibus test statistic converging to a $\chi_2^2$ distribution. As is often the case for small to moderate sample sizes, we observed that the distributions of Rao's score test statistics $Z^{*\!}(S_{\lambda})$ and $Z^{*\!}(K_{\lambda})$ are not well approximated, under the null hypothesis, by their asymptotic $\mathcal{N}(0,1)$ distribution. Moreover, the independence of these two measures does not hold. To solve this problem, we generalize the approach adopted in \cite{Desgagne_Lafaye_2018}. We address here the specific cases $\lambda\in\{1,1.5,2,2.5,3\}$, but the approach would be the same for all $\lambda\ge 1$. To increase precision, we optimized our approximations for $n\geq 20$, but they remain valid for smaller sample sizes.

The most important issue to address is the dependence between $S_{\lambda}(\Xn)$ and $K_{\lambda}(\Xn)$, which increases as the sample size $n$ decreases and the value of $\lambda$ increases (meaning lighter tails). A skewed distribution with a large value of $|S_{\lambda}(\Xn)|$ is often also heavy-tailed with a large value of $K_{\lambda}(\Xn)$. This situation, at first sight complex, can fortunately be solved by using $K^{\mathrm{net}}_{\lambda}(\Xn)$, the $\lambda$-th-power net kurtosis introduced in Definition~\ref{def:net.S.K.lambda}. We find numerically that $S_{\lambda}(\Xn)$ and $(K^{\mathrm{net}}_{\lambda}(\Xn))^{1/4}$ are closely distributed as a normal and that the dependency between them is negligible, for all $n\geq 20$. The power of $1/4$ is inspired by an improvement of the Wilson-Hilferty cube root transformation that leads to approximate normality (see \cite{doi:10.2307/2684608}). Since we designed the $\lambda$-th-power net kurtosis to replace the $\lambda$-th-power kurtosis, the next step is to define an asymptotic $Z$-score denoted $Z^{*\!}(K^{\mathrm{net}}_{\lambda})$ to replace $Z^{*\!}(K_{\lambda})$. We found that the two $Z$-scores are asymptotically equivalent.

\begin{definition}
Consider the test of fit for the $\mathrm{EPD}_{\lambda}(\mu,\sigma)$ with a fixed value of $\lambda\ge 1$. A statistic equivalent to $Z^{*\!}(K_{\lambda})$ for the asymptotic directional test of fit against symmetric alternatives is given by
\footnotesize
\begin{equation}
    Z^{*\!}(K^{\mathrm{net}}_{\lambda})\!\leqdef\! \sqrt{n} \bigg[\frac{1}{16}\left(\frac{\lambda+\log
    \lambda + \psi(1/\lambda)}{\lambda}\right)^{\!\!-3/2} \frac{(1+1/\lambda) \psi_1(1+1/\lambda)-1}{\lambda}\bigg]^{-1/2}\!\bigg(\big(K^{\mathrm{net}}_{\lambda}(\Xn)\big)^{1/4} - \left(\frac{\lambda+\log \lambda +
    \psi(1/\lambda)}{\lambda}\right)^{1/4}\bigg),
\end{equation}
\normalsize
where $K^{\mathrm{net}}_{\lambda}(\Xn)$ is the $\lambda$-th-power net kurtosis given in Definition~\ref{def:net.S.K.lambda}.
\end{definition}

\begin{proposition}\label{prop:asymp.conv.H0}
Under the null hypothesis, we have, as $n\to \infty$,
\begin{equation}
   Z^{*\!}(K^{\mathrm{net}}_{\lambda})/Z^{*\!}(K_{\lambda})\stackrel{\mathcal{P}}{\longrightarrow}1,
\end{equation}
\begin{equation}\label{eq:asymp.conv.Zstar2}
    \begin{pmatrix}
        Z^{*\!}(S_{\lambda})\\[1mm]
        Z^{*\!}(K^{\mathrm{net}}_{\lambda})
    \end{pmatrix}
    \stackrel{\mathcal{D}}{\longrightarrow} \mathcal{N}_2
    \left(
    \begin{pmatrix}
        0 \\[1mm]
        0
    \end{pmatrix},
    \begin{pmatrix}
        1 & 0 \\[1mm]
        0 & 1
    \end{pmatrix}
    \right)
    \quad \text{ and } \quad \big(Z^{*\!}(S_{\lambda})\big)^2+\big(Z^{*\!}(K^{\mathrm{net}}_{\lambda})\big)^2\stackrel{\mathcal{D}}{\longrightarrow} \chi_2^2.
\end{equation}
\end{proposition}
The proof of Proposition~\ref{prop:asymp.conv.H0} is given in the Supplementary Material Section~B.5. Since the $Z$-scores  $Z^{*\!}(S_{\lambda})$ and $Z^{*\!}(K^{\mathrm{net}}_{\lambda})$ are affine transformations of $S_{\lambda}(\Xn)$ and $(K^{\mathrm{net}}_{\lambda}(\Xn))^{1/4}$, they are also very close to being independent and normally distributed for all sample sizes. The final step is to adapt $Z^{*\!}(S_{\lambda})$ and $Z^{*\!}(K^{\mathrm{net}}_{\lambda})$ to ensure that their expectations and variances are very close to 0 and 1 respectively, for all sample sizes $n\ge 20$, always under the null hypothesis. We can show that $\EE(S_{\lambda}(\Xn))=0$ for all sample sizes under the null hypothesis, but we have to use simulation and regression techniques to approximate $\VV(n^{1/2}S_{\lambda}(\Xn))$, $\EE((K^{\mathrm{net}}_{\lambda}(\Xn))^{1/4})$ and $\VV(n^{1/2}(K^{\mathrm{net}}_{\lambda}(\Xn))^{1/4})$. Specifically, for $n=20,21,\dots,200$, we fit the three following linear regression models with no intercept term:

\begin{equation}
    \frac{\VV(n^{1/2}S_{\lambda}(\Xn))}{1 + \lambda - \frac{\lambda ^2}{\Gamma(2-1/\lambda)\Gamma(1/\lambda)}}-1= \cone\cdot n^{-\hspace{-0.3mm}\aone} +\varepsilon_n,
    \quad\quad\quad\quad\frac{\EE\big((K^{\mathrm{net}}_{\lambda}(\Xn))^{1/4}\big)}{\left(\frac{\lambda+\log \lambda + \psi(1/\lambda)}{\lambda}\right)^{1/4}}-1=
    \ctwo\cdot n^{-\hspace{-0.3mm}\atwo}+\varepsilon_n,
\end{equation}

\begin{equation}
    \text{and }\quad\quad\frac{\VV\big(n^{1/2}(K^{\mathrm{net}}_{\lambda}(\Xn))^{1/4}\big)}{\frac{1}{16}\left(\frac{\lambda+\log \lambda + \psi(1/\lambda)}{\lambda}\right)^{-3/2}
    \frac{(1+1/\lambda) \psi_1(1+1/\lambda)-1}{\lambda}}-1= \cthree\cdot
    n^{-\hspace{-0.3mm}\athree}+\cfour\cdot
    n^{-\hspace{-0.3mm}\afour}+\varepsilon_n,
\end{equation}
where $\VV(n^{1/2}S_{\lambda}(\Xn))$, $\EE((K^{\mathrm{net}}_{\lambda}(\Xn))^{1/4})$ and $\VV(n^{1/2}(K^{\mathrm{net}}_{\lambda}(\Xn))^{1/4})$ are estimated using 1,000,000 simulations for each value of $n=20,21,\dots,200$, with $\Xn\sim \mathrm{EPD}_{\lambda}(\mu,\sigma)$. The values of $\mu$ and $\sigma$ are arbitrary since $S_{\lambda}(\Xn)$, $K_{\lambda}(\Xn)$ and $K^{\mathrm{net}}_{\lambda}(\Xn)$ are location-scale invariant. The dependent variables of the regressions are given by the ratio of the finite and asymptotic variances (or expectations), minus 1. The explanatory variables are given by $n^{-\aone},n^{-\atwo}, n^{-\athree}, n^{-\afour}$, $n=20,21,\dots,200$, and the coefficients are given by $\cone,\ctwo,\cthree,\cfour$. The values $\aone$, $\atwo$, $\athree$, $\afour$ are chosen to maximize the $R^2$ value of their respective regressions. The quality of fit of the regressions is remarkable with $R^2$ values close to 1.

We present the resulting test statistics in Definition~\ref{def:APD.tests}, with their asymptotic distributions is Proposition~\ref{prop:asymp.conv.H02} and their approximate distributions for all sample sizes in Proposition~\ref{prop:X.APD.finite.sample}.

\begin{definition}\label{def:APD.tests}
Consider the tests of fit for the $\mathrm{EPD}_{\lambda}(\mu,\sigma)$ with a fixed value of $\lambda\geq 1$. The statistic for the directional test of fit against asymmetric alternatives, for all sample sizes, is given by
\begin{equation}\label{eq:ZS}
    Z(S_{\lambda})\leqdef\frac{n^{1/2} S_{\lambda}(\Xn)}{\Big[\big(1 + \lambda - \frac{\lambda^2}{\Gamma(2-1/\lambda)\Gamma(1/\lambda)}\big)\big(1 +
    \frac{\cone}{n^{\hspace{-0.3mm}\aone}}\big)\Big]^{1/2}},
\end{equation}
while the statistic for the directional test of fit against symmetric alternatives, for all sample sizes, is given by
\begin{equation}\label{eq:net.ZK}
    Z(K^{\mathrm{net}}_{\lambda})\leqdef \frac{n^{1/2}\Big(\big(K^{\mathrm{net}}_{\lambda}(\Xn)\big)^{1/4} - \left(\frac{\lambda+\log \lambda +
    \psi(1/\lambda)}{\lambda}\right)^{1/4}\left(1 +
    \frac{\ctwo}{n^{\hspace{-0.3mm}\atwo}}\right)\Big)}
    {\Big[\frac{1}{16}\left(\frac{\lambda+\log
    \lambda + \psi(1/\lambda)}{\lambda}\right)^{-3/2} \frac{(1+1/\lambda) \psi_1(1+1/\lambda)-1}{\lambda}\left(1 +
    \frac{\cthree}{n^{\hspace{-0.3mm}\athree}}+
    \frac{\cfour}{n^{\hspace{-0.3mm}\afour}}\right)\Big]^{1/2}},
\end{equation}
where $S_{\lambda}(\Xn)$ and $K^{\mathrm{net}}_{\lambda}(\Xn)$ are respectively the sample $\lambda$-th-power skewness and $\lambda$-th-power net kurtosis given in Definitions~\ref{def:S.K.lambda}~and~\ref{def:net.S.K.lambda}, $\Gamma(z),\psi(z),\psi_1(z)$ are defined in \eqref{eq:gamma.functions} and the constants $\alpha_{j\hspace{-0.3mm},\hspace{-0.3mm}\lambda},c_{j\hspace{-0.3mm},\hspace{-0.3mm}\lambda}, j=1,2,3,4$, are provided in Table~\ref{table:cte} for $\lambda\in\{1, 1.5, 2, 2.5, 3\}$. The statistic for the omnibus test, for all sample sizes, is given by
\begin{equation}\label{eq:Xapd}
    X^{\mathrm{APD}}_{\lambda}\leqdef Z^2(S_{\lambda})+Z^2(K^{\mathrm{net}}_{\lambda}).
\end{equation}
\end{definition}

\def\arraystretch{1.2}
\begin{table}[H]
\begin{center}
\captionsetup{width=0.8\linewidth}
\caption{The constants $\alpha_{j\hspace{-0.3mm},\hspace{-0.3mm}\lambda}$ and $c_{j\hspace{-0.3mm},\hspace{-0.3mm}\lambda}$, $j=1,2,3,4$ for $\lambda\in\{1, 1.5, 2, 2.5, 3\}$ used in Definition~\ref{def:APD.tests}. \label{table:cte}}
\begin{tabular}{|l|rrrrrrrr|}
    \hline
    & $\aone$ & $\cone$ & $\atwo$ &
    $\ctwo$ & $\athree$ & $\cthree$ &
    $\afour$ & $\cfour$ \cr
    \hline
    $\lambda=1$ (even $n$) & 1.06 & -1.856 & 1.01 & -0.422 & 0.92 & -1.950  & 2.3 & 39.349  \cr
    $\lambda=1$ (odd $n$)  & 1.03 & -0.281 & 0.86 & -0.198 & 1.04 & -3.827  & 1.0 & 0.000       \cr
    $\lambda=1.5$            & 0.99 & -0.952 & 0.99 & -0.637 & 0.55 & -3.488  & 0.5 & 2.434   \cr
    $\lambda=2$            & 0.99 & -1.890 & 1.00 & -0.788 & 1.05 & -9.327  & 1.4 & 14.208  \cr
    $\lambda=2.5$            & 0.99 & -2.981 & 0.99 & -0.844 & 1.10 & -23.104 & 1.3 & 30.028  \cr
    $\lambda=3$            & 0.97 & -3.855 & 0.98 & -0.880 & 1.14 & -95.743 & 1.2 & 103.871 \cr
    \hline
\end{tabular}
\end{center}
\end{table}

\begin{proposition}\label{prop:asymp.conv.H02}
Under the null hypothesis, we have, as $n\to \infty$,
\begin{equation}
    Z(S_{\lambda})/Z^{*\!}(S_{\lambda})\stackrel{\mathrm{a.s.}}{\longrightarrow}1,\quad  \quad
    Z(K^{\mathrm{net}}_{\lambda})/Z^{*\!}(K^{\mathrm{net}}_{\lambda})\stackrel{\mathcal{P}}{\longrightarrow}1,
\end{equation}
\begin{equation}\label{eq:asymp.conv.Zstar22}
    \begin{pmatrix}
        Z(S_{\lambda})\\[1mm]
        Z(K^{\mathrm{net}}_{\lambda})
    \end{pmatrix}
    \stackrel{\mathcal{D}}{\longrightarrow} \mathcal{N}_2
    \left(
    \begin{pmatrix}
        0 \\[1mm]
        0
    \end{pmatrix},
    \begin{pmatrix}
        1 & 0 \\[1mm]
        0 & 1
    \end{pmatrix}
    \right)
    \quad \text{ and } \quad X^{\mathrm{APD}}_{\lambda}\stackrel{\mathcal{D}}{\longrightarrow} \chi_2^2.
\end{equation}
\end{proposition}
The proof of Proposition~\ref{prop:asymp.conv.H02} is given in the Supplementary Material Section~B.6.

\begin{proposition}\label{prop:X.APD.finite.sample}
Under the null hypothesis, we have, for all sample sizes (up to very high precision for $n\geq 20$),
\begin{equation}\label{eq:dist.Z}
    \begin{pmatrix}
        Z(S_{\lambda})\\[1mm]
        Z(K^{\mathrm{net}}_{\lambda})
    \end{pmatrix}
    \stackrel{\mathrm{app}}{\sim} \mathcal{N}_2
    \left(
    \begin{pmatrix}
        0 \\[1mm]
        0
    \end{pmatrix},
    \begin{pmatrix}
        1 & 0 \\[1mm]
        0 & 1
    \end{pmatrix}
    \right)
    \quad \text{ and } \quad X^{\mathrm{APD}}_{\lambda}\stackrel{\mathrm{app}}{\sim} \chi_2^2.
\end{equation}
\end{proposition}

\begin{figure}[ht]
\begin{center}
\begin{tabular}{cc}
    \includegraphics[width=0.3\textwidth]{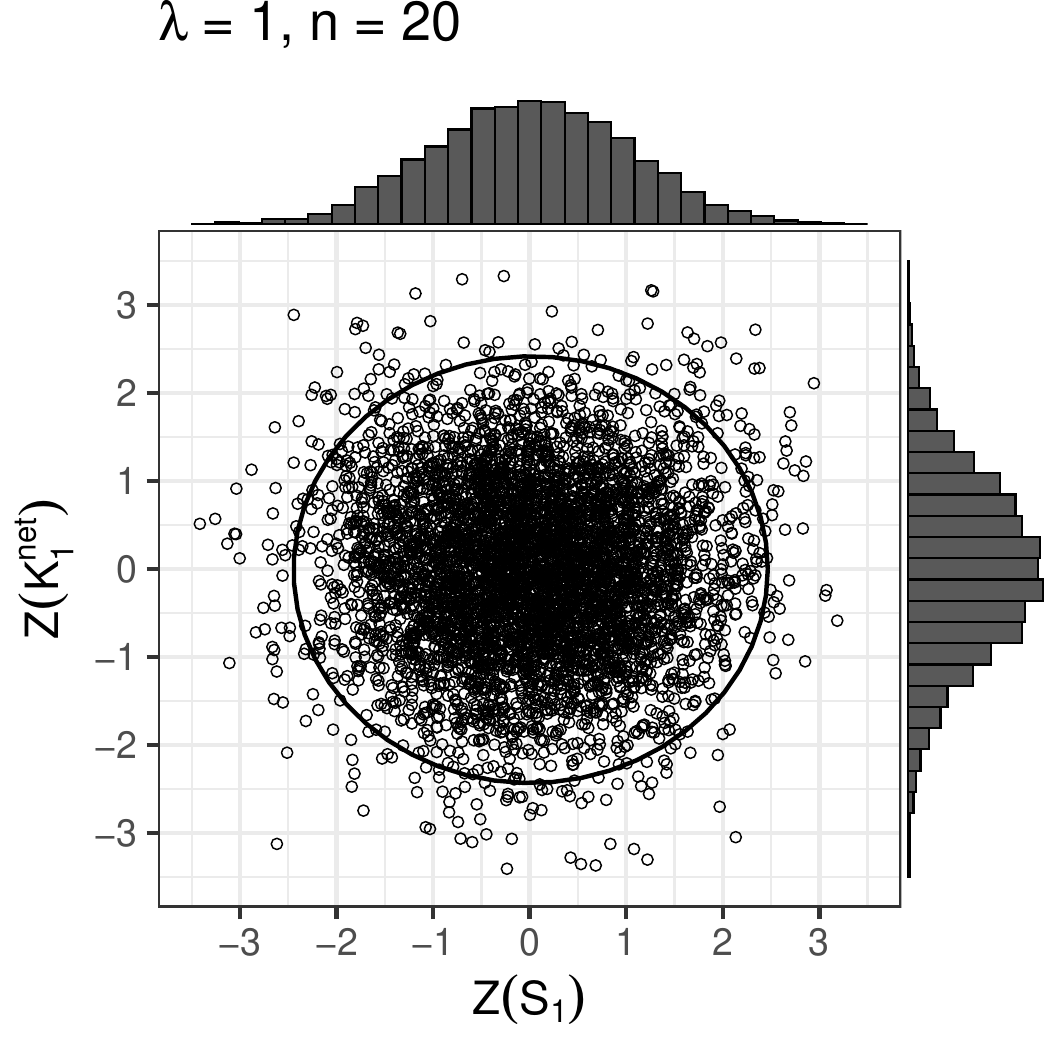}
    \includegraphics[width=0.3\textwidth]{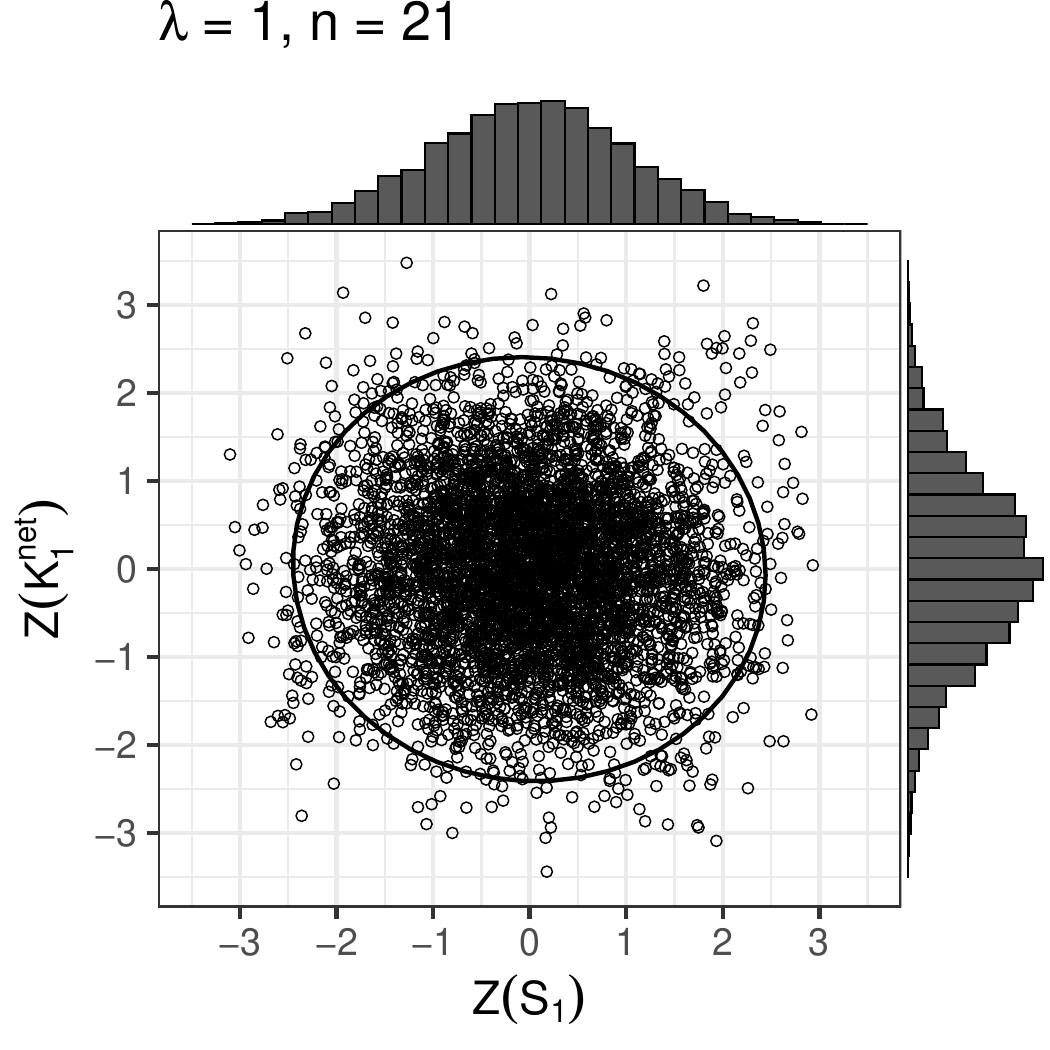}
    \includegraphics[width=0.3\textwidth]{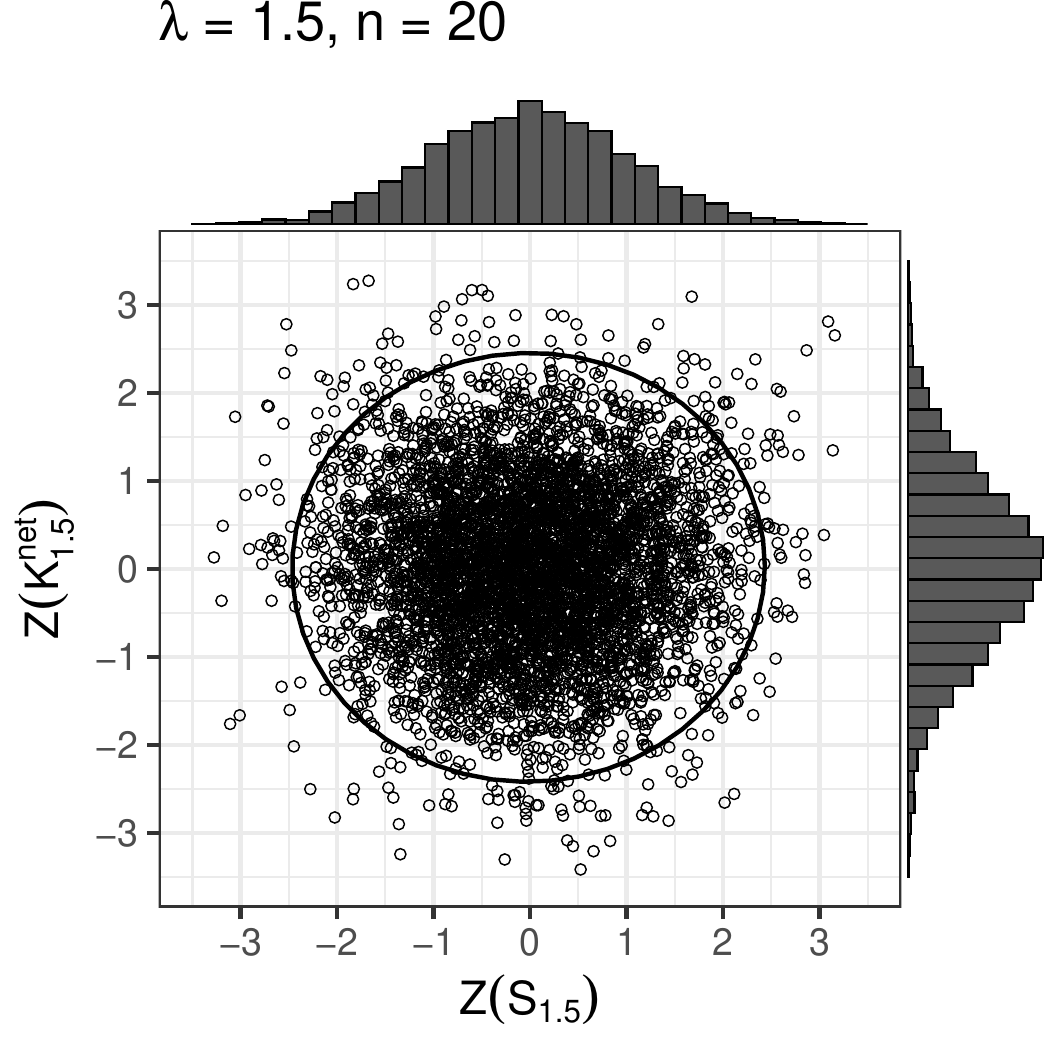}
\end{tabular}
\begin{tabular}{cc}
    \includegraphics[width=0.3\textwidth]{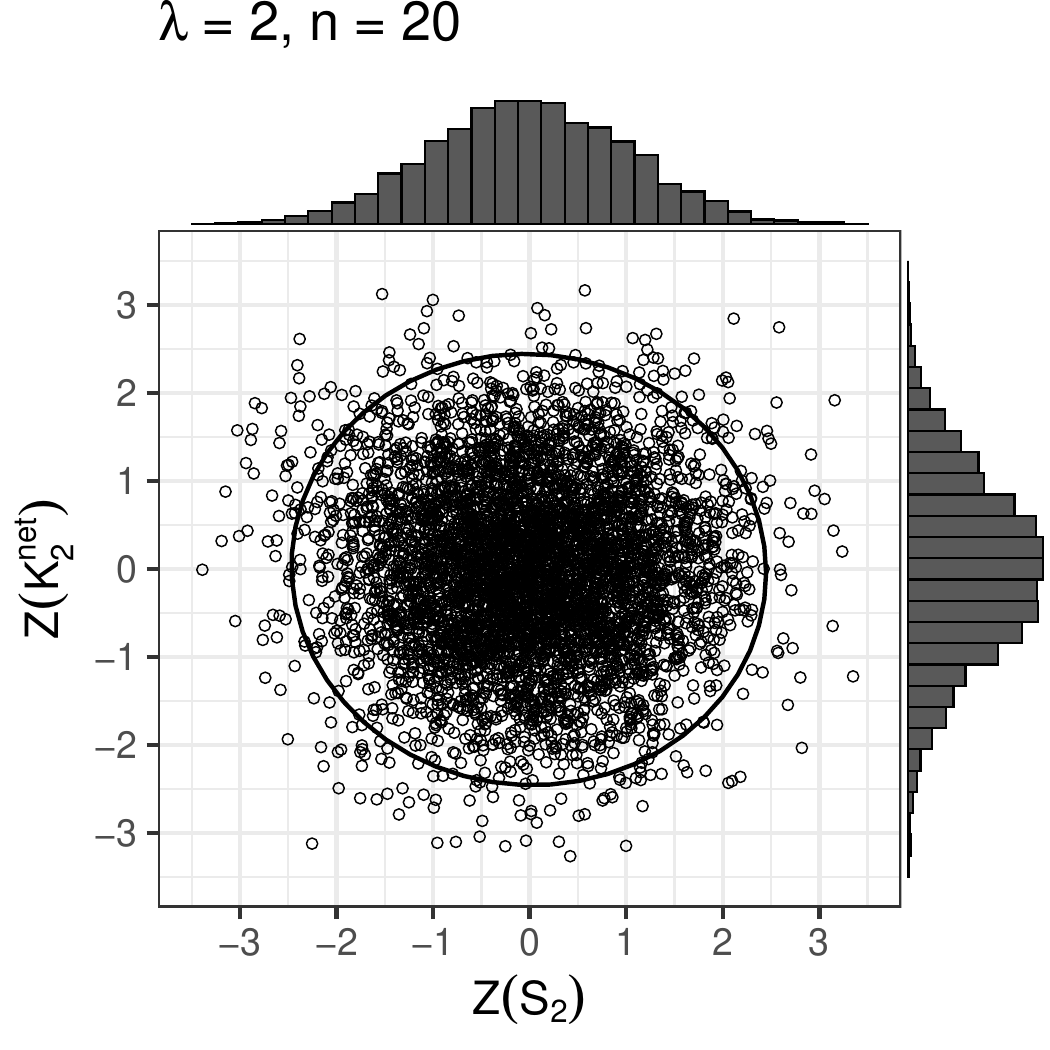}
    \includegraphics[width=0.3\textwidth]{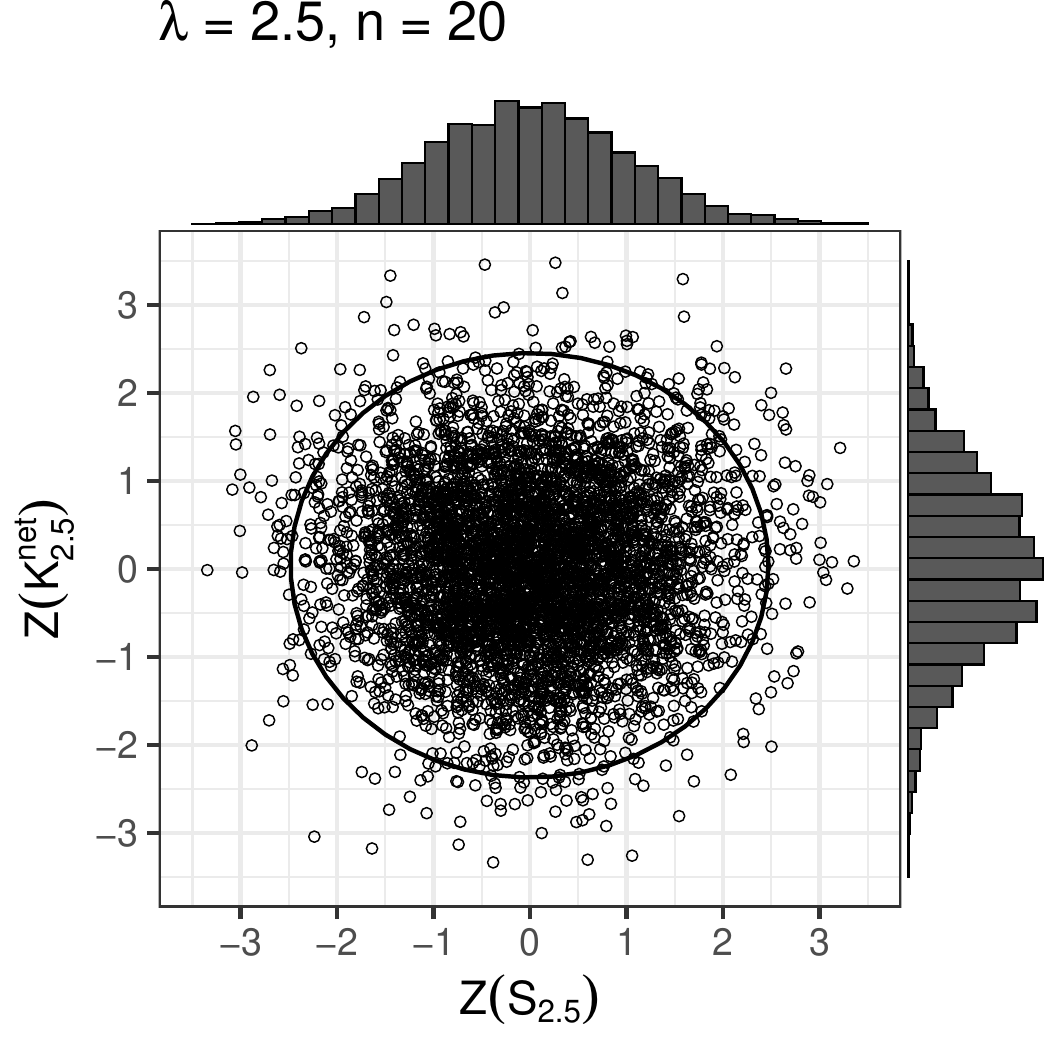}
    \includegraphics[width=0.3\textwidth]{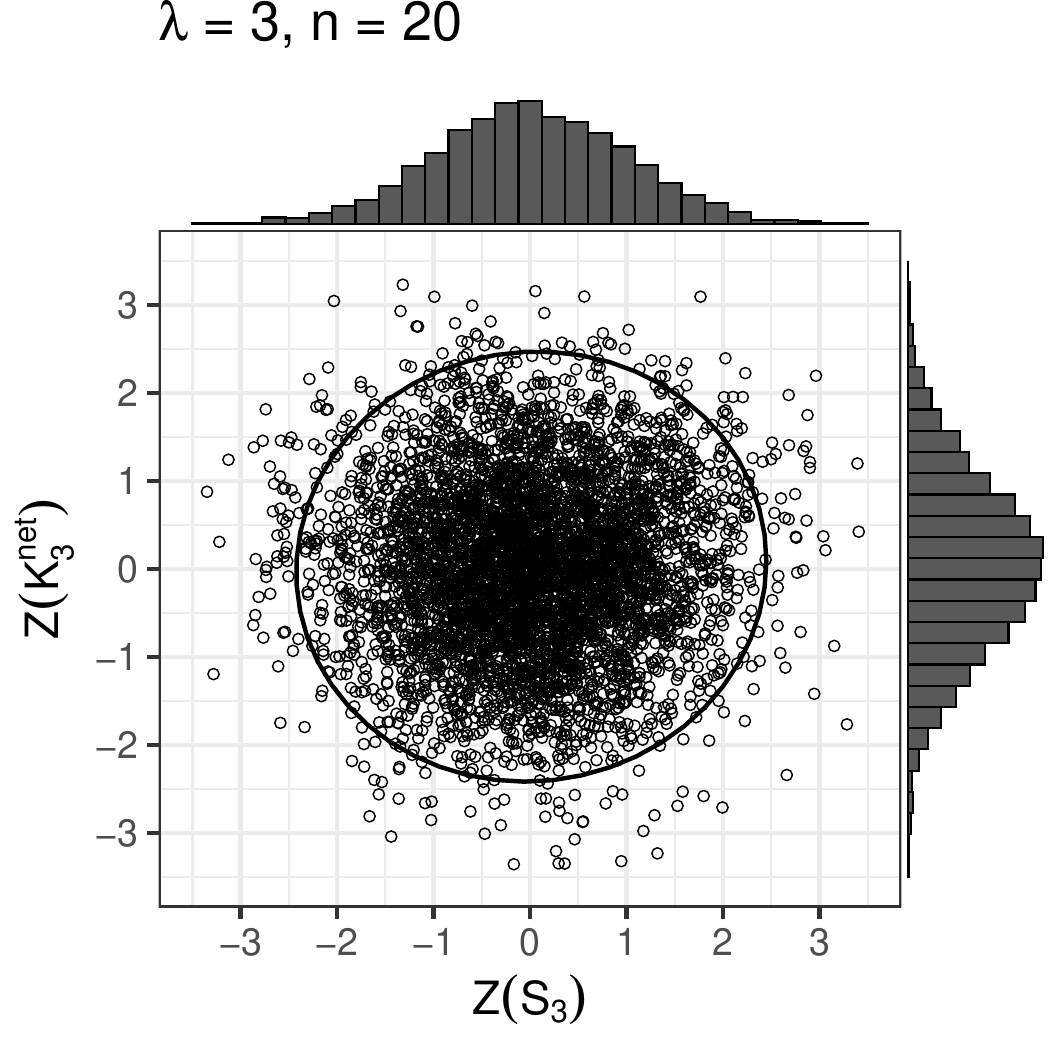}
\end{tabular}
\end{center}
\captionsetup{width=0.8\linewidth}
\vspace{-6mm}
\caption{The test statistics $Z(S_{\lambda})$ and $Z(K^{\mathrm{net}}_{\lambda})$, for 5000 samples of size $n=20$ (or $n=21$ for the case of an odd sample size when $\lambda=1$) generated under the null hypothesis, for values of $\lambda\in\{1, 1.5, 2, 2.5, 3\}$.}\label{fig:S.vs.NK}
\end{figure}

Proposition~\ref{prop:asymp.conv.H02} states that the test statistics $Z(S_{\lambda})$ and $Z(K^{\mathrm{net}}_{\lambda})$  are asymptotically
equivalent to $Z^{*\!}(S_{\lambda})$ and $Z^{*\!}(K^{\mathrm{net}}_{\lambda})$, respectively, and therefore share the same asymptotic distributions. We combine them to obtain the omnibus test statistic $X^{\mathrm{APD}}_{\lambda}$, which is asymptotically $\chi_2^2$ distributed. Proposition~\ref{prop:X.APD.finite.sample} states that these asymptotic distributions are also valid approximations for all sample sizes, and this with high numerical precision when $n\geq 20$. We assessed the quality of the fit by performing a study of the empirical level (empirical power under the null hypothesis) of the $X^{\mathrm{APD}}_{\lambda}$ statistic based on 1,000,000 simulations and on $\chi_2^2$ quantiles, for different sample sizes between 20 and 200, for $\lambda\in\{1,1.5,2,2.5,3\}$ and for significance levels $\alpha\in\{0.01,0.02,\ldots,0.15\}$. We found that the empirical and nominal levels are very close, with absolute differences less than 0.001. Furthermore, the $Z(S_{\lambda})$ and $Z(K^{\mathrm{net}}_{\lambda})$ test statistics are plotted in Figure~\ref{fig:S.vs.NK} for 5000 samples of size $n=20$ (or $n=21$ for the case of an odd sample size when $\lambda=1$) generated under the null hypothesis, for values of $\lambda\in\{1, 1.5, 2, 2.5, 3\}$. We observe that they behave like two independent standard Gaussian variables. The points outside the circle (which defines the critical region for a significance level of 5\%) are relatively well distributed, as desired.

For the omnibus test, the null hypothesis $H_0$ that the observations come from the $\mathrm{EPD}_{\lambda}(\mu,\sigma)$, with unknown $\mu$ and $\sigma$, is rejected for high values of the test statistic, specifically if the observed value of $X^{\mathrm{APD}}_{\lambda}$ is greater than the chi-square quantile $\chi^2_{2,\alpha}$, at a significance level of $\alpha$. The \textit{p}-value can be computed as $\Pr(W>X^{\mathrm{APD}}_{\lambda})$, where $W$ is a  $\chi^2_2$ distributed random variable. For the directional test against symmetric alternatives, $H_0$ is rejected if the observed test statistic $Z(K^{\mathrm{net}}_{\lambda})$ is far from 0, specifically if $Z(K^{\mathrm{net}}_{\lambda})<-z_{\alpha/2}$ or $Z(K^{\mathrm{net}}_{\lambda})>z_{\alpha/2}$, where $z_{\alpha/2}>0$ is the standard Gaussian quantile at a significance level of $\alpha$. The \textit{p}-value can be computed as $2\Pr(Z>|Z(K^{\mathrm{net}}_{\lambda})|)$, where $Z$ is a $\mathcal{N}(0,1)$ distributed random variable. We obtain the same conclusion for the directional test against asymmetric alternatives, replacing $Z(K^{\mathrm{net}}_{\lambda})$ by $Z(S_{\lambda})$.

If the null hypothesis is rejected, the cause can easily be identified. We observe that $Z(S_{\lambda})$  and $Z(K^{\mathrm{net}}_{\lambda})$ are $Z$-scores of $S_{\lambda}(\Xn)$ and $(K^{\mathrm{net}}_{\lambda}(\Xn))^{1/4}$, respectively. The further they deviate from 0, the more evidence there is for the rejection of the null hypothesis. A significant positive (resp.\ negative) value of $Z(S_{\lambda})$ suggests that the distribution is right-skewed (resp.\ left-skewed), while a value close to 0 suggests that the distribution is symmetric. A significant positive (resp.\ negative) value of $Z(K^{\mathrm{net}}_{\lambda})$ suggests that the tails of the distribution are heavier (resp.\ lighter) than those of the $\mathrm{EPD}_{\lambda}(\mu,\sigma)$, given the level of skewness.

\subsection{Asymptotic distribution of the test statistics under local alternatives}\label{sec:asymp.dist.H1}

We are interested in this section in calculating the asymptotic power of our tests, that is, the probability of correctly rejecting the null hypothesis when a specific alternative hypothesis is the true distribution, as $n\rightarrow\infty$. Since the asymptotic power of a well-designed test is equal to 1 for any fixed alternative different from the null hypothesis, we instead consider local alternatives (as was done by \citet{MR2442221} in the context of a Pareto distribution) that approach the $\mathrm{EPD}_{\lambda}(\mu,\sigma)\leqdef\mathrm{APD}_{\lambda}(1/2,\lambda,\mu,\sigma)$  as $n\rightarrow\infty$.

We then refine our general alternative hypothesis $H_1 : X_i\sim \mathrm{APD}_{\lambda}(\theta_1,\theta_2,\mu,\sigma)$, $(\theta_1,\theta_2)\neq (1/2,\lambda)$, by a family of local alternatives, defined as
\begin{equation}\label{eq:H1n.X}
    H_{\hspace{-0.2mm}1\hspace{-0.1mm},\hspace{-0.2mm}n}(\delta_1,\delta_2) : X_i \sim \mathrm{APD}_{\lambda}(\theta_{\hspace{-0.2mm}1\hspace{-0.1mm},\hspace{-0.2mm}n},
   \theta_{2\hspace{-0.3mm},\hspace{-0.3mm}n},\mu,\sigma), \quad \theta_{\hspace{-0.2mm}1\hspace{-0.1mm},\hspace{-0.2mm}n} = 1/2 +
   \frac{\delta_1}{\sqrt{n}} (1 + o(1)), \quad \theta_{2\hspace{-0.3mm},\hspace{-0.3mm}n} = \lambda + \frac{\delta_2}{\sqrt{n}} (1 + o(1)),
\end{equation}
where $(\delta_1,\delta_2)\in\R^2\backslash\{(0,0)\}$ are fixed (but arbitrary) for the omnibus test. We fix $\delta_2=0$ ($\theta_{2\hspace{-0.3mm},\hspace{-0.3mm}n}=\lambda$) for the test against asymmetric alternatives or $\delta_1=0$ ($\theta_{\hspace{-0.2mm}1\hspace{-0.1mm},\hspace{-0.2mm}n}=1/2$) for the test against symmetric alternatives. The notation $o(1)$ should be interpreted as any functions converging to 0 as $n\rightarrow\infty$. The constants $\delta_1$ and $\delta_2$ indicate the direction of the alternative relative to the $\mathrm{EPD}_{\lambda}(\mu,\sigma)$ specified in the null hypothesis. More precisely, $\delta_1<0$ (resp.\ $\delta_1>0$) represents a right-skewed (resp.\ left-skewed) alternative and $\delta_2<0$ (resp.\ $\delta_2 >0$) results in an alternative with heavier (resp.\ lighter) tails.

\begin{theorem}\label{thm:asymp.conv.H1n}
Under the local alternatives $H_{\hspace{-0.2mm}1\hspace{-0.1mm},\hspace{-0.2mm}n}(\delta_1,\delta_2)$, for fixed $\lambda\geq 1$ and $\delta_1,\delta_2\in\R$, we have, as $n\to\infty$,
\begin{equation}\label{eq:conv.ZSKstar.H1n}
    \begin{pmatrix}
        Z^{*\!}(S_{\lambda})\\[1mm]
        Z^{*\!}(K_{\lambda})
    \end{pmatrix}
    \stackrel{\mathcal{D}}{\longrightarrow} \mathcal{N}_2
    \left(
    \begin{pmatrix}
        -\delta_1 \Vone^{1/2} \\[1mm]
        -\delta_2 \Vtwo^{1/2}
    \end{pmatrix},
    \begin{pmatrix}
        1 & 0 \\[1mm]
        0 & 1
    \end{pmatrix}
    \right)
    \quad \text{ and } \quad \big(Z^{*\!}(S_{\lambda})\big)^2+\big(Z^{*\!}(K_{\lambda})\big)^2\stackrel{\mathcal{D}}{\longrightarrow}\chi_2^2(\delta_1^2\,
    \Vone+\delta_2^2 \, \Vtwo),
\end{equation}
where $Z^{*\!}(S_{\lambda})$, $Z^{*\!}(K_{\lambda})$  are given in Definition~\ref{def:asymp.stats} and $\delta_1^2\, \Vone+\delta_2^2 \,\Vtwo$ is the noncentrality parameter of the $\chi_2^2$ distribution, with
\begin{equation}
    \Vone=4(1 + \lambda) - \frac{4 \lambda^2}{\Gamma(2-1/\lambda) \Gamma(1/\lambda)}\quad \text{ and } \quad
    \Vtwo=\frac{(1+1/\lambda) \psi_1(1+1/\lambda) - 1}{\lambda^3}.
\end{equation}
\end{theorem}

The proof of Theorem~\ref{thm:asymp.conv.H1n} is presented in the Supplementary Material Section~B.7.
As a consequence, we obtain analogous convergence results for our test statistics adapted for small to moderate sample sizes.

\begin{corollary}\label{prop:asymp.conv.H1n}
Under the local alternatives $H_{\hspace{-0.2mm}1\hspace{-0.1mm},\hspace{-0.2mm}n}(\delta_1,\delta_2)$, for fixed $\lambda\geq 1$ and $\delta_1,\delta_2\in\R$, we have, as $n\to\infty$,
\begin{equation}
    Z(S_{\lambda})/Z^{*\!}(S_{\lambda})\stackrel{\mathrm{a.s.}}{\longrightarrow}1,\quad  \quad
    Z(K^{\mathrm{net}}_{\lambda})/Z^{*\!}(K_{\lambda})\stackrel{\mathcal{P}}{\longrightarrow}1,
\end{equation}
\begin{equation}\label{eq:conv.ZSK.H1n}
    \begin{pmatrix}
        Z(S_{\lambda})\\[1mm]
        Z(K^{\mathrm{net}}_{\lambda})
    \end{pmatrix}
    \stackrel{\mathcal{D}}{\longrightarrow} \mathcal{N}_2
    \left(
    \begin{pmatrix}
        -\delta_1 \Vone^{1/2} \\[1mm]
        -\delta_2 \Vtwo^{1/2}
    \end{pmatrix},
    \begin{pmatrix}
        1 & 0 \\[1mm]
        0 & 1
    \end{pmatrix}
    \right)
    \quad \text{ and } \quad X^{\mathrm{APD}}_{\lambda}\stackrel{\mathcal{D}}{\longrightarrow}\chi_2^2(\delta_1^2\,
    \Vone+\delta_2^2 \, \Vtwo),
\end{equation}
where $Z(S_{\lambda})$, $Z(K^{\mathrm{net}}_{\lambda})$ and $X^{\mathrm{APD}}_{\lambda}$ are given in Definition~\ref{def:APD.tests}, $\Vone$ and $\Vtwo$ are given in Theorem~\ref{thm:asymp.conv.H1n}.
\end{corollary}
The proof of Corollary~\ref{prop:asymp.conv.H1n} is given in Section~B.9. Corollary~\ref{prop:asymp.conv.H1n} states that
$Z^{*\!}(S_{\lambda})$ and $Z(S_{\lambda})$, like $Z^{*\!}(K_{\lambda})$ and $Z(K^{\mathrm{net}}_{\lambda})$, are asymptotically equivalent also under local alternatives. We are therefore able to compute the asymptotic powers of our test statistics under local alternatives. For example, if the true distribution is $H_{\hspace{-0.2mm}1\hspace{-0.1mm},\hspace{-0.2mm}n}(\delta_1,\delta_2)$ for fixed values of $(\delta_1,\delta_2)\in\R^2\backslash\{(0,0)\}$, the asymptotic power of the omnibus test $X^{\mathrm{APD}}_{\lambda}$ is given by $\Pr(W> \chi^2_{2,\alpha})$, where $\chi^2_{2,\alpha}$ is the quantile of the $\chi_2^2$ distribution at a significance level of $\alpha$ and $W\sim \chi_2^2(\delta_1^2\, V_{1\hspace{-0.2mm},\hspace{-0.2mm}\lambda} +\delta_2^2 \, V_{2\hspace{-0.2mm},\hspace{-0.2mm}\lambda})$. If the true distribution is $H_{\hspace{-0.2mm}1\hspace{-0.1mm},\hspace{-0.2mm}n}(0,\delta_2)$ for a fixed value of $\delta_2\neq 0$, the asymptotic power of the test $Z(K^{\mathrm{net}}_{\lambda})$ is given by $\Pr(Z<-z_{\alpha/2})+\Pr(Z>z_{\alpha/2})$, where $z_{\alpha/2}$ is the quantile of the $\mathcal{N}(0,1)$ distribution at a significance level of $\alpha$ and $Z\sim \mathcal{N}(-\delta_2 V_{2\hspace{-0.2mm},\hspace{-0.2mm}\lambda}^{1/2},1)$.  We obtain similar results for the test $Z(S_{\lambda})$. Note that the asymptotic powers are the same regardless of the sign of $\delta_1$ and $\delta_2$  by symmetry of the normal distribution.

In Figure~\ref{fig:power.curves.local.alternatives}, the asymptotic power curves are shown under local alternatives for the case $\lambda=1$ and a significance level of $\alpha=5\%$. In this case, the critical values are $\chi_{2,0.05}^2=5.991465$ and $z_{0.05/2}=1.959964$. On the left graph, the power curves of $X^{\mathrm{APD}}_{\lambda}$ and $Z(S_{\lambda})$ are compared under local alternatives, $H_{\hspace{-0.2mm}1\hspace{-0.1mm},\hspace{-0.2mm}n}(\delta_1, 0)$, for values of $0\leq \delta_1\leq 3$. We observe that the power of the directional test against asymmetric alternatives $Z(S_{\lambda})$ is uniformly higher than that of the omnibus test $X^{\mathrm{APD}}_{\lambda}$, as expected given the asymmetric family of alternatives. We also see that the power reaches its minimum at $\delta_1=0$, where it is equal to the significance level $\alpha=0.05$, and gradually increases from 0.05 to 1 as $\delta_1$ moves away from 0, as expected. On the right graph, the power curves of $X^{\mathrm{APD}}_{\lambda}$ and $Z(K^{\mathrm{net}}_{\lambda})$ are compared under local alternatives, $H_{\hspace{-0.2mm}1\hspace{-0.1mm},\hspace{-0.2mm}n}(0, \delta_2)$, for values of $0\leq \delta_2\leq 12$. We observe similar results.

\begin{figure}[H]
\begin{center}
\begin{tabular}{cc}
    \includegraphics[width=7cm]{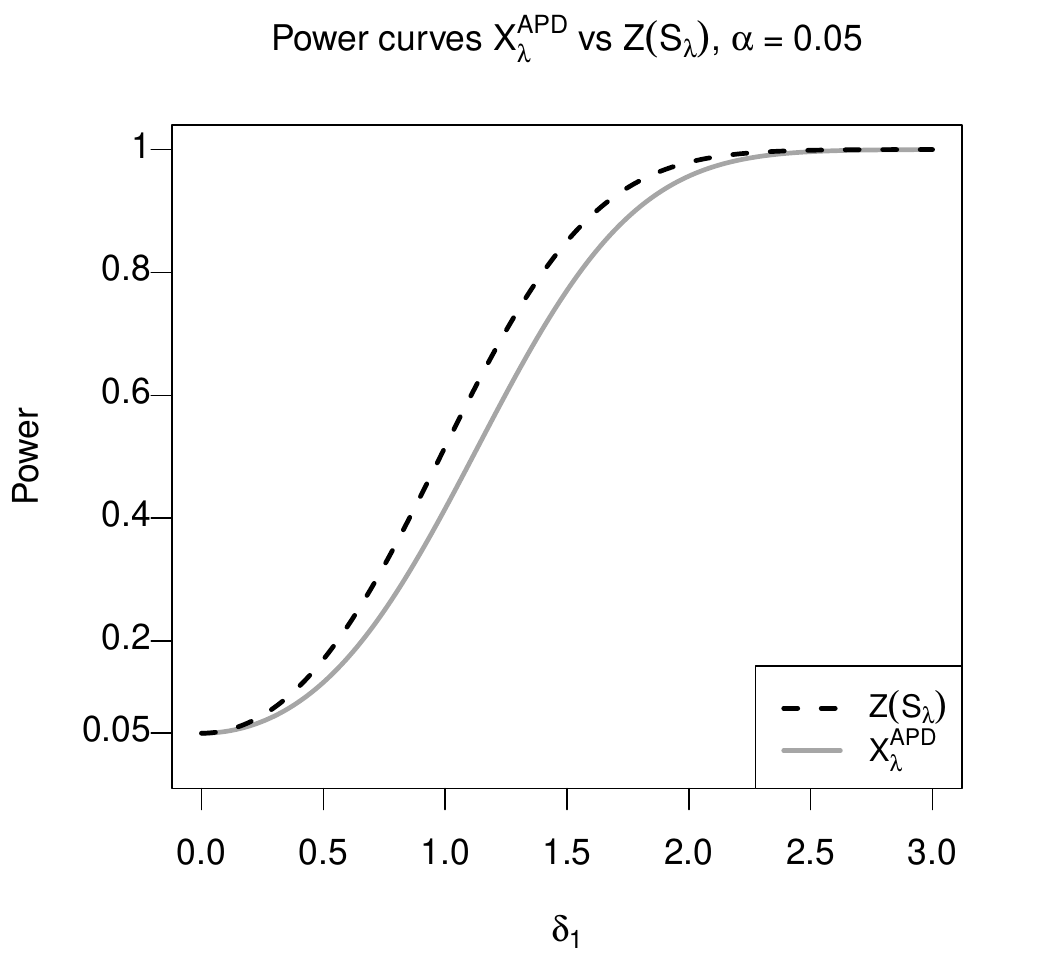}
    \includegraphics[width=7cm]{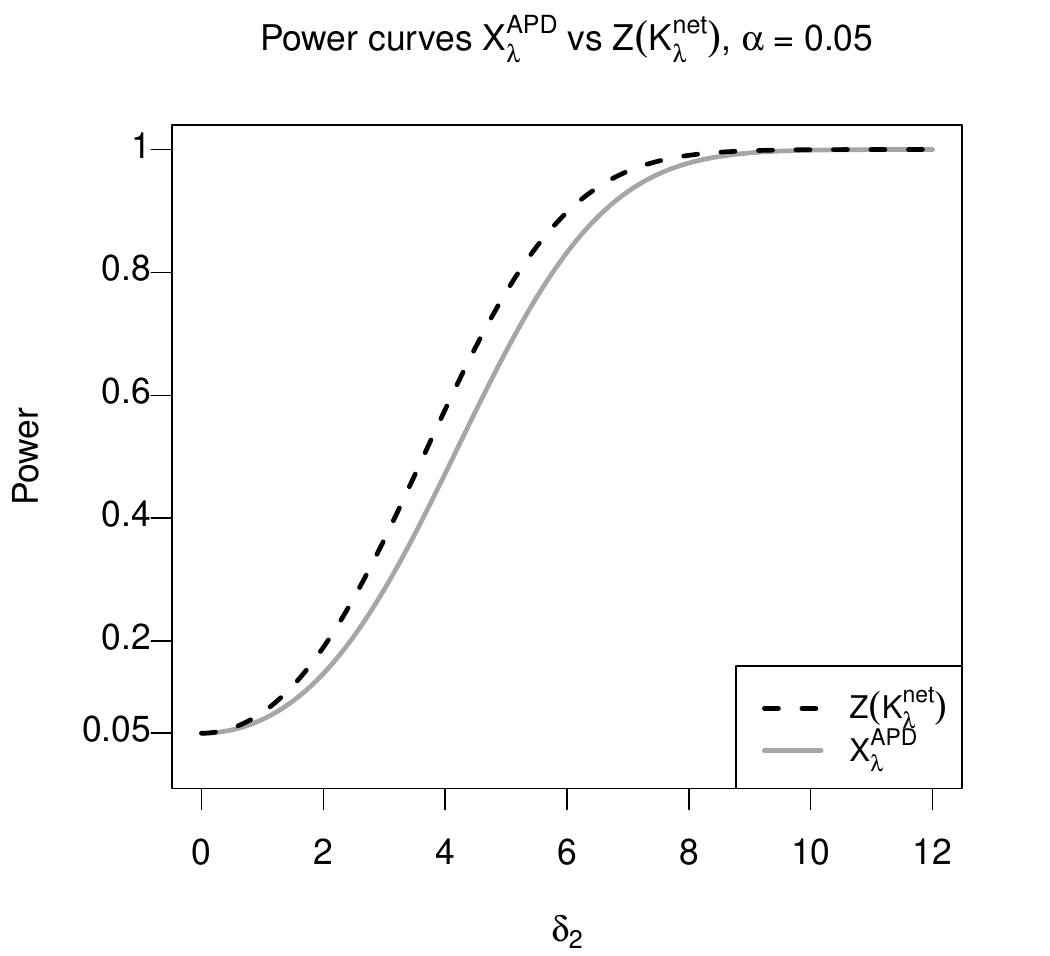}
\end{tabular}
\end{center}
\captionsetup{width=0.8\linewidth}
\vspace{-8mm}
\caption{Asymptotic power curves of $X^{\mathrm{APD}}_{\lambda}$ versus $Z(S_{\lambda})$ (left) or $Z(K^{\mathrm{net}}_{\lambda})$ (right),
    under local alternatives as a function of $\delta_1$ (left) or $\delta_2$ (right), for a fixed value of $\delta_2=0$ (left) or $\delta_1=0$ (right).
    We set $\lambda=1$ and a significance level $\alpha=0.05$.}\label{fig:power.curves.local.alternatives}
\end{figure}

\subsection{Tests of fit for Laplace and Gaussian distributions}\label{sec:Laplace.Gaussian}

In this section, we summarize our previous results for the specific cases where $\lambda$ is set to 1 or 2, namely the goodness-of-fit tests for Laplace and Gaussian distributions. We obtain simplified formulas which are worth presenting.

\subsubsection{Tests of fit for the Laplace distribution}

The omnibus test statistic $X^{\mathrm{APD}}_{1}$ for testing the null hypothesis
\begin{align*}\label{eq:H0.Laplace}
    H_0 : \Xn \sim\text{Laplace}(\mu,\sigma)=\mathrm{EPD}_{1}(\mu,\sigma), \text{ where } \mu \text{ and } \sigma \text{ are unknown,}
\end{align*}
is given by
\begin{equation}\label{eq:Xapd1}
    X^{\mathrm{APD}}_{1}\leqdef Z^2(S_{1})+Z^2(K^{\mathrm{net}}_{1}),
\end{equation}
where (using Definition~\ref{def:APD.tests} with $2\psi_1(2)=\pi^2/3-2$)
\begin{equation}\label{eq:ZSK1even}
    Z(S_{1})\leqdef\frac{n^{1/2} S_{1}(\Xn)}{\big(1 - 1.856/n^{1.06}\big)^{1/2}},\,\,\,\,
    Z(K^{\mathrm{net}}_{1})\leqdef \frac{n^{1/2}\Big(\big(K^{\mathrm{net}}_{1}(\Xn)\big)^{1/4} - (1-\gamma)^{1/4}\left(1 -
    0.422/n^{1.01}\right)\Big)}{\Big[\frac{1}{16}(1-\gamma)^{-3/2} (\pi^2/3-3)\left(1 -
    1.950/n^{0.92}+ 39.349/n^{2.3}\right)\Big]^{1/2}},
\end{equation}
for an even sample size $n$, and by
\begin{equation}\label{eq:ZSK1odd}
    Z(S_{1})\leqdef\frac{n^{1/2} S_{1}(\Xn)}{\big(1 - 0.281/n^{1.03}\big)^{1/2}},\quad
    Z(K^{\mathrm{net}}_{1})\leqdef \frac{n^{1/2}\Big(\big(K^{\mathrm{net}}_{1}(\Xn)\big)^{1/4} - (1-\gamma)^{1/4}\left(1 -
    0.198/n^{0.86}\right)\Big)}{\Big[\frac{1}{16}(1-\gamma)^{-3/2} (\pi^2/3-3)\left(1 -
    3.827/n^{1.04}\right)\Big]^{1/2}},
\end{equation}
for an odd sample size $n$, where the Euler-Mascheroni constant is given by
\begin{equation}
    \gamma\leqdef -\psi(1)=0.577215665\ldots
\end{equation}
The first-power skewness and the first-power kurtosis (see Definition~\ref{def:S.K.lambda}) are given by
\begin{equation}
    S_{1}(\Xn)\leqdef \frac{1}{n} \sum_{i=1}^n Y_i=\hat{\sigma}_{1}^{-1}\big(\bar{X}-\mathrm{median}(\Xn)\big)
    \quad \text{ and } \quad K_{1}(\Xn) \leqdef \frac{1}{n} \sum_{i=1}^n  |Y_i|\log|Y_i|,
\end{equation}
where $Y_i \leqdef \hat{\sigma}_{1}^{-1} (X_i - \hat{\mu}_{1})$, while the first-power net kurtosis (see Definition~\ref{def:net.S.K.lambda}) is given by
\begin{equation}
    K^{\mathrm{net}}_{1}(\Xn)\leqdef \max\big(0,K_{1}(\Xn)-(1/2)S^2_{1}(\Xn)\big),
\end{equation}
where the maximum likelihood estimators under $H_0$ (see Proposition~\ref{prop:MLE.estimators}) are given by
\begin{equation}
    \hat{\mu}_{1}=\mathrm{median}(\Xn)\quad \text{ and } \quad \hat{\sigma}_{1} = \frac{1}{n} \sum_{i=1}^n |X_i - \hat{\mu}_{1}|.
\end{equation}

Under the local alternatives $H_{\hspace{-0.2mm}1\hspace{-0.1mm},\hspace{-0.2mm}n}(\delta_1,\delta_2)$ with $\lambda=1$ (see \eqref{eq:H1n.X}), for fixed $(\delta_1,\delta_2)\in\R^2\backslash\{(0,0)\}$, or under the null hypothesis if $(\delta_1,\delta_2)=(0,0)$, we have (see Corollary~\ref{prop:asymp.conv.H1n}), as $n\to\infty$,
\begin{equation}
    \begin{pmatrix}
        Z(S_{1})\\[1mm]
        Z(K^{\mathrm{net}}_{1})
    \end{pmatrix}
    \stackrel{\mathcal{D}}{\longrightarrow} \mathcal{N}_2
    \left(
    \begin{pmatrix}
        -2\delta_1 \\[1mm]
        -(\pi^2/3 - 3)^{1/2}\delta_2
    \end{pmatrix},
    \begin{pmatrix}
        1 & 0 \\[1mm]
        0 & 1
    \end{pmatrix}
    \right)
    \quad \text{ and } \quad X^{\mathrm{APD}}_{1}\stackrel{\mathcal{D}}{\longrightarrow}\chi_2^2\big(4\delta_1^2+ (\pi^2/3 - 3)\delta_2^2\big).
\end{equation}
Asymptotic power curves for $Z(S_{1})$, $Z(K^{\mathrm{net}}_{1})$ and $X^{\mathrm{APD}}_{1}$ under local alternatives are illustrated in Figure~\ref{fig:power.curves.local.alternatives}.

%

\subsubsection{Tests of fit for the Gaussian distribution}

The omnibus test statistic $X^{\mathrm{APD}}_{2}$ for testing the null hypothesis
\begin{align*}\label{eq:H0.Gaussian}
   H_0 : \Xn \sim \mathcal{N}(\mu,\sigma)=\mathrm{EPD}_{2}(\mu,\sigma), \text{ where } \mu \text{ and } \sigma \text{ are unknown,}
\end{align*}
is given by
\begin{equation}\label{eq:Xapd2}
    X^{\mathrm{APD}}_{2}\leqdef Z^2(S_{2})+Z^2(K^{\mathrm{net}}_{2}),
\end{equation}
 where (using Definition~\ref{def:APD.tests} with $2\Gamma(3/2)=\Gamma(1/2)=\sqrt{\pi}$, $\psi(1/2)=-2\log(2)-\gamma$ and $\psi_1(3/2)=\pi^2/2-4$)
\begin{equation}\label{eq:ZS2}
    Z(S_{2})\leqdef\frac{n^{1/2} S_{2}(\Xn)}{\Big((3-8/\pi)\big(1 -
    1.890/n^{0.99}\big)\Big)^{1/2}},
\end{equation}
\begin{equation}\label{eq:ZK2}
    Z(K^{\mathrm{net}}_{2})\leqdef \frac{n^{1/2}\Big(\big(K^{\mathrm{net}}_{2}(\Xn)\big)^{1/4} - \big((2-\log 2-\gamma)/2\big)^{1/4}\left(1 -
    0.788/n\right)\Big)}
    {\Big[\frac{1}{16}\big((2-\log 2-\gamma)/2\big)^{-3/2} \big((3\pi^2-28)/8\big)\left(1 -
    9.327/n^{1.05}+
    14.208/n^{1.4}\right)\Big]^{1/2}},
\end{equation}
where $\gamma\leqdef -\psi(1)=0.577215665\ldots$ is the Euler-Mascheroni constant. The second-power skewness and the second-power kurtosis (see Definition~\ref{def:S.K.lambda}) are given by
\begin{equation}\label{eq:S.K.lambda2}
    S_{2}(\Xn) \leqdef \frac{1}{n} \sum_{i=1}^n Y_i^2 \mathrm{sign}(Y_i) \quad \text{and} \quad
    K_{2}(\Xn) \leqdef \frac{1}{n} \sum_{i=1}^n  Y_i^2 \log|Y_i|,
\end{equation}
where $Y_i \leqdef \hat{\sigma}_{2}^{-1} (X_i - \hat{\mu}_{2})$, while the second-power net kurtosis (see Definition~\ref{def:net.S.K.lambda}) is given by
\begin{equation}
    K^{\mathrm{net}}_{2}(\Xn)\leqdef \max\big(0,K_{2}(\Xn)-S^2_{2}(\Xn)\big),
\end{equation}
where the maximum likelihood estimators under $H_0$ (see Proposition~\ref{prop:MLE.estimators}) are given by
\begin{equation}
    \hat{\mu}_{2}=\bar{X}\quad \text{ and } \quad \hat{\sigma}_{2} = \Big(\frac{1}{n} \sum_{i=1}^n (X_i - \bar{X})^2\Big)^{1/2}.
\end{equation}

Under the local alternatives $H_{\hspace{-0.2mm}1\hspace{-0.1mm},\hspace{-0.2mm}n}(\delta_1,\delta_2)$ with $\lambda=2$ (see \eqref{eq:H1n.X}), for fixed
$(\delta_1,\delta_2)\in\R^2\backslash\{(0,0)\}$, or under the null hypothesis if $(\delta_1,\delta_2)=(0,0)$, we have (see Corollary~\ref{prop:asymp.conv.H1n}), as $n\to\infty$,
\begin{equation}
    \begin{pmatrix}
        Z(S_{2})\\[1mm]
        Z(K^{\mathrm{net}}_{2})
    \end{pmatrix}
    \stackrel{\mathcal{D}}{\longrightarrow} \mathcal{N}_2
    \left(
    \begin{pmatrix}
        -2(3-8/\pi)^{1/2}\delta_1 \\[1mm]
        -(1/2)\big((3\pi ^ 2  - 28)/ 8\big)^{1/2}\delta_2
    \end{pmatrix},
    \begin{pmatrix}
        1 & 0 \\[1mm]
        0 & 1
    \end{pmatrix}
    \right)
\end{equation}
and
\begin{equation}
    X^{\mathrm{APD}}_{2}\stackrel{\mathcal{D}}{\longrightarrow}\chi_2^2\Big(4(3-8/\pi)\delta_1^2+ (1/4)\big((3\pi ^ 2  - 28)/ 8\big)\delta_2^2\Big).
\end{equation}


\section{Three applications to real temperature data}\label{sec:example}

\subsection{Prediction errors in geostatistical modelling of ocean surface temperatures}

Mesoscale oceanography is the study of weather in the ocean at medium (`meso') scale. The most used mesoscale model MM5 (NCAR–Penn State Mesoscale Model Generation 5) enables one to make numerical ocean weather predictions. \citet{Gel2004} proposed a model to improve such predictions. Their model is the sum of two terms: one that corrects the forecasts for additive and multiplicative bias through spatio-temporal covariables, and one mean-zero stationary Gaussian spacetime stochastic process error term.

In that context, our first set of temperature data consists of prediction errors of 48-hour ahead MM5 forecasts of surface temperature measured at 96 locations in the US Pacific Northwest on 3-January-2000. The prediction error is the difference between the forecasted and observed surface temperature. This dataset is available as \texttt{bias.rda} in the \texttt{R} package \texttt{lawstat} \citep{lawstat}.

The histogram of this data (Figure~\ref{fig:histbias}) exhibits some heavier tails than the fitted $\mathcal{N}(\hat{\mu}_2=0.158,\hat{\sigma}_2=3.209)$ density (dashed line), and a peak slightly lower than the one captured by the fitted Laplace($\hat{\mu}_1=-0.163,\hat{\sigma}_1=2.409$) density (dotted line). 

Consequently, we decided to fit an asymmetric EPD distribution with $\lambda=1.5$, a value between $\lambda=1$ (Laplace) and $\lambda=2$ (Normal). This leads to the $\textrm{EPD}_{1.5}(\hat{\mu}_{1.5}=0.0267, \hat{\sigma}_{1.5}=2.819)$ density (bold line) which better captures the pattern of the histogram.

\begin{figure}[ht]
\centering
\includegraphics[height=7cm]{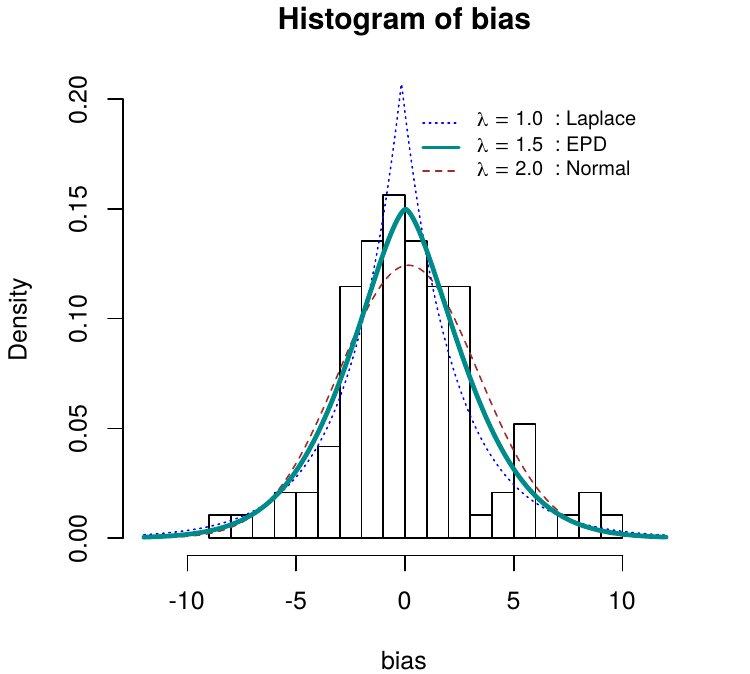}\caption{Histogram of $n=96$ prediction errors of ocean surface temperature data and the best fits obtained via a maximum likelihood approach assuming a Laplace or $\text{EPD}_1(\hat{\mu}_1,\hat{\sigma}_1)$ (dotted line), an $\text{EPD}_{1.5}(\hat{\mu}_{1.5},\hat{\sigma}_{1.5})$ (bold line) and a Gaussian or $\text{EPD}_2(\hat{\mu}_2,\hat{\sigma}_2)$ (dashed line) distribution.}\label{fig:histbias}
\end{figure}

Of course, one can apply goodness-of-fit tests to confirm these findings. \citet{Gel_et_al_2007} concluded to the non-normality of these data using the Shapiro-Wilks ($p=0.051$), Jarque-Bera ($p=0.043$) and Bonett-Seier ($p=0.036$) tests and our omnibus normality $X_{\lambda}^{\text{APD}}$ test goes in the same direction ($p=0.020$). In Table~\ref{Tab:bias}, we apply our tests for $\lambda=1.0,1.5$ and $2.0$, and we conclude that the $\textrm{EPD}_{1.5}$ is better than the normal and Laplace distributions fitted above, even if the Laplace also appears as a good fit.

\begin{figure}[ht]
\centering
\includegraphics[height=7cm]{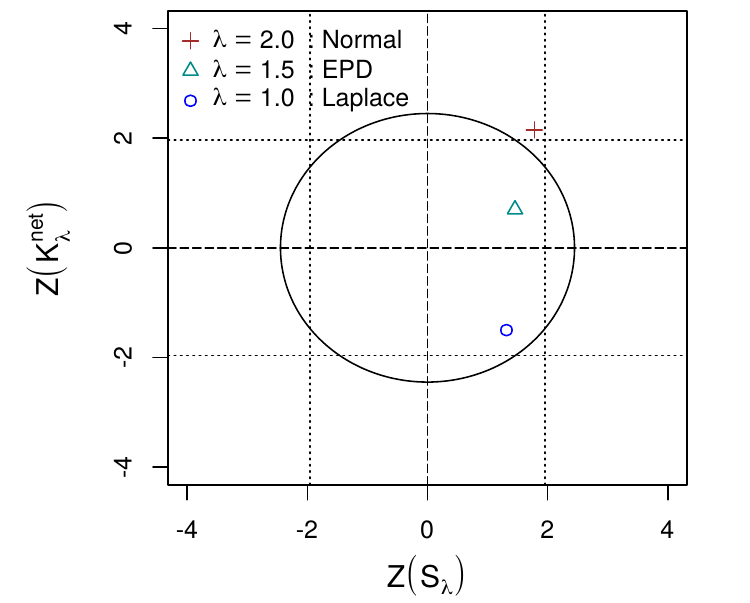}\caption{The test statistics $Z(K_{\lambda}^{\text{net}})$ versus $Z(S_{\lambda})$, for $\lambda\in\{1.0,1.5,2.0\}$. The region outside the solid circle (with a radius of $\sqrt{\chi^2_{2,0.05}}=2.44775$) defines the critical region for the test statistic $X_{\lambda}^{\text{APD}}=Z^2(S_{\lambda})+Z^2(K_{\lambda}^{\text{net}})$ at a significance level of 5\%. The region outside the two vertical (resp.\ horizontal) dotted lines, with abscissa (resp.\ ordinate) $-1.96$ and $1.96$, defines the critical region for the $Z(S_{\lambda})$ (resp.\ $Z(K_{\lambda}^{\text{net}})$) test statistic at a significance level of 5\%.}\label{fig:testsbias}
\end{figure}

\begin{table}[H]
  \centering\caption{Test statistic values and their associated $p$-values.}\label{Tab:bias}
  \begin{tabular}{c|rrr}
    \hline
    $\lambda$                 & \multicolumn{1}{c}{1.0}                   & \multicolumn{1}{c}{1.5}                   & \multicolumn{1}{c}{2.0} \\
    \hline
    $Z(S_{\lambda})$            & 1.314 (0.189)   & 1.457 (0.145)  & 1.778 (0.075) \\
    $Z(K_{\lambda}^{\text{net}})$ &  -1.501 (0.133) & 0.700 (0.484) & 2.149 (0.032) \\
    $X_{\lambda}^{\text{APD}}$    & 3.979 (0.137)  & 2.612 (0.271)  & 7.781 (0.020) \\
    \hline
    \end{tabular}
  \end{table}

The information in Table~\ref{Tab:bias} is also displayed graphically in Figure~\ref{fig:testsbias}. Overall, the $\textrm{EPD}_{1.5}$ symbol is well within the circle whereas the symbols for the two other distributions are located closer to the circle at an angle of about $\pm45^{\circ}$ which leads us to favour the former distribution.

From Figure~\ref{fig:testsbias}, for the three models ($\lambda = 1.0, 1.5$ and $2.0$), the data are right-skewed with $Z(S_{\lambda}) >0$, though these results are not significant at the 5\% level (all three symbols are located between the two vertical dotted lines with abscissa $-1.96$ and $1.96$). We also see that the data have
\begin{itemize}
\item heavier tails than the Gaussian distribution (positive value of $Z(K^{\text{net}}_2)=2.14926$; significant at the 5\% level as indicated by the $+$ symbol above the upper horizontal dotted line);
\item shorter tails than the Laplace distribution (negative value of $Z(K^{\text{net}}_1)=-1.50052$; non-significant since the $\circ$~symbol is between the two horizontal dotted lines);
\item slightly heavier tails than the $\textrm{EPD}_{1.5}$ distribution (positive value of $Z(K^{\text{net}}_{1.5})=0.69968$; clearly non-significant with a $\triangle$~symbol close to the origin).
  \end{itemize}

Based on these results, one could extend the initial model of \cite{Gel2004} to use a non-Gaussian random field error term (see, e.g., \citet{Aberg2011}). This would then allow for the estimation of probabilities of various weather scenarios by simulating $\textrm{EPD}_{1.5}$ or Laplace random fields for the errors.

\subsection{Home refrigeration temperatures and food safety}\label{refrig}


In this section, we analyse a dataset\footnote{Source: Q62 in XLS file at \url{https://www.foodrisk.org/resources/sendFile/49} and see \url{https://www.foodrisk.org/resources/sendFile/46} for a description of the format.} containing $n = 2,037$ refrigerator temperatures collected in a study aiming at helping consumers reduce bacterial growth and thus ensure the quality and safety of food products stored at home for U.S.\ households \citep{Kosa_et_al_2007}. For convenience, we also provide a CSV version of this file, called \texttt{refrig.csv}, in the Supplementary Material A.3. This dataset was also analysed by \cite{Pouillot_et_al_2010}. In the original file, the data are encoded with integers in $\{1,\dots,22\}$, corresponding to the temperatures (in degrees Fahrenheit) ``more than $60^{\circ}$F'', 58F, 56F, 54F, ..., 24F, 22F, 20F, and ``less than $20^{\circ}$F'', respectively.
In order to transform these grouped and censored data into uncensored continuous observations, we added to each one of the 20 grouped data an independent observation sampled from a $\text{Uniform}\hspace{0.2mm}(-1,1)$. We also replaced the ``less than $20^{\circ}$F'' and ``more than $60^{\circ}$F'' censored values with random Laplace($\mu=39.3,\sigma=4.23$) observations on the intervals $(-\infty, 20)$ and $(60,\infty)$, respectively, where $\mu=39.3$ and $\sigma=4.23$ are the maximum likelihood estimates provided by \cite{Pouillot_et_al_2010} using the censored data.
An histogram of these uncensored jittered values is shown in Figure~\ref{fig.refrig} together with the best fit obtained via a maximum likelihood approach, assuming a Laplace ($\hat{\mu}_1=39.36,\hat{\sigma}_1= 4.23$) or a Gaussian ($\hat{\mu}_2=39.29,\hat{\sigma}_2= 6.89$) distribution. The figure makes it clear that a Laplace distribution is more appropriate for these data than a normal distribution.


\begin{figure}[H]
        \captionsetup{width=.80\textwidth}
        \centering
        \includegraphics[width=8cm]{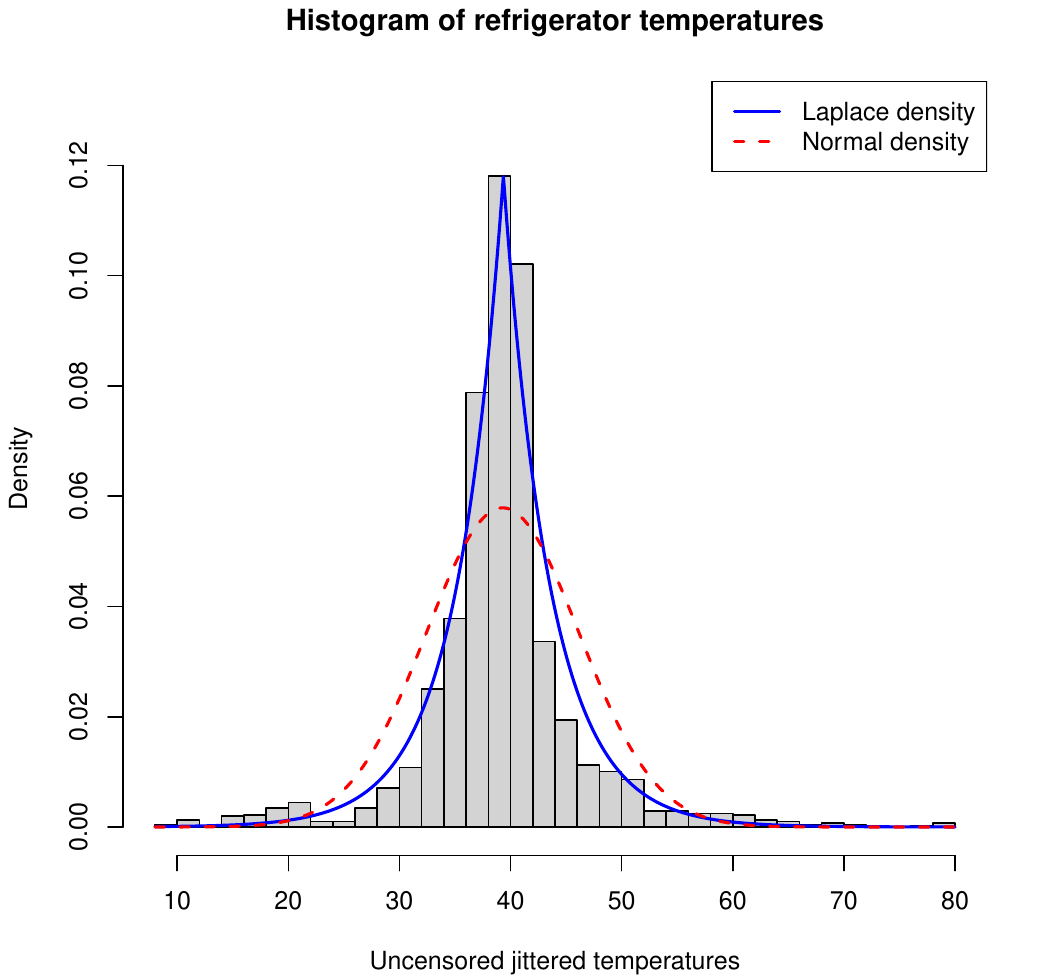}
        \vspace{-2mm}
        \caption{Histogram of $n=2,037$ uncensored jittered refrigerator temperatures and the best fits obtained via a maximum likelihood approach assuming a Laplace (solid line)
        or a  Gaussian (dashed line) distribution.}
        \label{fig.refrig}
\end{figure}

Of course, for such a large sample size, any goodness-of-fit test would almost certainly reject any \textit{a priori} distributional assumption. In order to make the analysis more interesting, we considered the original $n = 2,037$ (uncensored jittered) temperatures as the population from which we sampled at random and without replacement, $1,000,000$ times, a sub-sample of (moderate) size $50$.
For each such sub-sample (of size $50$) we computed the $p$-value of our new test $X^{\mathrm{APD}}_{\lambda}$ under both the null hypothesis that $\lambda=1$ (i.e., a Laplace distribution) or under the null hypothesis that $\lambda=2$ (i.e., a normal distribution). Numerical summaries of the $p$-values for these two cases are given in Table~\ref{tab.refrig}. We also calculated that Laplacity was not rejected for 57.5\% of the sub-samples, while Gaussianity was not rejected only for 2.5\% of these sub-samples, for a significance level of 5\%. Based on these results, one can safely assume that a Laplace distribution is more appropriate than a normal distribution for these data.

\vspace{3mm}
\begin{table}[ht]
        \captionsetup{width=.80\textwidth}
        \centering
        \begin{tabular}{c|cccc}
            \hline
            \textbf{Case}  &     1st Quart.  &  Median  &    Mean &  3rd Quart.   \\
            \hline
            Laplace &  0.01412  & 0.07688 & 0.18661 & 0.27234  \\
            Normal &   0.00000  & 0.00000 & 0.00629 & 0.00016  \\
            \hline
        \end{tabular}
        \vspace{4mm}
        \caption{Numerical summary of $1,000,000$ $p$-values for our new test $X^{\mathrm{APD}}_{\lambda}$ applied on random sub-samples of size $50$ of
        the original dataset of $n=2,037$ temperatures. We test for Laplacity and for Gaussianity.}
        \label{tab.refrig}
\end{table}
\vspace{-3mm}

In order to steer clear of foodborne illnesses, the U.S.\ Food and Drug Administration recommends to ``keep the refrigerator temperature at or below $40^{\circ}$F''%
    \footnote{Source: \url{https://www.fda.gov/consumers/consumer-updates/are-you-storing-food-safely}}.
Thanks to our parametric fit using a Laplace distribution, one can estimate the proportion of refrigerators having a temperature above this threshold to be $42.9\%$ (in the population of 2007 U.S.\ households). In comparison, this proportion under the Gaussian distribution is $45.9\%$, reflecting the underestimation of the central values and the overestimation of the values in the ``shoulders''.

    %
    %

\subsection{London time series temperatures}

\cite{Shea_1987} proposed an algorithm for the computation of the \textit{exact} likelihood of a multivariate time series and illustrated the methodology on a bivariate dataset of wind speed and temperature values. This data was originally studied by \cite{Piggott_1980} in an attempt to predict the residential consumption of gas in London. Here we revisit the task of building a univariate time series model for the $n=366$ daily temperature values $T_t$, $t=1,\dots,366$, represented in Figure~\ref{fig:tempLondon}.

\begin{figure}[H]
        \centering
        \includegraphics[width=7cm]{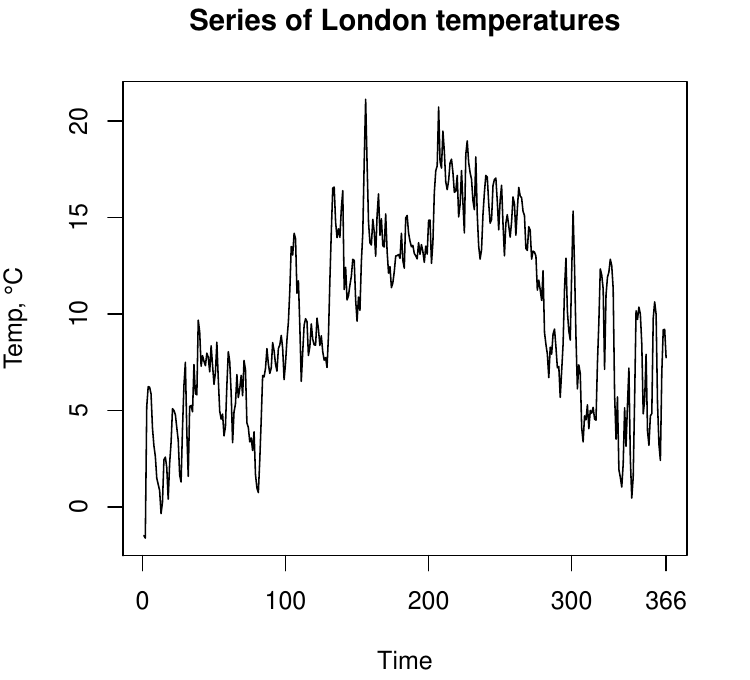}
        \vspace{-3mm}
        \caption{Daily temperature values in London \citep{Piggott_1980}.}\label{fig:tempLondon}
\end{figure}

It is clear that this series is non-stationary, leading us to compute first order differences $\Delta T_t=T_t-T_{t-1}$, $t=1,\ldots,n$ (taking $T_0=0$), the histogram of which is shown in Figure~\ref{fig:histdiff} (left). An ACF plot (Figure~\ref{fig:histdiff}; right) suggests that we use the $\text{MA}(4)$ model $\Delta T_t=c+\varepsilon_t+\theta_1\varepsilon_{t-1}+\cdots+\theta_4\varepsilon_{t-4}$ if a Gaussian distribution (with variance $\sigma^2$) is assumed for the innovations ($\varepsilon_t$). This is often the case for temperatures as discussed in the previous example.

\begin{figure}[ht]
        \centering
        \includegraphics[width=7cm]{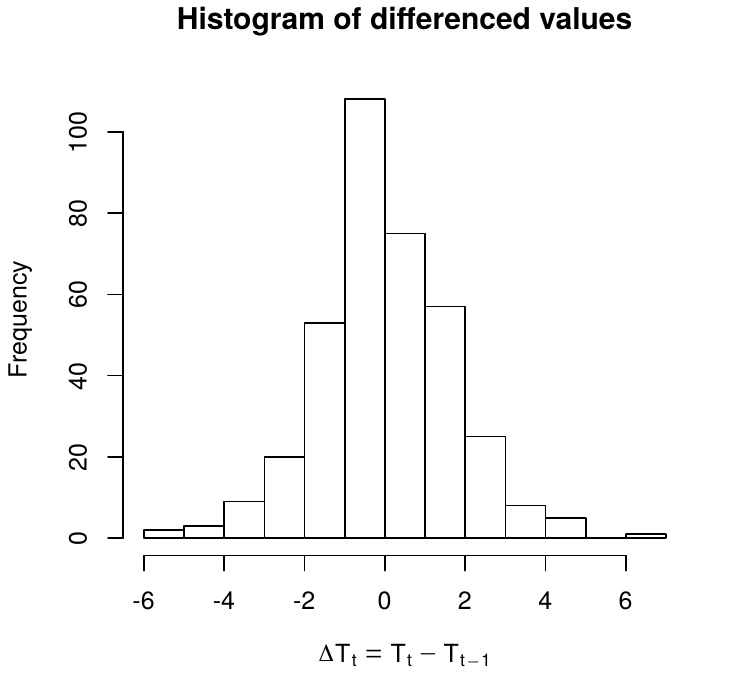}\includegraphics[width=7cm]{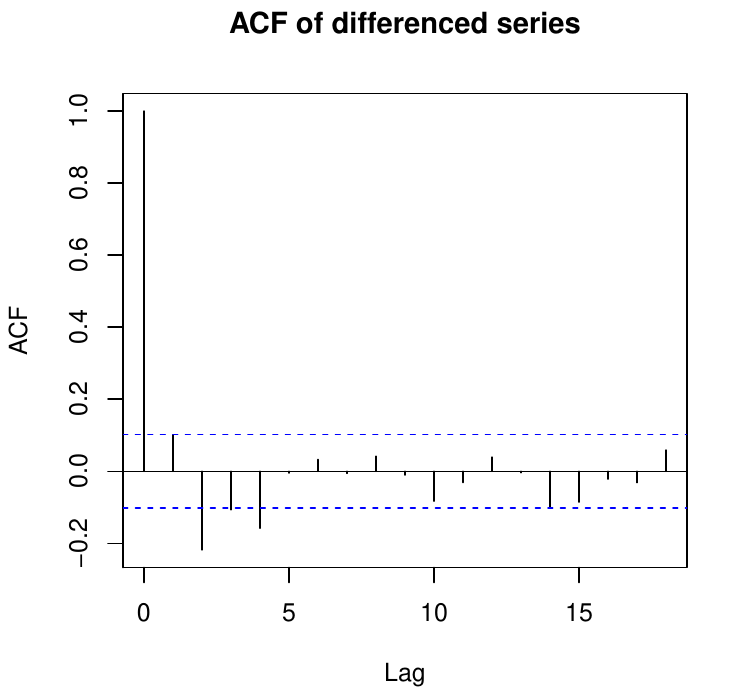}
        \vspace{-3mm}
        \caption{Histogram (left) and ACF plot (right) for the differenced series $\{\Delta T_t;t=1,\ldots,n\}$.}
        \label{fig:histdiff}
\end{figure}


This is also confirmed by selecting the best model, in terms of AIC, among all $\text{MA}(q)$ models, $q = 1,\dots,10$, fitted on the centred $\Delta T_t$ values. (An AIC value of $1380$ was found using \texttt{Matlab}; the coefficients of the fitted Gaussian $\text{MA}(4)$ model, without intercept, being $\hat{\theta}_1=0.0708$, $\hat{\theta}_2=-0.2978$, $\hat{\theta}_3=-0.1506$ and $\hat{\theta}_4=-0.1955$, with an estimated variance of $\hat{\sigma}^2=2.4734$ and a log-likelihood of $-685.05$.) The exact same model was fitted by \cite{Shea_1987} and by \cite{Ducharme_2004}. 

However, a look at the histogram of the residuals ($e_t$) of the fitted Gaussian MA(4) model, and at the associated QQ-plot
in Figure~\ref{fig:histGauss}, does not fully support the Gaussian assumption of the innovations. It is especially clear in the QQ-plot that the tails are heavier than those of a Gaussian distribution, while the histogram reveals that the middle peak is not adequately captured. This is confirmed by our new test of normality $X^{\mathrm{APD}}_{2}$ applied on these residuals ($p=1.25 \cdot 10^{-5}$). Similar results are obtained if one uses the Jarque-Bera test ($p = 5.76 \cdot 10^{-7}$) or the \cite{Duchesne_et_al_2016} test ($p = 2.56 \cdot 10^{-4}$).


\begin{figure}[t]
        \centering
        \includegraphics[width=11cm]{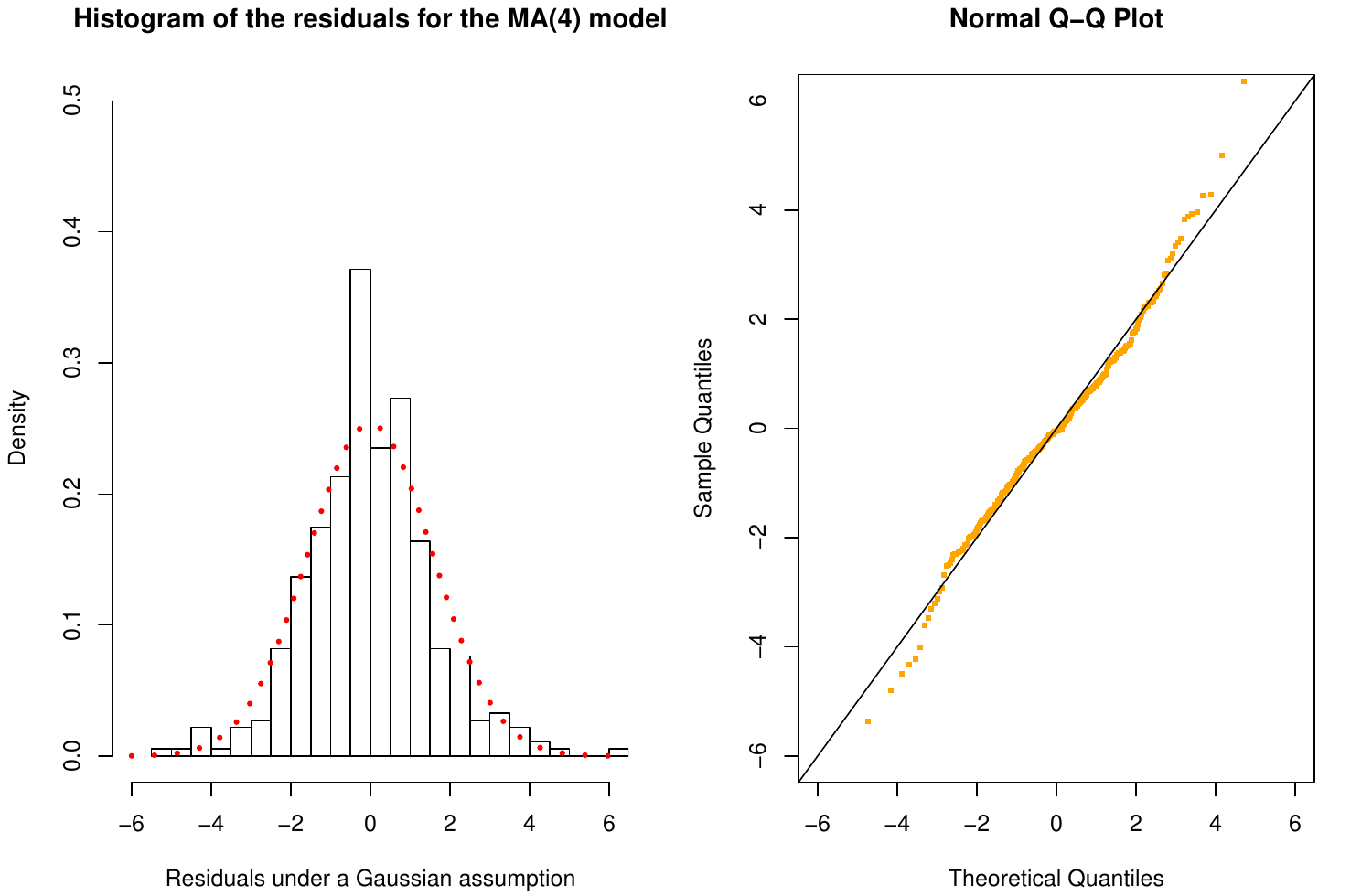}
        \vspace{-2mm}
        \caption{Histogram (left) and Normal QQ plot (right) for the residuals $e_t$ of the fitted Gaussian MA(4) model.
        A density curve of a $\mathcal{N}(\bar{x}_e,\hat{\sigma}^2_e)$ is superimposed to the histogram. The QQ-line passes through the origin with a slope of 1.}
        \label{fig:histGauss}
\end{figure}

According to \cite{Lomnicki_1961}, if the innovations $\epsilon_t$ are Gaussian, then the observations $\Delta T_t$ are also Gaussian, thanks to the central limit theorem. Therefore, we also applied our test of normality $X^{\mathrm{APD}}_{2}$ directly on the $\Delta T_t$ observations and obtained a $p$-value of $1.41 \cdot 10^{-4}$, which suggests that the observations, and hence the innovations, are not Gaussian. Consequently, it might not be such a good idea to fit an MA model with Gaussian innovations to these data. Now, even if non-Gaussian ARMA models are rather scarce in the literature (see \cite{Li_McLeod_1988,Lehr_Lii_1998,Ozaki_Iino_2001,Trindade_et_al_2010} for the few references we could find), we believe this topic deserves more attention as illustrated below.

The histogram in Figure~\ref{fig:histGauss} (left) seems more peaky than the corresponding Gaussian density, with heavier tails. The statistic for the directional test of fit against symmetric alternatives is 
$Z(K_2^{\text{net}})=4.58$, which is statistically significant. There is a slight right asymmetry. The statistic for the directional test of fit against asymmetric alternatives is 
$Z(S_2)=0.85$.
Even though a linear combination of Laplace distributed random variables is not necessarily of Laplace type, the shape of the histogram suggests to fit an MA model using an $\mathcal{A}\mathcal{L}(\kappa,\delta,\tau)$ distribution (asymmetric Laplace, AL) for the innovations. The best model (fitted using \texttt{Matlab} on the centred $\Delta T_t$ values and assuming an AL likelihood parametrized as in \cite{Trindade_et_al_2010}), in terms of both AIC and parsimony, is also an $\text{MA}(4)$ model. The estimated coefficients are $\hat{\theta}_1=0.1061$, $\hat{\theta}_2=-0.3645$, $\hat{\theta}_3=-0.1838$ and $\hat{\theta}_4=-0.1538$, with an $\text{AIC}$ of $1365$, a value much smaller than the AIC of 1380 we obtained when fitting an $\text{MA}(4)$ model with Gaussian innovations.
The estimated values of the location and scale parameters of the $\mathcal{A}\mathcal{L}(\kappa,\delta,\tau)$ distribution are respectively $\hat{\delta} = -0.1149$ and $\hat{\tau} = 1.6496$. The estimated skewness parameter value is $\hat{\kappa} = 0.9519$, quite close to 1, a value associated to perfect symmetry. A look at the histogram of the residuals (for this second model) and the associated $\text{QQ}$-plot (Figure~\ref{fig:histLap}) now favours an asymmetric Laplace assumption
As we can see in Figure~\ref{fig:histLap}, a density curve of an $\mathcal{A}\mathcal{L}(\hat{\kappa},\hat{\delta},\hat{\tau})$ (solid blue line) has been superimposed to the histogram of the residuals and the fit is good. We also added a symmetric Laplace $\mathcal{A}\mathcal{L}(1,\hat{\delta},\hat{\tau})$ (dashed red line) and we observe that it is practically identical, given that $\hat{\kappa}=0.9519$ is close to 1, as mentioned above. Therefore, we applied our Laplacity test $X^{\mathrm{APD}}_{1}$ on the residuals of this $\text{MA}(4)$ model, and we do not reject the null hypothesis ($p\text{-value} = 0.126$).

\begin{figure}[t]
        \centering
        \includegraphics[width=11cm]{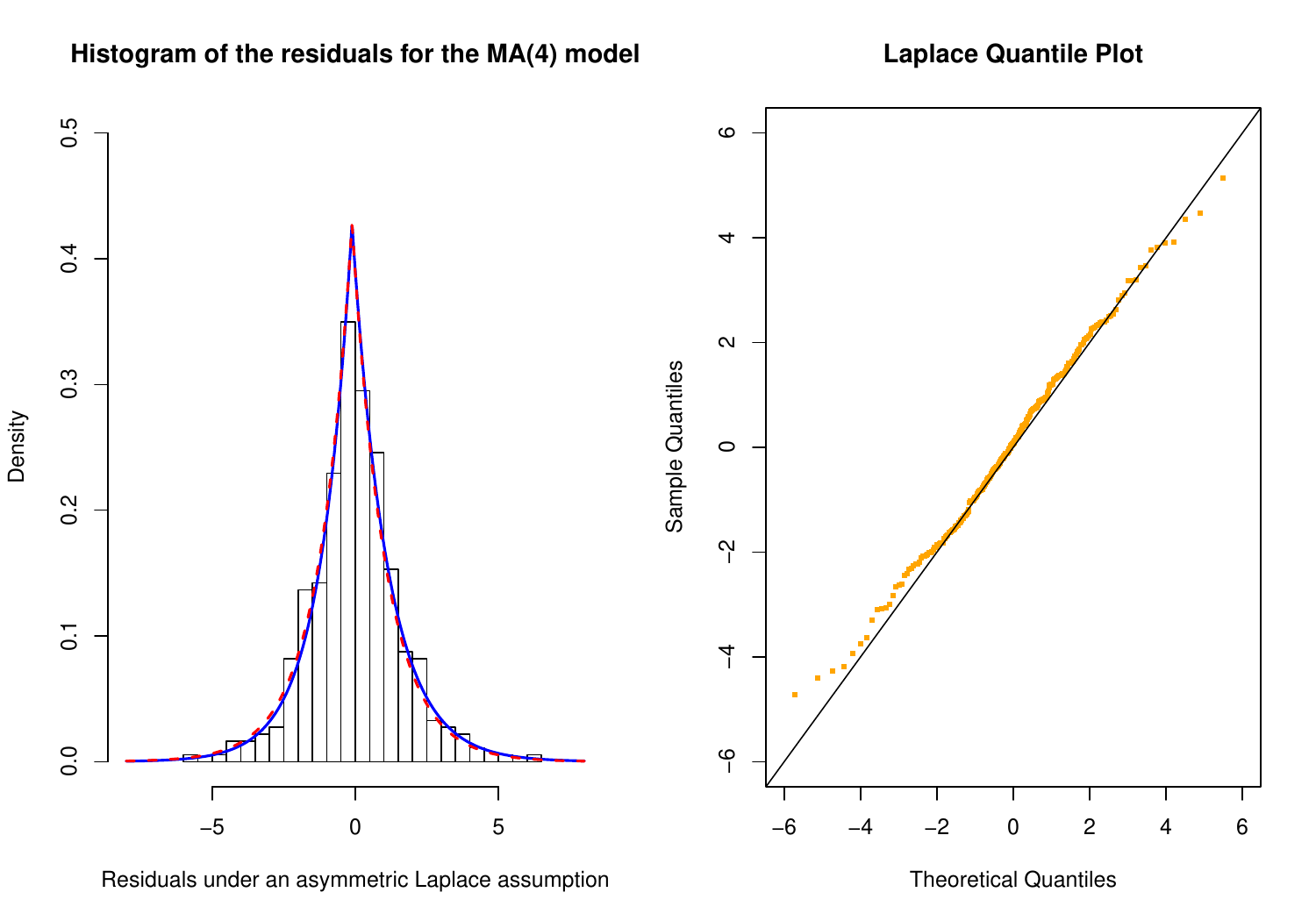}
        \vspace{-2mm}
        \caption{Histogram (left) and asymmetric Laplace QQ plot (right) for the residuals $e_t$ of the fitted MA(4) model. Density curves of an
        $\mathcal{A}\mathcal{L}(\hat{\kappa},\hat{\delta},\hat{\tau})$ (solid blue line) and a symmetric Laplace $\mathcal{A}\mathcal{L}(1,\hat{\delta},\hat{\tau})$
        (dashed red line) are superimposed to the histogram. The QQ-line passes through the origin with a slope of 1.}
        \label{fig:histLap}
\end{figure}

To summarize, based on these results, there is a substantial support in favour of the $\text{MA}(4)$ model with asymmetric Laplace errors (AIC=1365; $p=0.126$) compared to the $\text{MA}(4)$ model with Gaussian errors (AIC=1380; $p=1.25 \cdot 10^{-5}$); see \cite{Burnham2004}.
One could also consider the more parsimonious $\text{MA}(2)$ model with asymmetric Laplace innovations (AIC=1380; $p=0.054$). 
The histogram of the residuals and the asymmetric Laplace QQ-plot for this $\text{MA}(2)$ model (not shown here) are very similar to the ones displayed in Figure~\ref{fig:histLap}. 


\section{Empirical power comparison}\label{sec:empirical.power.analysis}

Empirical power comparison using Monte Carlo simulations is useful if done thoroughly, with a large number of tests and alternatives. Given the magnitude of this task, we focus our analysis on Laplace tests, or equivalently on $\mathrm{EPD}_{\lambda}$ tests with $\lambda=1$. Note that an empirical power comparison was performed in \cite{Desgagne_Lafaye_2018}, where 13 of the best goodness-of-fit tests for the normal distribution ($\lambda=2$) were compared against 85 alternatives for different sample sizes. One of these tests, denoted by $X_{\mathrm{APD}}$, is an omnibus test also based on second-power skewness and kurtosis and practically equivalent to our $X^{\mathrm{APD}}_{2}$ test. Overall, three normality tests stood out as the most powerful, with the $X_{\mathrm{APD}}$ test in second place just behind the \cite{doi:10.1080/00949659508811711} test based on normalized spacings and ahead of the famous \cite{MR205384} test.

In \cite{Laplace_tests_compilation}, we performed a comprehensive empirical power comparison of 40 goodness-of-fit tests for the univariate Laplace distribution -- including our new  $X^{\mathrm{APD}}_{1}$ and $Z(K^{\mathrm{net}}_{1})$ tests -- carried out using Monte Carlo simulations with sample sizes $n = 20, 50, 100, 200$, significance levels $\alpha = 0.01, 0.05, 0.10$, and 400 alternatives. The set of alternatives, formed of asymmetric and symmetric light/heavy-tailed distributions, consists of 20 specific cases of 20 submodels drawn from 11 main models. For each submodel, the 20 specific cases correspond to parameter values chosen to cover the entire power range.

We first identify in Table~\ref{table:overall.average.power.n20.to.200} the best omnibus tests against all 400 alternative distributions considered in the simulation study, whether symmetric/asymmetric light/heavy-tailed. The 10 most powerful tests (among the 40 candidates) are listed for each sample size and $\alpha=0.05$.  As an interpretational aid, we define the ``gap'' as the difference between the maximum average power amongst the 40 tests and the average power of a given test. We obtained four such gaps for each test, one for each sample size. The ``maximum gap'' and ``average gap'' for a given test are defined as the maximum and the average of those 4 gaps. To complete Table~\ref{table:overall.average.power.n20.to.200}, we have included the 10 best tests in terms of maximum and average gaps. We observe that the most powerful omnibus test, regardless of sample size, is our new $X^{\mathrm{APD}}_{1}$ test, with an average gap of 1.5\% and a maximum gap of 3\%. Note that the full name of the test abbreviations listed in Table~\ref{table:overall.average.power.n20.to.200} can be found in Table 2 in \cite{Laplace_tests_compilation}.

    \bgroup
    \def\arraystretch{1.0}
    \begin{table}[H]
    \scriptsize
        \caption{The average \% power of the 10 best performing tests (among 40) against all alternatives, as a function of sample size ($\alpha=0.05$).}\label{table:overall.average.power.n20.to.200}
        \vspace{-3mm}
        \begin{center}
            \setlength\tabcolsep{4.87pt} 
            \begin{tabular}{ccccccccccc}
            \hline
            $n = 20$ &   $\text{AP}_y$ &   $\text{AP}_e$ &   $\text{AP}_z$ &  $\text{AP}_v$ &   $\text{CK}_v$ &  $\text{AP}_y^{\normalfont \text{(MLE)}}$ &  $\text{Me}_2^{(1)}$ &  $\text{AP}_a$ &  Wa & $X^{\mathrm{APD}}_{1}$\\
            {\bf Power} &  47.0 &  46.4 &  46.2 & 45.3 & 45.2 & 44.9 & 44.7 & 44.5 & 44.2 & 44.0 \\
            {\bf Gap} &  - &  0.6 &  0.8 & 1.7 & 1.8 & 2.1 & 2.3 & 2.5 & 2.8 & 3.0 \\
            \hline
            $n = 50$ &   $\text{AP}_v$ &  $\text{AP}_e$ &  $\text{AP}_y^{\scriptscriptstyle (\text{MLE})}$ & $X^{\mathrm{APD}}_{1}$ &  $\text{AP}_y$ &  $\text{Me}_{0.5}^{\scriptscriptstyle (2)}$ &  $\text{Me}_2^{\scriptscriptstyle (1)}$ &  LK &   Wa &  $\text{CK}_v$ \\
            {\bf Power} &  63.7 & 62.6 & 62.3 & 60.8 & 60.4 & 60.4 & 60.1  & 59.0 & 58.9 & 58.8 \\
            {\bf Gap} &  - & 1.1 & 1.4 & 2.8 & 3.2 & 3.3 & 3.5 & 4.7 & 4.7 & 4.8 \\
            \hline
            $n = 100$ &  $X^{\mathrm{APD}}_{1}$ &  LK &  Wa &  $\text{AP}_v$ &  $\text{Me}_{0.5}^{\scriptscriptstyle (2)}$ &  $\text{Me}_2^{\scriptscriptstyle (1)}$ &  $\text{AP}_y^{\normalfont \text{(MLE)}}$ &  $\text{AP}_e$ &  $\text{AB}_{\text{He}}$ &  Ku \\
            {\bf Power} &  74.3 & 72.3 & 71.8 & 71.7 & 71.7 & 71.4 & 69.0 & 68.6 & 68.1 & 67.8 \\
            {\bf Gap} &  - & 1.9 & 2.4 & 2.5 & 2.6 & 2.8 & 5.2 & 5.7 & 6.2 & 6.5 \\
            \hline
            $n = 200$ &  $X^{\mathrm{APD}}_{1}$ &  Wa &  LK &  $\text{Me}_2^{\scriptscriptstyle (1)}$ &  $\text{Me}_{0.5}^{\scriptscriptstyle (2)}$ &  Ku &     $\text{AB}_{\text{He}}$ &  BS &  $\text{AB}_{\text{Je}}$ &  $\text{AP}_v$ \\
            {\bf Power} &  81.8 &  80.7 & 80.3 & 79.5 & 79.1 & 78.0 & 77.3 & 77.1 & 76.8 & 76.8 \\
            {\bf Gap} &  - & 1.1 & 1.5 & 2.3 & 2.7 & 3.8 & 4.5 & 4.7 & 5.0 & 5.0 \\
            \hline
            \hline
             {\bf Max} &  $X^{\mathrm{APD}}_{1}$ &   $\text{Me}_{0.5}^{(2)}$ &  $\text{Me}_2^{(1)}$ &    Wa &  LK &  $\text{AP}_v$ &  $\text{AB}_{\text{He}}$ &  $\text{AP}_y^{\normalfont \text{(MLE)}}$ &  $\text{Z}_A$ & $\text{AP}_e$ \\
            {\bf Gap} & 3.0 &  3.3 &  3.5 &  4.7 &  5.0 &  5.0 &   7.3 & 7.7 & 7.8 & 7.9 \\
            \hline
            {\bf Average} &  $X^{\mathrm{APD}}_{1}$ &  $\text{AP}_v$ &  $\text{Me}_2^{(1)}$ &   Wa &   $\text{Me}_{0.5}^{(2)}$ &  LK &  $\text{AP}_e$ &  $\text{AP}_y^{\normalfont \text{(MLE)}}$ &  $\text{AP}_y$ & Ku \\
            {\bf Gap} &   1.5 &  2.3 &  2.8 &  2.8 &  3.0 &  3.3 & 3.8 & 4.1 & 5.7 & 5.9 \\
            \hline
            \end{tabular}
        \end{center}
    \end{table}
    \egroup

Similarly, we identify in Table~\ref{table:symmetric.average.power.n20.to.200} the best tests against the 240 symmetric alternative distributions (from the 12 symmetric submodels) considered in the simulation study, whether light or heavy-tailed. The 10 most powerful tests (among the 40 candidates) are listed by sample size and in terms of maximum and average gaps, for $\alpha=0.05$.  We observe that our new $Z(K^{\mathrm{net}}_{1})$ test stands out as the best for sample sizes of $n=50,100,200$, and also regardless of sample size, with an average gap of 1.0\% (1st) and a maximum gap of 4.1\% (2nd). Although directional tests designed specifically to detect symmetric alternatives are favoured here, our new omnibus $X^{\mathrm{APD}}_{1}$ test performs well with an average gap of 3.0\% (4th) and a maximum gap of 4.0\% (1st).

    \bgroup
    \def\arraystretch{1.0}
    \begin{table}[H]
    \scriptsize
        \caption{The average \% power of the 10 best performing tests (among 40) against the symmetric alternatives, as a function of sample size ($\alpha=0.05$).}\label{table:symmetric.average.power.n20.to.200}
        \vspace{-3mm}
        \begin{center}
            \setlength\tabcolsep{4.87pt} 
            \begin{tabular}{ccccccccccc}
            \hline
            $n = 20$ &  Wa &  $\text{Me}_2^{(1)}$ &  $\text{AP}_z$ & $X^{\mathrm{APD}}_{1}$ &  $\text{Me}_{0.5}^{\scriptscriptstyle (2)}$ &  Ku &  $\text{AP}_y$ & $Z(K^{\mathrm{net}}_{1})$ &  $\text{Ho}_U$ &  $\text{AP}_e$\\
            {\bf Power} &  44.5 &  43.5 & 42.9 & 42.8 & 42.8 & 42.4 & 40.6 & 40.3 & 39.5 & 39.1 \\
            {\bf Gap} &  - &  1.0 & 1.6 & 1.6 & 1.7 & 2.0 & 3.9 & 4.1 & 5.0 & 5.4 \\
            \hline
            $n = 50$ &  $Z(K^{\mathrm{net}}_{1})$ &  $\text{Ho}_U$ & $X^{\mathrm{APD}}_{1}$ &  Wa &  GV &  $\text{Me}_{0.5}^{\scriptscriptstyle (2)}$ &  BS &  $ \text{Me}_2^{\scriptscriptstyle (1)}$ &  $\text{AP}_e$ &  $\text{AP}_v$ \\
            {\bf Power} &  61.9 & 60.1 & 58.2 & 58.2 & 58.1 & 57.5 & 57.4 & 57.3 & 56.6 & 56.5\\
            {\bf Gap} &  - & 1.8 & 3.7 & 3.7 & 3.8 & 4.3 & 4.5 & 4.5 & 5.3 & 5.4 \\
            \hline
            $n = 100$ &  $Z(K^{\mathrm{net}}_{1})$ &  $\text{Ho}_U$ &  GV &  BS & $X^{\mathrm{APD}}_{1}$ &  Wa &  $\text{Me}_{0.5}^{\scriptscriptstyle (2)}$ &  $\text{Me}_2^{\scriptscriptstyle (1)}$ &  LK &  Ku \\
            {\bf Power} &  73.4 & 71.4 & 70.5 & 70.5 & 69.3 & 68.7 & 67.3 & 67.1 & 65.9 & 65.6 \\
            {\bf Gap} &  - &  2.0 & 2.9 & 2.9 & 4.0 & 4.6 & 6.0 & 6.3 & 7.5 & 7.8 \\
            \hline
            $n = 200$ &  $Z(K^{\mathrm{net}}_{1})$ &   BS &  $\text{Ho}_U$ &  GV &  Wa & $X^{\mathrm{APD}}_{1}$ &  $\text{Me}_2^{\scriptscriptstyle (1)}$ &  $\text{Me}_{0.5}^{\scriptscriptstyle (2)}$ &  LK &  Ku \\
            {\bf Power} &  79.6 &  79.5 & 77.4 & 77.1 & 76.9 & 76.8 & 74.9 & 74.7 & 74.2 & 74.0 \\
            {\bf Gap} &  - &   0.1 & 2.2 & 2.5 & 2.7 & 2.8 & 4.7 & 4.9 & 5.4 & 5.6 \\
            \hline
            \hline
             {\bf Max} &  $X^{\mathrm{APD}}_{1}$ &  $Z(K^{\mathrm{net}}_{1})$ &    Wa &   $\text{Ho}_U$ &  GV &  $\text{Me}_{0.5}^{\scriptscriptstyle (2)}$ &  $\text{Me}_2^{\scriptscriptstyle (1)}$ & BS &  Ku  &  LK \\
            {\bf Gap} &  4.0 &  4.1 &  4.6 &  5.0 & 6.0 & 6.0 & 6.3 & 7.6 & 7.8 & 7.8 \\
            \hline
            {\bf Average} &  $Z(K^{\mathrm{net}}_{1})$ &   $\text{Ho}_U$ &   Wa & $X^{\mathrm{APD}}_{1}$   &  BS  &  GV &  $\text{Me}_2^{\scriptscriptstyle (1)}$ &  $\text{Me}_{0.5}^{\scriptscriptstyle (2)}$  &  Ku  &  LK \\
            {\bf Gap} &   1.0 &  2.7 &  2.8 & 3.0 & 3.8 & 3.8 & 4.1 & 4.2 & 5.5 & 6.7 \\
            \hline
            \end{tabular}
        \end{center}
    \end{table}
    \egroup

\vspace{-6mm}
\section{Conclusion}\label{sec:conclusion}

In this article, we introduced a family of goodness-of-fit tests for the $\mathrm{EPD}_{\lambda}(\mu,\sigma)$ distribution with $\lambda\geq 1$, including tests for the Laplace and Gaussian distributions. We obtained directional tests of fit against asymmetric ($Z(S_{\lambda})$) and symmetric ($Z(K^{\mathrm{net}}_{\lambda})$) alternatives, which we combined into an omnibus test ($X^{\mathrm{APD}}_{\lambda}\leqdef Z^2(S_{\lambda})+Z^2(K^{\mathrm{net}}_{\lambda})$). These tests are based on interesting new moment-type statistics called `$\lambda$-th-power skewness’, ‘$\lambda$-th-power kurtosis’ and ‘$\lambda$-th-power net kurtosis’. The new tests are very powerful and can be used as diagnostics to understand which aspects of the null hypothesis are rejected. Their null distribution is well approximated by the well-known chi-square or Gaussian distributions, for all sample sizes (up to very high precision for $n\geq 20$), which allow accurate and easy calculation of critical values and \textit{p}-values without the need to rely on simulated quantiles. We applied these tests on three sets of real temperature data and were able to demonstrate that a Laplace distribution or an $\mathrm{EPD}_{1.5}$ distribution is sometimes a better fit than a Gaussian distribution to model such measurements.

\section*{Funding}

F.\ Ouimet was supported by postdoctoral fellowships from the Natural Sciences and Engineering Research Council of Canada (PDF) and the Fond qu\'eb\'ecois de la recherche – Nature et technologies (B3X supplement and B3XR). F.\ Ouimet is currently supported by a CRM-Simons postdoctoral fellowship from the Centre de recherches math\'ematiques and the Simons foundation.

\section*{Acknowledgements}

We thank the anonymous referee for his/her comments.
This research includes computations performed using the computational cluster Katana supported by Research Technology Services at UNSW Sydney.
We thank the anonymous referee for his/her comments.

\section*{Disclosure statement}

No potential conflict of interest was reported by the authors.



%
%

\section*{References}


\appendix

\begin{appendices}

\setcounter{section}{0}

\section{Supplementary material (\texttt{R} codes)}\label{sec:supplementary.materialcode}

All the detailed programming codes using the \texttt{R} software are provided online at the address \href{https://doi.org/10.1080/02331888.2022.2144859}{https://doi.org/doi:10.1080/02331888.2022.2144859} in two \texttt{R} files (A1 and A2). We also provide a CSV file (A3) of the dataset analysed in Section~\ref{refrig}.

\section{Supplementary material (proofs)}\label{sec:supplementary.materialA}

In this section, we gather the proofs of the various results we stated in Sections 2 and 3 of the paper entitled ``Goodness-of-Fit Tests for Laplace, Gaussian and Exponential Power Distributions Based on $\lambda$-th Power Skewness and Kurtosis''. Throughout, the convergence in law and in probability, under a given measure $\PP$, will be denoted by $\stackrel{\PP}{\rightsquigarrow}$ and $\stackrel{\PP}{\rightarrow}$, respectively. A random term $\varepsilon$ going to $0$ in $\PP$-probability as $n\to \infty$ will be denoted by $o_{\PP}(\varepsilon)$. A random term $\beta$ bounded in $\PP$-probability as $n\to \infty$ will be denoted by $O_{\PP}(\beta)$.

\subsection{Proof of Equation (3)}\label{proof:sim.APD}

    Let
    \begin{equation}
        X = \mu + \sigma \big(\delta_{\hspace{-0.2mm}\theta_1\hspace{-0.5mm},\theta_2}^{-1}\lambda W\big)^{1/\theta_2}
        \big((1 - \theta_1)(1-V) - \theta_1  V\big),
    \end{equation}
    where $W\sim \text{Gamma}\hspace{0.2mm}(1/\theta_2,1)$ and $V\sim\text{Bernoulli}\hspace{0.2mm}(\theta_1)$ are independent.
    In order to conclude, we need to show that the c.d.f.\ of $X$ is equal to the c.d.f.\ of the $\mathrm{APD}_{\lambda}(\theta_1,\theta_2,\mu,\sigma)$ distribution, denoted by $F_{\lambda}(x \nvert \theta_1,\theta_2,\mu,\sigma)$ in (2).
    Let $Y \leqdef \sigma^{-1} (X - \mu)$ and $y \leqdef \sigma^{-1} (x - \mu)$.
    Now, if $y < 0$, then
    \begin{align}\label{eq:2.3.negative}
        \PP(X \leq x) = \PP(Y \leq y)
        &= \PP(Y \leq y \nvert V = 1) \cdot \PP(V = 1) + \PP(Y \leq y \nvert V = 0) \cdot \PP(V = 0) \notag \\
        &= \PP\bigg(\theta_1 \Big(\frac{\lambda W}{\delta_{\hspace{-0.2mm}\theta_1\hspace{-0.5mm},\theta_2}}\Big)^{1/\theta_2} \geq -y\bigg) \cdot \theta_1 + 0 \cdot (1 - \theta_1) \notag \\
        &= \theta_1 \left[1 - F_W\bigg(\frac{\delta_{\hspace{-0.2mm}\theta_1\hspace{-0.5mm},\theta_2}}{\lambda} \cdot \Big(\frac{-y}{\theta_1}\Big)^{\theta_2}\bigg)\right],
    \end{align}
    and, if $y \geq 0$, then
    \begin{align}\label{eq:2.3.positive}
        \PP(X \leq x) = \PP(Y \leq y)
        &= \PP(Y \leq y \nvert V = 1) \cdot \PP(V = 1) + \PP(Y \leq y \nvert V = 0) \cdot \PP(V = 0) \notag \\
        &= 1 \cdot \theta_1 + \PP\bigg((1-\theta_1) \Big(\frac{\lambda W}{\delta_{\hspace{-0.2mm}\theta_1\hspace{-0.5mm},\theta_2}}\Big)^{1/\theta_2} \leq y\bigg) \cdot (1 - \theta_1) \notag \\
        &= \theta_1 + (1 - \theta_1) \, F_W\bigg(\frac{\delta_{\hspace{-0.2mm}\theta_1\hspace{-0.5mm},\theta_2}}{\lambda} \cdot \Big(\frac{y}{1 - \theta_1}\Big)^{\theta_2}\bigg).
    \end{align}
    The right-hand sides of \eqref{eq:2.3.negative} and \eqref{eq:2.3.positive} are both equal to $F_{\lambda}(x \nvert \theta_1,\theta_2,\mu,\sigma)$ in (2).

\subsection{Proof of Remark 2.5}\label{proof:strong.consistency}

    In Lemma~\ref{lem:standard.ULLN} below, we state a small adaptation of a well-known uniform law of large numbers due to Lucien Le Cam.
    We will use it several times for different proofs in this supplementary material, including the proof of the next lemma (Lemma~\ref{lem:convergence.almost.sure.MLE.estimators}) regarding the strong consistency of the maximum likelihood estimators $\hat{\mu}_{\lambda}$ and $\hat{\sigma}_{\lambda}$, which we stated in Remark~2.5.
    The proof of Lemma~\ref{lem:standard.ULLN} follows the strategy described in Section 16 of \cite{MR1699953}.
    A small adaptation is needed to treat the case where the parameter space is not compact.

    \begin{lemma}\label{lem:standard.ULLN}
        Let $X_1,X_2,X_3,\dots$ be a sequence of i.i.d.\ random variables, and let $\hat{\bb{\xi}}_n \leqdef \hat{\bb{\xi}}_n(X_1,X_2,\dots,X_n)$ be an estimator such that $\hat{\bb{\xi}}_n \xrightarrow{\mathrm{a.s.}} \bb{\xi}\in \R^d$. For $\delta \geq 0$, let $B_{\delta}[\bb{\xi}] \leqdef \{\bb{t}\in \R^d : \|\bb{t} - \bb{\xi}\|_2 \leq \delta\}$. Assume that $U : \R \times \R^d \to \R$ is a measurable function and there exists $\delta > 0$ such that
        \begin{description}
            \item[(C.1)] For all $x\in \R$, $\bb{t}\mapsto U(x,\bb{t})$ is continuous on $B_{\delta}[\bb{\xi}]$;
            \item[(C.2)] There exists $K : \R \to \R$ such that $|U(x,\bb{t})| \leq K(x)$ for all $(x,\bb{t})\in \R \times B_{\delta}[\bb{\xi}]$ and $\EE\big[|K(X_1)|\big] < \infty$.
        \end{description}
        If $\rho_n \leqdef \|\hat{\bb{\xi}}_n - \bb{\xi}\|_2$ and $\overline{U}(\bb{t}) \leqdef \EE[U(X_1,\bb{t})]$, then
        \begin{equation}\label{eq:standard.ULLN}
            \PP\bigg(\limsup_{n\to\infty} \sup_{\bb{t} \in B_{\rho_n}[\bb{\xi}]} \Big|\frac{1}{n} \sum_{i=1}^n U(X_i,\bb{t}) - \overline{U}(\bb{\xi})\Big| > 0\bigg) = 0.
        \end{equation}
    \end{lemma}

    \begin{proof}[\bf Proof of Lemma~\ref{lem:standard.ULLN}]
        Fix $\delta > 0$ to a value for which $({\rm C.1})$ and $({\rm C.2})$ hold.
        By the triangle inequality, and since $\rho_n \xrightarrow{\mathrm{a.s.}} 0$ by assumption, we have
        \begin{align}\label{eq:lem:standard.ULLN.convergence.pivot}
            &\PP\bigg(\limsup_{n\to\infty} \sup_{\bb{t} \in B_{\rho_n}[\bb{\xi}]} \Big|\frac{1}{n} \sum_{i=1}^n U(X_i,\bb{t}) - \overline{U}(\bb{\xi})\Big| > 0\bigg) \notag \\[-0.5mm]
            &\leq \PP\bigg(\limsup_{n\to\infty} \sup_{\bb{t} \in B_{\rho_n}[\bb{\xi}]} \Big|\frac{1}{n} \sum_{i=1}^n U(X_i,\bb{t}) - \overline{U}(\bb{t})\Big| > 0\bigg) + \PP\bigg(\limsup_{n\to\infty} \sup_{\bb{t} \in B_{\rho_n}[\bb{\xi}]} \big|\overline{U}(\bb{t}) - \overline{U}(\bb{\xi})\big| > 0\bigg) \notag \\[-0.5mm]
            &\leq \PP\bigg(\limsup_{n\to\infty} \sup_{\bb{t} \in B_{\delta}[\bb{\xi}]} \Big|\frac{1}{n} \sum_{i=1}^n U(X_i,\bb{t}) - \overline{U}(\bb{t})\Big| > 0\bigg) + \PP\bigg(\limsup_{n\to\infty} \sup_{\bb{t} \in B_{\rho_n}[\bb{\xi}]} \big|\overline{U}(\bb{t}) - \overline{U}(\bb{\xi})\big| > 0\bigg).
        \end{align}

        \noindent
        By applying a uniform law of large numbers on the compact set $B_{\delta}[\bb{\xi}]$ (Theorem~16 (a) in \cite{MR1699953} with our assumptions $({\rm C.1})$ and $({\rm C.2})$), the first probability on the right-hand side of \eqref{eq:lem:standard.ULLN.convergence.pivot} is zero.
        By $({\rm C.1})$, $({\rm C.2})$ and the dominated convergence theorem, we know that $\overline{U}(\bb{t}) \leqdef \EE[U(X_1,\bb{t})]$ is continuous on $B_{\delta}[\bb{\xi}]$.
        Since $\rho_n\xrightarrow{\mathrm{a.s.}} 0$ by hypothesis, the second probability on the right-hand side of \eqref{eq:lem:standard.ULLN.convergence.pivot} is also zero.
    \end{proof}

    We can now prove the strong consistency of the maximum likelihood estimators.

    \begin{lemma}\label{lem:convergence.almost.sure.MLE.estimators}
        Let $\hat{\mu}_{\lambda}$ and $\hat{\sigma}_{\lambda}$ be defined as in Proposition~2.3.
        Assume that the observations $X_1,X_2,\dots,X_n$ are i.i.d.\ and $\mathrm{APD}_{\lambda}(\theta_1,\theta_2,\mu,\sigma)$ distributed.
        Then
        \begin{equation}\label{eq:lem:convergence.almost.sure.MLE.estimators.eq}
            \begin{pmatrix}
                \hat{\mu}_{\lambda} \\[1mm]
                \hat{\sigma}_{\lambda}
            \end{pmatrix}
            \xrightarrow{\mathrm{a.s.}}
            \begin{pmatrix}
                \mu \\[1mm]
                \sigma
            \end{pmatrix},
            \quad \text{as } n\to\infty.
        \end{equation}
        In other words, the above convergence holds under both $H_0$ and $H_1$.
    \end{lemma}

    \begin{proof}[\bf Proof of Lemma~\ref{lem:convergence.almost.sure.MLE.estimators}]
        By definition, for all $\lambda\geq 1$, the estimator $\hat{\mu}_{\lambda}$ is determined by the equation
        \begin{equation}\label{eq:MLE.system.mu.critical.points}
            \sum_{i=1}^n w(X_i,\hat{\mu}_{\lambda}) = 0, \quad \text{where } w(x,\mu) \leqdef |x - \mu|^{\lambda - 1} \mathrm{sign}(x - \mu).
        \end{equation}
        For any $x\in \R$, $w(x,\cdot)$ is non-increasing when $\lambda \geq 1$. From Theorem~2 and Remark 1 in \cite{MR709248} (the proof is a simple application of Chernoff's theorem), we get that, for any $\e > 0$, the probabilities $\PP(|\hat{\mu}_{\lambda} - \mu| > \e)$ decay exponentially fast in $n$ (using the fact that $\EE[w(X_1,\mu + \e)] < 0$ and $\EE[w(X_1,\mu - \e)] > 0$ both hold).
        In particular, for any $\e > 0$, the probabilities are summable in $n$.
        Hence, by the Borel-Cantelli lemma, we have  $\hat{\mu}_{\lambda} \rightarrow \mu$ a.s.

        Also, from Proposition~2.3, we have
        \begin{equation}
           \hat{\sigma}_{\lambda}^{\lambda} = \frac{1}{n} \sum_{i=1}^n |X_i - \hat{\mu}_{\lambda}|^{\lambda}.
        \end{equation}
        If we denote $U(x,t) \leqdef |x - t|^{\lambda}$ and $\overline{U}(t) \leqdef \EE\big[U(X_1,t)\big]$, then it is straightforward to verify that $\overline{U}(\mu) =  \sigma^{\lambda}$.
        From Lemma~\ref{lem:standard.ULLN}, we deduce
        \begin{equation}\label{eq:lem:convergence.probability.MLE.estimators.to.show}
            \PP\bigg(\lim_{n\to\infty} \Big|\frac{1}{n} \sum_{i=1}^n U(X_i,\hat{\mu}_{\lambda}) - \overline{U}(\mu)\Big| = 0\bigg) = 1.
        \end{equation}
        This implies $\hat{\sigma}_{\lambda} \rightarrow \sigma$ a.s.
    \end{proof}

    \begin{remark}
        If $n$ is odd or if $n$ is even with $X_{(n/2)}=X_{(n/2+1)}$, then we have the right-continuity of $\hat{\mu}_{\lambda}$ around $\lambda=1$, namely $\lim_{\lambda\searrow 1}\hat{\mu}_{\lambda}=\mathrm{median}(X_1,X_2,\dots,X_n)$. Otherwise, if $n$ is even, we have more generally $\lim_{\lambda\searrow 1}\hat{\mu}_{\lambda}\in (X_{(n/2)},X_{(n/2+1)})$, and thus  $\lim_{n\to \infty}(\lim_{\lambda\searrow 1}\hat{\mu}_{\lambda}-\mathrm{median}(X_1,X_2,\dots,X_n))=0$.
    \end{remark}

\subsection{Proof of Theorem~3.3}\label{proof:thm.H0}

    For short, write
    \begin{equation}
        \bb{\kappa}\leqdef
            \begin{pmatrix}
                \mu \\[1mm]
                \sigma
            \end{pmatrix},
        \qquad
        \hat{\bb{\kappa}}_n \leqdef
            \begin{pmatrix}
                \hat{\mu}_{\lambda} \\[1mm]
                \hat{\sigma}_{\lambda}
            \end{pmatrix},
        \qquad
            \bb{\theta}\leqdef
            \begin{pmatrix}
                \theta_1 \\[1mm]
                \theta_2
            \end{pmatrix}
        \quad \text{and} \quad
            \bb{\theta}_0\leqdef
            \begin{pmatrix}
                1/2 \\[1mm]
                \lambda
            \end{pmatrix}.
    \end{equation}
    If $y\leqdef \sigma^{-1}(x-\mu)$, we define
    \begin{align}
        \bb{d}_{\bb{\theta}}(y)
        &\leqdef \frac{\partial}{\partial \bb{\theta}} \log f_{\lambda}(x \nvert \bb{\theta},\bb{\kappa})\big|_{\bb{\theta}=\bb{\theta}_0} = \frac{\partial}{\partial \bb{\theta}} \log f_{\lambda}(y \nvert \bb{\theta},(0,1)^{\top})\big|_{\bb{\theta}=\bb{\theta}_0}, \label{eq:d.theta} \\[2mm]
        \bb{d}_{\bb{\kappa}}(y)
        &\leqdef \sigma \frac{\partial}{\partial \bb{\kappa}} \log f_{\lambda}(x \nvert \bb{\theta}_0,\bb{\kappa}) =
            \begin{pmatrix}
                -\frac{\partial}{\partial y} \log f_{\lambda}(y \nvert \bb{\theta}_0,(0,1)^{\top}) \\[2mm]
                -1 - y\frac{\partial}{\partial y} \log f_{\lambda}(y \nvert \bb{\theta}_0,(0,1)^{\top})
            \end{pmatrix}. \label{eq:d.kappa}
    \end{align}
    We can easily verify (using for example \texttt{Wolfram Mathematica}) that
    \begin{equation}\label{eq:vector.d}
        \begin{aligned}
            &\bb{d}_{\bb{\theta}}(y) =
            \begin{pmatrix}
                -2 |y|^{\lambda} \mathrm{sign}(y) \\[2mm]
                -\frac{1}{\lambda} \Big[|y|^{\lambda} \log|y| - \frac{1}{\lambda} \big(\lambda+\log \lambda + \psi(1/\lambda)\big)\Big]
            \end{pmatrix}, \qquad \bb{d}_{\bb{\kappa}}(y) =
            \begin{pmatrix}
                 |y|^{\lambda - 1} \mathrm{sign}(y) \\[2mm]
                 |y|^{\lambda} - 1
            \end{pmatrix},
        \end{aligned}
    \end{equation}
    where recall that $\psi(z) \leqdef \frac{\rd}{\rd z} \log \Gamma(z)$ denotes the digamma function, and $\Gamma(z) \leqdef \int_0^{\infty} t^{z-1} e^{-t} \rd t$ for $z > 0$.

    Using the notation in \eqref{eq:d.theta}, we can write the Rao's score statistic (see (6))
    \begin{equation}\label{eq:score.statistic.known.2}
        \bb{r}_n(\bb{\kappa}) \leqdef \frac{1}{n} \sum_{i=1}^n \frac{\partial}{\partial \bb{\theta}} \log f_{\lambda}(X_i \nvert \bb{\theta},\bb{\kappa})\big|_{\bb{\theta}=\bb{\theta}_0} \quad \text{ as } \quad \bb{r}_n(\bb{\kappa}) = \frac{1}{n} \sum_{i=1}^n \bb{d}_{\bb{\theta}}\bigg(\frac{X_i - \mu}{\sigma}\bigg),
    \end{equation}
    and the {\it modified score statistic}
    \begin{equation}\label{eq:score.statistic.unknown.2}
        \bb{r}_n(\hat{\bb{\kappa}}_n) = \frac{1}{n} \sum_{i=1}^n \frac{\partial}{\partial \bb{\theta}} \log f_{\lambda}(X_i \nvert \bb{\theta},\hat{\bb{\kappa}}_n)\big|_{\bb{\theta}=\bb{\theta}_0} \quad \text{ as } \quad \bb{r}_n(\hat{\bb{\kappa}}_n) = \frac{1}{n} \sum_{i=1}^n \bb{d}_{\bb{\theta}}\bigg(\frac{X_i - \hat{\mu}_{\lambda}}{\hat{\sigma}_{\lambda}}\bigg).
    \end{equation}
    Note that $\bb{r}_n(\hat{\bb{\kappa}}_n)$ is location and scale invariant, or said otherwise, $\bb{\kappa}$-invariant.

    The first step of the proof of Theorem~3.3 consists in determining the asymptotic law of the vector
    \begin{equation}\label{eq:asymp.normality.sum.H0}
        \frac{1}{\sqrt{n}} \sum_{i=1}^n
            \begin{pmatrix}
                \bb{d}_{\bb{\theta}}(Y_i) \\[1mm]
                \bb{d}_{\bb{\kappa}}(Y_i)
            \end{pmatrix}, \quad \text{where } ~Y_i \leqdef \sigma^{-1} (X_i - \mu),
    \end{equation}
    under $H_0$ (see Proposition~\ref{prop:asymptotic.normality.vector.d} below). The proof is a direct application of the central limit theorem.
    The second step consists in writing $\bb{r}_n(\hat{\bb{\kappa}}_n)$ as a linear combination of the components of this vector plus a negligible term via a first-order Taylor expansion (see Proposition~\ref{prop:MLE.score.asymptotics.Taylor}).
    The estimation of the derivative part of the expansion is dealt with in Proposition~\ref{prop:MLE.score.asymptotics.explicit}.
    Using these three propositions (which will be proved in Section~\ref{sec:main.result.1.proofs}), we will then be able to deduce the asymptotic distribution of $n^{1/2} \bb{r}_n(\hat{\bb{\kappa}}_n)$ under $H_0$.

    \begin{proposition}\label{prop:asymptotic.normality.vector.d}
        We have, as $n\to \infty$,
        \begin{equation}\label{eq:asymptotic.normality.vector.d}
            \frac{1}{\sqrt{n}} \sum_{i=1}^n
            \begin{pmatrix}
                \bb{d}_{\bb{\theta}}(Y_i) \\[1mm]
                \bb{d}_{\bb{\kappa}}(Y_i)
            \end{pmatrix}
            \stackrel{\PP_{H_0}}{\scalebox{2}[1.2]{$\rightsquigarrow$}}
            \mathcal{N}_4\left(\bb{0}, J \leqdef
            \begin{pmatrix}
                J_{\bb{\theta}\bb{\theta}} &J_{\bb{\theta}\bb{\kappa}} \\[1mm]
                J_{\bb{\theta}\bb{\kappa}}^{\top} &J_{\bb{\kappa}\bb{\kappa}}
            \end{pmatrix}
            \right),
        \end{equation}
        where $\bb{d}_{\bb{\theta}}$ and $\bb{d}_{\bb{\kappa}}$ are given in \eqref{eq:vector.d} and the covariance matrix $J$ is composed of $J_{\bb{\theta}\bb{\theta}}=\EE[\bb{d}_{\bb{\theta}}(Y) \bb{d}_{\bb{\theta}}(Y)^{\top}]$, $J_{\bb{\kappa}\bb{\kappa}}=\EE[\bb{d}_{\bb{\kappa}}(Y)\bb{d}_{\bb{\kappa}}(Y)^{\top}]$ and
        $J_{\bb{\theta}\bb{\kappa}}=\EE[\bb{d}_{\bb{\theta}}(Y)\bb{d}_{\bb{\kappa}}(Y)^{\top}]$, with
        \begin{equation}\label{eq:asymptotic.normality.vector.d.matrix}
            \begin{aligned}
            &J_{\bb{\theta}\bb{\theta}}=
                \begin{pmatrix}
                    4(1 + \lambda) &0  \\[1mm]
                    0 &\frac{(1+1/\lambda)\psi_1(1+1/\lambda)+\phi^2-1}{\lambda^3}
                \end{pmatrix}, \quad
            J_{\bb{\kappa}\bb{\kappa}}=
                \begin{pmatrix}
                    \frac{\lambda^{2-2/\lambda}\Gamma(2-1/\lambda)}{\Gamma(1/\lambda)} &0  \\[1mm]
                    0 &\lambda
                \end{pmatrix}
                , \quad
            J_{\bb{\theta}\bb{\kappa}}=
                \begin{pmatrix}
                   -\frac{2\lambda^{2-1/\lambda}}{\Gamma(1/\lambda)} &0  \\[1mm]
                    0 &-\frac{\phi}{\lambda}
                \end{pmatrix},
            \end{aligned}
        \end{equation}
        where $\phi\leqdef 1+\lambda+\log \lambda + \psi(1/\lambda)$ and $\psi_1(z)\leqdef \frac{\rd}{\rd z}\psi(z)$ is the trigamma function.
    \end{proposition}

    \begin{proposition}\label{prop:MLE.score.asymptotics.Taylor}
        We have, as $n\to \infty$,
        \begin{equation}\label{eq:prop:MLE.score.asymptotics.Taylor}
            n^{1/2} \bb{r}_n(\hat{\bb{\kappa}}_n) = n^{1/2} \bb{r}_n(\bb{\kappa}) + \bb{r}_n'(\bb{\kappa}) \, n^{1/2} (\hat{\bb{\kappa}}_n - \bb{\kappa}) + o_{\hspace{0.3mm}\PP_{H_0}}(1) \bb{1}_2,
        \end{equation}
        where $\bb{1}_2 \leqdef (1,1)^{\top}$ and $\bb{r}_n'(\bb{\kappa}) \leqdef \big(\frac{\partial}{\partial\mu}\bb{r}_n(\bb{\kappa}), \frac{\partial}{\partial\sigma}\bb{r}_n(\bb{\kappa})\big)$.
    \end{proposition}

    Now, we study the term $\bb{r}_n'(\bb{\kappa})\, n^{1/2} (\hat{\bb{\kappa}}_n - \bb{\kappa})$ and rewrite \eqref{eq:prop:MLE.score.asymptotics.Taylor}.

    \begin{proposition}\label{prop:MLE.score.asymptotics.explicit}
        Recall $J_{\bb{\theta}\bb{\kappa}}$ and $J_{\bb{\kappa}\bb{\kappa}}$ from Proposition~\ref{prop:asymptotic.normality.vector.d}. Then, as $n\to \infty$,
        \begin{align}
            \bb{r}_n'(\bb{\kappa})
            &= -\sigma^{-1} J_{\bb{\theta}\bb{\kappa}} + o_{\hspace{0.3mm}\PP_{H_0}}(1) \bb{1}_2^{\phantom{\top}}\hspace{-1mm}\bb{1}_2^{\top}, \label{eq:prop:MLE.score.asymptotics.explicit.r.prime} \\
            n^{1/2} (\hat{\bb{\kappa}}_n - \bb{\kappa})
            &= \sigma J_{\bb{\kappa}\bb{\kappa}}^{-1} \frac{1}{\sqrt{n}} \sum_{i=1}^n \bb{d}_{\bb{\kappa}}(Y_i) + o_{\hspace{0.3mm}\PP_{H_0}}(1) \bb{1}_2. \label{eq:prop:MLE.score.asymptotics.explicit.MLE}
        \end{align}
        Furthermore,
        \begin{equation}\label{eq:prop:MLE.score.asymptotics.explicit.score}
            n^{1/2} \bb{r}_n(\hat{\bb{\kappa}}_n)
            =
            \left(I_2 \, ; \, - J_{\bb{\theta}\bb{\kappa}} J_{\bb{\kappa}\bb{\kappa}}^{-1}\right) \frac{1}{\sqrt{n}} \sum_{i=1}^n
            \begin{pmatrix}
                \bb{d}_{\bb{\theta}}(Y_i) \\[1mm]
                \bb{d}_{\bb{\kappa}}(Y_i)
            \end{pmatrix}
            + o_{\hspace{0.3mm}\PP_{H_0}}(1)\bb{1}_2.
        \end{equation}
    \end{proposition}

    By combining Proposition~\ref{prop:asymptotic.normality.vector.d} and Proposition~\ref{prop:MLE.score.asymptotics.explicit}, we see that
    \begin{equation}\label{eq:asymp.r.n.under.H0}
        n^{1/2} \bb{r}_n(\hat{\bb{\kappa}}_n) \stackrel{\PP_{H_0}}{\scalebox{2}[1.2]{$\rightsquigarrow$}} \mathcal{N}_2(\bb{0}_2, \Sigma), \quad \text{as } n\to \infty,
    \end{equation}
    where the asymptotic covariance matrix $\Sigma$ is given by:
    \begin{align}\label{eq:calculation.Sigma}
        \Sigma
        &\stackrel{\phantom{\eqref{eq:asymptotic.normality.vector.d}}}{=} \left(I_2 \, ; \, - J_{\bb{\theta}\bb{\kappa}} J_{\bb{\kappa}\bb{\kappa}}^{-1}\right)
        \begin{pmatrix}
            J_{\bb{\theta}\bb{\theta}} &J_{\bb{\theta}\bb{\kappa}} \\[1mm]
            J_{\bb{\theta}\bb{\kappa}}^{\top} &J_{\bb{\kappa}\bb{\kappa}}
        \end{pmatrix}
        \begin{pmatrix}
            I_2 \\[1mm]
            -J_{\bb{\theta}\bb{\kappa}} J_{\bb{\kappa}\bb{\kappa}}^{-1}
        \end{pmatrix}
        \stackrel{\phantom{\eqref{eq:asymptotic.normality.vector.d}}}{=}
        J_{\bb{\theta}\bb{\theta}} - J_{\bb{\theta}\bb{\kappa}}J_{\bb{\kappa}\bb{\kappa}}^{-1} J_{\bb{\theta}\bb{\kappa}}^{\top} \notag \\
        &\stackrel{\eqref{eq:asymptotic.normality.vector.d}}{=}
        \begin{pmatrix}
            4(1 + \lambda) &\hspace{-1.5mm}0 \\
            0 &\hspace{-1.5mm}\frac{(1+1/\lambda)\psi_1(1+1/\lambda)+\phi^2-1}{\lambda^3}
        \end{pmatrix}
        -
        \begin{pmatrix}
            -\frac{2\lambda^{2-1/\lambda}}{\Gamma(1/\lambda)} &\hspace{-1.5mm}0 \\
            0 &\hspace{-1.5mm}-\frac{\phi}{\lambda}
        \end{pmatrix}
        \begin{pmatrix}
            \frac{\Gamma(1/\lambda)}{\lambda^{2-2/\lambda}\Gamma(2-1/\lambda)} &\hspace{-1.5mm}0 \\
            0 &\hspace{-1.5mm}\frac{1}{\lambda}
        \end{pmatrix}
        \begin{pmatrix}
            -\frac{2\lambda^{2-1/\lambda}}{\Gamma(1/\lambda)} &\hspace{-1.5mm}0 \\
            0 &\hspace{-1.5mm}-\frac{\phi}{\lambda}
        \end{pmatrix} \notag \\[1mm]
        &\stackrel{\phantom{\eqref{eq:asymptotic.normality.vector.d}}}{=}
        \begin{pmatrix}
            4(1 + \lambda) - \frac{4\lambda^2}{\Gamma(2 - 1/\lambda) \Gamma(1/\lambda)} &0 \\
            0 &\frac{(1+1/\lambda)\psi_1(1+1/\lambda)-1}{\lambda^3}
        \end{pmatrix}.
    \end{align}
    Given \eqref{eq:vector.d} and \eqref{eq:score.statistic.unknown.2}, we deduce from \eqref{eq:asymp.r.n.under.H0} and \eqref{eq:calculation.Sigma} that, as $n\to \infty$,
    \begin{equation}\label{eqn:law.B.K.H0}
        n^{1/2}
        \begin{pmatrix}
            S_{\lambda}(\Xn)\\[1mm]
            K_{\lambda}(\Xn) - \frac{\lambda+\log \lambda + \psi(1/\lambda)}{\lambda}
        \end{pmatrix}
        \stackrel{\PP_{H_0}}{\scalebox{2}[1.2]{$\rightsquigarrow$}} \mathcal{N}_2
        \left(\bb{0},
            \begin{pmatrix}
                1 + \lambda - \frac{\lambda^2}{\Gamma(2-1/\lambda) \Gamma(1/\lambda)} & 0 \\[1mm]
                0 & \frac{(1+1/\lambda) \psi_1(1+1/\lambda)-1}{\lambda}
            \end{pmatrix}
        \right).
    \end{equation}
    Assuming that we have proofs for Propositions~\ref{prop:asymptotic.normality.vector.d},~\ref{prop:MLE.score.asymptotics.Taylor}~and~\ref{prop:MLE.score.asymptotics.explicit} (see Section~\ref{sec:main.result.1.proofs} below), this ends the proof of Theorem~3.3.

\subsection{Proofs of Propositions~\ref{prop:asymptotic.normality.vector.d},~\ref{prop:MLE.score.asymptotics.Taylor}~and~\ref{prop:MLE.score.asymptotics.explicit} to complete the proof of Theorem~3.3}\label{sec:main.result.1.proofs}

    \begin{proof}[\bf Proof of Proposition~\ref{prop:asymptotic.normality.vector.d}]
        The asymptotic normality in \eqref{eq:asymptotic.normality.vector.d} is a direct consequence of the central limit theorem.
        Let $X\sim \mathrm{APD}_{\lambda}(\bb{\theta}_0,\bb{\kappa})$ and $Y \leqdef \sigma^{-1} (X - \mu)$.
        In order to conclude the proof, we show below how to compute the covariances between $d_{\theta_1}(Y)$, $d_{\theta_2}(Y)$, $d_{\mu}(Y)$ and $d_{\sigma}(Y)$.
        Before that, we gather some facts. For any given $\lambda\geq 1$, the density function of $Y$ is
        \begin{equation}\label{eq:prop:asymptotic.normality.vector.d.density.calculs}
            f(y \nvert \bb{\theta}_0,(0,1)^{\top}) = \frac{e^{-\frac{1}{\lambda} |y|^{\lambda}}}{2 \lambda^{1/\lambda} \Gamma(1+1/\lambda)}
            = \frac{e^{-\frac{1}{\lambda} |y|^{\lambda}}}{2 \lambda^{1/\lambda-1} \Gamma(1/\lambda)}, \quad y\in \R.
        \end{equation}
        Recall the definitions of the gamma, digamma and trigamma functions (for $z > 0$):
        \begin{equation}\label{eq:property.gamma.definition}
            \Gamma(z) \leqdef \int_0^{\infty} t^{z - 1} e^{-t} \rd t, \quad \psi(z) \leqdef \frac{\rd}{\rd z} \log \Gamma(z) = \frac{\Gamma'(z)}{\Gamma(z)}
            \quad \text{and} \quad \psi_1(z) \leqdef \frac{\rd}{\rd z} \psi(z),
        \end{equation}
        and some well-known properties they satisfy (see, e.g., \cite[Chapter 6]{MR0167642}):
        \begin{align}
            \Gamma(1 + z) &= z \Gamma(z), \label{eq:property.gamma.factorial} \\[1.5mm]
            \psi(1 + z) &= \psi(z) + \frac{1}{z}, \label{eq:property.digamma} \\[0.5mm]
            \psi_1(1 + z) &= \psi_1(z) - \frac{1}{z^2}, \label{eq:property.trigamma} \\[0.5mm]
            \int_0^{\infty} t^{z - 1} (\log t) e^{-t} \rd t &= \Gamma(z)\psi(z), \label{eq:property.gamma.integral.log.1} \\
            \int_0^{\infty} t^{z - 1} (\log t)^2 e^{-t} \rd t &= \Gamma(z)(\psi_1(z) + \psi^2(z)). \label{eq:property.gamma.integral.log.2}
        \end{align}
        Using these properties, we can easily verify that, for $k > -1/\lambda$ and $U = |Y|^{\lambda}/\lambda\sim\mathrm{Gamma}(1/\lambda,1)$,
        \begin{align}
            \EE\big[U^k\big]&=\frac{\Gamma(k+1/\lambda)}{\Gamma(1/\lambda)},\label{eq:exp.u}\\[0.5mm]
            \EE\big[U^k\log U\big]&=\frac{\Gamma(k+1/\lambda)\psi(k+1/\lambda)}{\Gamma(1/\lambda)},\label{eq:exp.u.log}\\[0.5mm]
            \EE\big[U^k(\log U)^2\big]&=\frac{\Gamma(k+1/\lambda)}{\Gamma(1/\lambda)}\big[\psi_1(k+1/\lambda)+\psi^2(k+1/\lambda)\big].\label{eq:exp.u.log2}
        \end{align}
        We obtain, for $a>-1$,
        \vspace{-2mm}
        \begin{align}
            \EE\big[|Y|^a\big]&=\lambda^{a/\lambda}\EE\big[(|Y|^{\lambda}/\lambda)^{a/\lambda}\big]
            =\lambda^{a/\lambda}\EE\big[U^{a/\lambda}\big] \stackrel{\eqref{eq:exp.u}}{=}\frac{\lambda^{a/\lambda}\Gamma((a+1)/\lambda)}{\Gamma(1/\lambda)},\label{eq:expectation.ya}\\
            \EE\big[|Y|^a\log|Y|\big]&=\lambda^{a/\lambda-1}\EE\big[(|Y|^{\lambda}/\lambda)^{a/\lambda}(\log(|Y|^{\lambda}/\lambda)+\log\lambda)\big]\notag\\
            &=
            \lambda^{a/\lambda-1}\left(\EE\big[U^{a/\lambda}\log U\big]+\EE\big[U^{a/\lambda}\big]\log\lambda\right)\notag\\
            &\hspace{-5mm}\stackrel{\eqref{eq:exp.u},\eqref{eq:exp.u.log}}{=}\lambda^{a/\lambda-1}\left(\frac{\Gamma((a+1)/\lambda)\psi((a+1)/\lambda)}{\Gamma(1/\lambda)}+\frac{\Gamma((a+1)/\lambda))\log\lambda}{\Gamma(1/\lambda)}\right)\notag\\
            &=\frac{\lambda^{a/\lambda-1}\Gamma((a+1)/\lambda)(\psi((a+1)/\lambda)+\log\lambda)}{\Gamma(1/\lambda)},\label{eq:expectation.ya.log}\\[2mm]
            \EE\big[|Y|^a(\log|Y|)^2\big]&=\lambda^{a/\lambda-2}\EE\big[(|Y|^{\lambda}/\lambda)^{a/\lambda}(\log(|Y|^{\lambda}/\lambda)+\log\lambda)^2\big]\notag\\
            &=
            \lambda^{a/\lambda-2}\left(\EE\big[U^{a/\lambda}(\log U)^2\big]+\EE\big[U^{a/\lambda}\big](\log\lambda)^2+2\,\EE\big[U^{a/\lambda}\log U\big]\log\lambda\right)\notag\\
            &\hspace{-7mm}\stackrel{\eqref{eq:exp.u},\eqref{eq:exp.u.log},\eqref{eq:exp.u.log2}}{=}\lambda^{a/\lambda-2}\left(
            \frac{\Gamma((a+1)/\lambda)}{\Gamma(1/\lambda)}\big[\psi_1((a+1)/\lambda)+\psi^2((a+1)/\lambda)\big]\right.\notag\\
            &\left.\hspace{20mm}+\frac{\Gamma((a+1)/\lambda)(\log\lambda)^2}{\Gamma(1/\lambda)}+\frac{2\,\Gamma((a+1)/\lambda)\psi((a+1)/\lambda)\log\lambda}{\Gamma(1/\lambda)}\right)\notag\\
            &=\frac{\lambda^{a/\lambda-2}\Gamma((a+1)/\lambda)}{\Gamma(1/\lambda)}
            \left(\psi_1((a+1)/\lambda)+[\psi((a+1)/\lambda)+\log\lambda]^2\right).\label{eq:expectation.ya.log2}
        \end{align}

        By symmetry of the density $f(\,\cdot\, |\,\bb{\theta}_0,(0,1)^{\top})$ with respect to 0 and anti-symmetry of the integrands, we have
        \begin{equation}
            J_{\theta_1 \theta_2} = J_{\theta_1 \sigma} = J_{\theta_2 \mu} = J_{\mu \sigma} = \boxed{0}.
        \end{equation}
        If we define $\nu\leqdef \lambda+\log \lambda + \psi(1/\lambda)$, we have $\phi=1+\nu$. Here is how we compute the other covariances:
        \begin{align}
            \hspace{-4.1mm}
            J_{\theta_1 \theta_1}
            &= \EE[d_{\theta_1}(Y) d_{\theta_1}(Y)] \stackrel{\eqref{eq:vector.d}}{=} 4\, \EE[|Y|^{2\lambda}]
            \stackrel{\eqref{eq:expectation.ya}}{=}\frac{4\lambda^{2}\Gamma(2+1/\lambda)}{\Gamma(1/\lambda)}
            \stackrel{\eqref{eq:property.gamma.factorial}}{=}4\lambda^{2}(1+1/\lambda)(1/\lambda)=\boxed{4(1+\lambda)},\label{eq:J.theta1.theta1}\\
            \hspace{-4.1mm}
            J_{\theta_1 \mu}
            &= \EE[d_{\theta_1}(Y) d_{\mu}(Y)] \stackrel{\eqref{eq:vector.d}}{=} -2 \,\EE[|Y|^{2\lambda - 1}]
            \stackrel{\eqref{eq:expectation.ya}}{=}\frac{-2\lambda^{2-1/\lambda}\Gamma(2)}{\Gamma(1/\lambda)}=\boxed{\frac{-2\lambda^{2-1/\lambda}}{\Gamma(1/\lambda)}},
            \label{eq:J.theta1.mu}\\
            \hspace{-4.1mm}
            J_{\mu \mu}
            &= \EE[d_{\mu}(Y) d_{\mu}(Y)] \stackrel{\eqref{eq:vector.d}}{=}\EE\big[|Y|^{2\lambda - 2}\big]\stackrel{\eqref{eq:expectation.ya}}{=}
            \boxed{\frac{\lambda^{2-2/\lambda}\Gamma(2-1/\lambda)}{\Gamma(1/\lambda)}},\label{eq:J.mu.muq}\\
            \hspace{-4.1mm}
            J_{\sigma \sigma}
            &= \EE[d_{\sigma}(Y) d_{\sigma}(Y)] \stackrel{\eqref{eq:vector.d}}{=}\EE\big[(|Y|^{\lambda}-1)^2\big]=\EE\big[|Y|^{2\lambda}\big]+1-2\,\EE\big[|Y|^{\lambda}\big]\notag\\
            &\stackrel{\eqref{eq:J.theta1.theta1},\eqref{eq:expectation.ya}}{=}
            (1+\lambda)+1-\frac{2\lambda\,\Gamma(1+1/\lambda)}{\Gamma(1/\lambda)}
            \stackrel{\eqref{eq:property.gamma.factorial}}{=}2+\lambda-2\lambda(1/\lambda)=\boxed{\lambda},\label{eq:J.sigma.sigma}\\
            \hspace{-4.1mm}
            J_{\theta_2 \sigma}
            &= \EE[d_{\theta_2}(Y) d_{\sigma}(Y)]  \stackrel{\eqref{eq:vector.d}}{=}\lambda^{-1}\EE\big[|Y|^{\lambda} \log|Y|\big]-
            \lambda^{-1}\EE\big[|Y|^{2\lambda} \log|Y|\big]+\lambda^{-2}\nu\EE\big[|Y|^{\lambda}\big]-\lambda^{-2}\nu\notag\\
            &\stackrel{\eqref{eq:expectation.ya},\eqref{eq:expectation.ya.log}}{=}\frac{\Gamma(1+1/\lambda)\lambda^{-1}\nu}{\Gamma(1/\lambda)}
            - \frac{\Gamma(2+1/\lambda)(\psi(2+1/\lambda)+\log\lambda)}{\Gamma(1/\lambda)}  +  \frac{\Gamma(1+1/\lambda)\lambda^{-1}\nu}{\Gamma(1/\lambda)}-\lambda^{-2}\nu\notag\\
            &\stackrel{\eqref{eq:property.gamma.factorial},\eqref{eq:property.digamma}}{=}\lambda^{-2}\nu - (1+1/\lambda)\lambda^{-1}(\psi(1+1/\lambda)+(1+1/\lambda)^{-1}+\log\lambda)  + \lambda^{-2}\nu - \lambda^{-2}\nu\notag\\
            &=\lambda^{-2}\nu - (1+1/\lambda)\lambda^{-1}(\nu+(1+1/\lambda)^{-1})
            =\lambda^{-2}\nu - (1+1/\lambda)\lambda^{-1}\nu - \lambda^{-1}   \notag\\
            &=- \lambda^{-1}(-\lambda^{-1}\nu + (1+1/\lambda)\nu + 1)=- \lambda^{-1}(1+\nu)=\boxed{-\phi/\lambda},\label{eq:J.theta2.sigma}\\
            \hspace{-4.1mm}
            J_{\theta_2 \theta_2}
            & = \EE[d_{\theta_2}(Y) d_{\theta_2}(Y)] \stackrel{\eqref{eq:vector.d}}{=} \lambda^{-2}\EE\big[|Y|^{2\lambda} (\log|Y|)^2\big]-2\lambda^{-3}\nu\EE\big[|Y|^{\lambda} \log|Y|\big]+\lambda^{-4}\nu^2\notag\\
            &\stackrel{\eqref{eq:expectation.ya.log},\eqref{eq:expectation.ya.log2}}{=}\frac{\lambda^{-2}\Gamma(2+1/\lambda)}{\Gamma(1/\lambda)}
            \left(\psi_1(2+1/\lambda)+[\psi(2+1/\lambda)+\log\lambda]^2\right)-\frac{2\Gamma(1+1/\lambda)\lambda^{-3}\nu^2}{\Gamma(1/\lambda)}+\lambda^{-4}\nu^2\notag\\
            &\stackrel{\eqref{eq:property.gamma.factorial},\eqref{eq:property.digamma}}{=}(1+1/\lambda)\lambda^{-3}
            \left(\psi_1(2+1/\lambda)+[\psi(1/\lambda)+\lambda+(1+1/\lambda)^{-1}+\log\lambda]^2\right)-\lambda^{-4}\nu^2\notag\\
            &\stackrel{\eqref{eq:property.trigamma}}{=}(1+1/\lambda)\lambda^{-3}
            \left(\psi_1(1+1/\lambda)-(1+1/\lambda)^{-2}+[(1+1/\lambda)^{-1}+\nu]^2\right)-\lambda^{-4}\nu^2\notag\\
            &=(1+1/\lambda)\lambda^{-3}\left(\psi_1(1+1/\lambda)+\nu^2+2\nu(1+1/\lambda)^{-1}\right)-\lambda^{-4}\nu^2\notag\\
            &=\lambda^{-3}\left((1+1/\lambda)\psi_1(1+1/\lambda)+\nu^2(1+1/\lambda)+2\nu\right)-\lambda^{-4}\nu^2\notag\\
            &=\lambda^{-3}\left[(1+1/\lambda)\psi_1(1+1/\lambda)+\nu(2+\nu)\right]
            =\boxed{\lambda^{-3}\left[(1+1/\lambda)\psi_1(1+1/\lambda)+\phi^2-1\right]}\label{eq:J.theta2.theta2}.
        \end{align}
        \vspace{-3mm}
        This ends the proof.
    \end{proof}

    \begin{proof}[\bf Proof of Proposition~\ref{prop:MLE.score.asymptotics.Taylor}]
        We work under $H_0$ throughout this proof.
        Using the fundamental theorem of calculus to expand $\bb{r}_n(\hat{\bb{\kappa}}_n)$ around $\bb{\kappa}$, we have
        \begin{equation}\label{eq:prop:MLE.score.asymptotics.expansion.start}
                \bb{r}_n(\hat{\bb{\kappa}}_n) = \bb{r}_n(\bb{\kappa}) + \int_0^1 \bb{r}_n'(\bb{\kappa}_{n,v}^{\star}) \rd v \, (\hat{\bb{\kappa}}_n - \bb{\kappa}),
        \end{equation}
        where $\bb{\kappa}_{n,v}^{\star} \leqdef \bb{\kappa} + v(\hat{\bb{\kappa}}_n - \bb{\kappa})$ for $v\in [0,1]$.

        From \eqref{eq:score.statistic.known.2} and \eqref{eq:vector.d}, we know that for all $\bb{t}\in \R \times (0,\infty)$,
        \begin{equation}\label{eq:partial.derivatives.expression.in.taylor.expansion}
            \bb{r}_n'(\bb{t}) \leqdef\left(\frac{\partial}{\partial\mu}\bb{r}_n(\bb{\kappa}), \frac{\partial}{\partial\sigma}\bb{r}_n(\bb{\kappa})\right)=
            \begin{pmatrix}
                \frac{1}{n} \sum_{i=1}^n U_1(X_i,\bb{t}) & \frac{1}{n} \sum_{i=1}^n U_2(X_i,\bb{t}) \\[2mm]
                \frac{1}{n} \sum_{i=1}^n U_3(X_i,\bb{t}) & \frac{1}{n} \sum_{i=1}^n U_4(X_i,\bb{t})
            \end{pmatrix}
        \end{equation}
        where $y \leqdef (x - t_1)/t_2$ and
        \begin{alignat*}{3}
            &U_1(x,\bb{t}) \leqdef \frac{2\lambda}{\sigma} |y|^{\lambda-1}; \quad &&U_2(x,\bb{t}) \leqdef \frac{2\lambda}{\sigma} y |y|^{\lambda-1}; \\
            &U_3(x,\bb{t}) \leqdef \frac{1}{\lambda\sigma} |y|^{\lambda-1} \mathrm{sign}\left(y\right) \left[\lambda \log |y| + 1\right]; \quad &&U_4(x,\bb{t}) \leqdef \frac{1}{\lambda\sigma} |y|^{\lambda} \left[\lambda \log |y| + 1\right].
        \end{alignat*}
        By the triangle inequality and Lemma~\ref{lem:standard.ULLN}, we have, for all $(k,\lambda)\in \{1,2,3,4\} \times [1,\infty) \backslash \{(3,1)\}$,
        \begin{equation}\label{eq:uniform.convergence.partial.derivatives.in.taylor.expansion}
            \begin{aligned}
                &\PP\bigg(\limsup_{n\to\infty} \sup_{v\in [0,1]} \Big|\frac{1}{n} \sum_{i=1}^n (U_k(X_i,\bb{\kappa}_{n,v}^{\star}) - U_k(X_i,\bb{\kappa}))\Big| > 0\bigg) \\
                &\quad\leq 2 \, \PP\bigg(\limsup_{n\to\infty} \sup_{v\in [0,1]} \Big|\frac{1}{n} \sum_{i=1}^n U_k(X_i,\bb{\kappa}_{n,v}^{\star}) - \overline{U}_k(\bb{\kappa})\Big| > 0\bigg) = 0.
            \end{aligned}
        \end{equation}
        Since we already know from \eqref{eq:prop:MLE.score.asymptotics.explicit.MLE} (this will be proved below independently of Proposition~\ref{prop:MLE.score.asymptotics.Taylor})
        that
        \begin{equation}\label{eq:convergence.mu.sigma.in.taylor.expansion}
            \hat{\bb{\kappa}}_n - \bb{\kappa} = O_{\PP}(n^{-1/2}) \bb{1}_2,
        \end{equation}
        we deduce from \eqref{eq:prop:MLE.score.asymptotics.expansion.start}, \eqref{eq:partial.derivatives.expression.in.taylor.expansion}, \eqref{eq:uniform.convergence.partial.derivatives.in.taylor.expansion} and \eqref{eq:convergence.mu.sigma.in.taylor.expansion} that, for all $(k,\lambda)\in \{1,2,3,4\} \times [1,\infty) \backslash \{(3,1)\}$,
        \begin{equation}\label{eq:lem:standard.ULLN.expansion.final}
            \bb{r}_n(\hat{\bb{\kappa}}_n) = \bb{r}_n(\bb{\kappa}) + \bb{r}_n'(\bb{\kappa}) (\hat{\bb{\kappa}}_n - \bb{\kappa}) + o_{\hspace{0.3mm}\PP}(n^{-1/2}) \bb{1}_2,
        \end{equation}
        which is the statement we wanted to prove, see \eqref{eq:prop:MLE.score.asymptotics.Taylor}.

        When $(k,\lambda) = (3,1)$, we have to be a bit more careful.
        Indeed, Lemma~\ref{lem:standard.ULLN} cannot be applied to $U_3$ in this case because the log term implies that, for any $\delta > 0$, $\sup_{\bb{t}\in B_{\delta}[\bb{\kappa}]} |U_3(x,\bb{t})| = \infty$ for all $x\in B_{\delta}[\mu]$, and thus ({\rm C.2}) cannot be satisfied.
        Instead, we use Lemma~\ref{lem:example.1} below (again, this will be proved independently of Proposition~\ref{prop:MLE.score.asymptotics.Taylor}), which is a consequence of a uniform law of large numbers developed in \cite{MR3842623} for summands that blow up.
        By using successively Jensen's inequality, Fubini's theorem, the triangle inequality and Lemma~\ref{lem:example.1}, we have
        \begin{equation}
            \begin{aligned}
                &\EE\bigg|\int_0^1 \frac{1}{n} \sum_{i=1}^n U_3(X_i,\bb{\kappa}_{n,v}^{\star}) \rd v - \int_0^1 \frac{1}{n} \sum_{i=1}^n U_3(X_i,\bb{\kappa}) \rd v \bigg| \\
                &\hspace{10mm}\leq \int_0^1 \EE\bigg|\frac{1}{n} \sum_{i=1}^n U_3(X_i,\bb{\kappa}_{n,v}^{\star}) - \frac{1}{n} \sum_{i=1}^n U_3(X_i,\bb{\kappa})\bigg| \rd v \\
                &\hspace{10mm}\leq 2 \sup_{v\in [0,1]} \EE\bigg|\frac{1}{n} \sum_{i=1}^n U_3(X_i,\bb{\kappa}_{n,v}^{\star}) -  \EE\big[U_3(X_1,\bb{\kappa})\big]\bigg|
                \stackrel{n\to\infty}{\longrightarrow} 0.
            \end{aligned}
        \end{equation}
        By Markov's inequality, this yields, for $\lambda = 1$,
        \begin{equation}\label{eq:prop:MLE.score.asymptotics.special.case.end.U.2}
            \bigg|\int_0^1 \frac{1}{n} \sum_{i=1}^n U_3(X_i,\bb{\kappa}_{n,v}^{\star}) \rd v - \int_0^1 \frac{1}{n} \sum_{i=1}^n U_3(X_i,\bb{\kappa}) \rd v \bigg| \stackrel{\PP}{\longrightarrow} 0.
        \end{equation}
        Putting \eqref{eq:convergence.mu.sigma.in.taylor.expansion} and \eqref{eq:prop:MLE.score.asymptotics.special.case.end.U.2} together into \eqref{eq:prop:MLE.score.asymptotics.expansion.start} proves the statement of the proposition when $(k,\lambda) = (3,1)$, assuming that Lemma~\ref{lem:example.1} is true.
    \end{proof}

    In order to conclude the proof of Proposition~\ref{prop:MLE.score.asymptotics.Taylor}, it remains to prove the following lemma.

    \begin{lemma}\label{lem:example.1}
        Let $X_1,X_2,X_3,\dots$ be a sequence of i.i.d.\ random variables such that $X_1\sim \mathrm{APD}_{\lambda}(\bb{\theta}_0,\bb{\kappa})$, where $\lambda = 1$, $\mu\in \R$ and $\sigma > 0$.
        In particular, the density function of $X_1$ is given by
        \begin{equation}
            f_{X_1}(x) \leqdef \frac{1}{2 \sigma} e^{-\left|\frac{x - \mu}{\sigma}\right|}, \quad x\in \R.
        \end{equation}
        Define $H : \R\backslash\{0\} \to \R$ by
        \begin{equation}
             H(y) \leqdef \mathrm{sign}(y) (\log|y| + 1).
        \end{equation}
        Let $\{\hat{\mu}_{1}\}_{n\in \N}$ and $\{\hat{\sigma}_{1}\}_{n\in \N}$ be the sequences of maximum likelihood estimators found in Proposition~2.3 for $\lambda = 1$:
        \begin{equation}\label{eq.MLEs.laplace.example}
            \hat{\mu}_{1} \leqdef \mathrm{median}(X_1,X_2,\dots,X_n) \quad \text{and} \quad \hat{\sigma}_{1} = \frac{1}{n} \sum_{i=1}^n  |X_i - \hat{\mu}_{1}|.
        \end{equation}
        The median is defined in Remark~2.4.
        For $v\in [0,1]$, let $\mu_{n,v}^{\star} \leqdef \mu + v(\hat{\mu}_{1} - \mu)$ and $\sigma_{n,v}^{\star} \leqdef \sigma + v(\hat{\sigma}_{1} - \sigma)$.
        Then,
        \begin{equation}\label{eq:lem:example.1.conclusion}
           \lim_{n\to\infty} \sup_{v\in [0,1]} \EE\bigg|\frac{1}{n} \sum_{i=1}^n \ind{1}_{\{X_i \neq \mu_{n,v}^{\star}\}} H\Big(\frac{X_i - \mu_{n,v}^{\star}}{\sigma_{n,v}^{\star}}\Big) - \EE\bigg[H\Big(\frac{X_1 - \mu}{\sigma}\Big)\bigg]\bigg| = 0.
        \end{equation}
    \end{lemma}

    \begin{proof}[\bf Proof of Lemma~\ref{lem:example.1}]
        Without loss of generality, assume that $\mu = 0$.
        Since $\sigma > 0$ and $\hat{\sigma}_{1} > 0$ a.s., we have $\sigma_{n,v}^{\star} > 0$ a.s. for any $v\in [0,1]$, which implies that the factors $\sigma_{n,v}^{\star}$ and $\sigma$ in the $\mathrm{sign}$ function of $H$ can be ignored. Also, $f_{X_1}$ is symmetric, so $\EE[\mathrm{sign}(X_1)] = 0$. Combining these facts together, the supremum in \eqref{eq:lem:example.1.conclusion} is bounded from above by
        \begin{equation}\label{eq:lem:example.1.c.plus.d}
            \begin{aligned}
                \hspace{-2mm}(c) + (d)
                &\leqdef \sup_{v\in [0,1]} \EE\bigg|\frac{1}{n} \sum_{i=1}^n \ind{1}_{\{X_i \neq \mu_{n,v}^{\star}\}} h(X_i - \mu_{n,v}^{\star}) - \EE\big[h(X_1)\big]\bigg| \\
                &\quad+ \sup_{v\in [0,1]} \EE\bigg|\big(1 - \log \sigma_{n,v}^{\star}\big) \cdot \frac{1}{n} \sum_{i=1}^n \ind{1}_{\{X_i \neq \mu_{n,v}^{\star}\}} \mathrm{sign}(X_i - \mu_{n,v}^{\star})\bigg|,
            \end{aligned}
        \end{equation}
        where $h(y) \leqdef \mathrm{sign}(y) \log|y|$.
        By Lemma~3.1 in \cite{MR3842623}, we have $(c)\to 0$.

        It remains to prove that $(d) \to 0$ in \eqref{eq:lem:example.1.c.plus.d}.
        By the Cauchy-Schwarz inequality,
        \begin{equation}
            \begin{aligned}
                (d)^2
                &\leq \EE\bigg[\sup_{v\in [0,1]} \big(1 - \log \sigma_{n,v}^{\star}\big)^2\bigg] \cdot \EE\bigg[\sup_{v\in [0,1]} \Big(\frac{1}{n} \sum_{i=1}^n \ind{1}_{\{X_i \neq \mu_{n,v}^{\star}\}} \mathrm{sign}(X_i - \mu_{n,v}^{\star})\Big)^2\bigg] \\
                &\reqdef (d.1) \cdot (d.2).
            \end{aligned}
        \end{equation}
        Below, we show that $(d.1)$ is bounded and $(d.2)$ tends to zero as $n\to \infty$.
        We start with $(d.2)$.
        Almost surely in $\omega\in \Omega$, the function
        \begin{equation}
            v\mapsto \frac{1}{n} \sum_{i=1}^n \ind{1}_{\{X_i(\omega) \neq \mu_{n,v}^{\star}(\omega)\}} \mathrm{sign}(X_i(\omega) - \mu_{n,v}^{\star}(\omega))
        \end{equation}
        is monotone and equal to zero at $v = 1$ (by definition of $\hat{\mu}_{1}$, recall \eqref{eq:MLE.system.mu.critical.points}).
        Therefore, almost surely in $\omega\in \Omega$, the supremum of the square in $(d.2)$ is always attained at $v = 0$.
        We deduce that
        \begin{equation}\label{eq:prop:example.1.LLN.L2}
            (d.2) = \EE\bigg[\Big(\frac{1}{n} \sum_{i=1}^n \ind{1}_{\{X_i \neq 0\}} \mathrm{sign}(X_i)\Big)^2\bigg] \stackrel{n\to \infty}{\longrightarrow} \Big(\EE\left[\ind{1}_{\{X_1 \neq 0\}} \mathrm{sign}(X_1)\right]\Big)^2 = 0,
        \end{equation}
        by the strong law of large numbers and the bounded convergence theorem.

        Now we show that $(d.1)$ is bounded.
        By successively using the inequality $(\alpha - \beta)^2 \leq 2\alpha^2 + 2\beta^2$, the fact that $z\mapsto (\log z)^2$ always maximizes at one of the two end points on any closed sub-interval of $(0,\infty)$, and the inequality $\max\{a,b\} \leq a + b$ for $a,b \geq 0$, we have
        \begin{equation}
            (d.1)
            \leq \EE\bigg[\sup_{v\in [0,1]} 2 + 2\, (\log \sigma_{n,v}^{\star})^2\bigg]
            \leq 2 + 2\, (\log \sigma)^2 + 2\, \EE\big[(\log \hat{\sigma}_{1})^2\big].
        \end{equation}
        It remains to show that $\EE[(\log \hat{\sigma}_{1})^2] < \infty$.
        Since $\hat{\sigma}_{1}$ is a mean of integrable terms (see \eqref{eq.MLEs.laplace.example}), we expect, at least heuristically (because of large deviations), that, as $n\to\infty$, its density function concentrates more and more around $\sigma$ and decays exponentially faster and faster in the right tail.
        The specific form of the density function of $\hat{\sigma}_{1}$ is given in Equation~(32) of \cite{MR0153074} and confirms the intuition.
        For $N\in\N$ large enough (depending on $\sigma$), there exists $c_{\sigma} > 0$ small enough that, for all $n\geq N$,
        \begin{equation}
            \begin{aligned}
                 \EE\big[(\log \hat{\sigma}_{1})^2\big]
                 &= \int_{(0,\sigma/2) \cup (\sigma/2,(3\sigma/2) \vee 1) \cup ((3\sigma/2) \vee 1, \infty)} (\log s)^2 \cdot f_{\hat{\sigma}_{1}}(s) \rd s \\
                 &\leq \underbrace{\int_0^{\sigma/2} (\log(s))^2 \cdot 1 \, \rd s}_{< ~\infty} \, + \, M_{\sigma} \underbrace{\int_{\sigma/2}^{(3\sigma/2) \vee 1} f_{\hat{\sigma}_{1}}(s) \rd s}_{\leq ~1} \, + \, \underbrace{\int_{(3\sigma/2) \vee 1}^{\infty} s \cdot e^{-c_{\sigma} s} \rd s}_{< ~\infty} < \infty,
            \end{aligned}
        \end{equation}
        where $a \vee b \leqdef \max\{a,b\}$ and $M_{\sigma} \leqdef \max_{s\in [\sigma/2,(3\sigma/2) \vee 1]} (\log s)^2 < \infty$.
        This ends the proof.
    \end{proof}

    \begin{proof}[\bf Proof of Proposition~\ref{prop:MLE.score.asymptotics.explicit}]
        We work under $H_0$ throughout this proof.
        Let $X\sim \mathrm{APD}_{\lambda}(\bb{\theta}_0,\bb{\kappa})$, $Y\leqdef \sigma^{-1}(X - \mu)$ and $y\leqdef \sigma^{-1}(x - \mu)$.
        By the weak law of large numbers, the chain rule, integration by parts and $\frac{\partial y}{\partial \bb{\kappa}^{\top}}\leqdef \big(\frac{\partial y}{\partial \mu}, \frac{\partial y}{\partial \sigma}\big)=-\sigma^{-1}(1, y)$, we have
        \begin{equation}
            \begin{aligned}
                \bb{r}_n'(\bb{\kappa})&\hspace{-1mm}\stackrel{\phantom{\eqref{eq:asymptotic.normality.vector.d}}}{=}\frac{\partial}{\partial \bb{\kappa}^{\top}}\frac{1}{n} \sum_{i=1}^n \bb{d}_{\bb{\theta}}(Y_i)
                \stackrel{\phantom{\eqref{eq:asymptotic.normality.vector.d}}}{=} \EE\Big[\frac{\partial}{\partial \bb{\kappa}^{\top}} \bb{d}_{\bb{\theta}}(Y)\Big] + o_{\hspace{0.3mm}\PP}(1) \bb{1}_2^{\phantom{\top}}\hspace{-1mm}\bb{1}_2^{\top} \stackrel{\phantom{\eqref{eq:asymptotic.normality.vector.d}}}{=} \EE\Big[\bb{d}_{\bb{\theta}}'(Y) \frac{\partial Y}{\partial \bb{\kappa}^{\top}}\Big] + o_{\hspace{0.3mm}\PP}(1)\bb{1}_2^{\phantom{\top}}\hspace{-1mm}\bb{1}_2^{\top}\\
                &\hspace{-1mm}\stackrel{\phantom{\eqref{eq:asymptotic.normality.vector.d}}}{=} \Big.\Big[\bb{d}_{\bb{\theta}}(y) \frac{\partial y}{\partial \bb{\kappa}^{\top}} f(y \nvert \bb{\theta}_0,(0,1)^{\top})\Big]\Big|_{-\infty}^{\infty} \hspace{-1mm}- \int_{-\infty}^{\infty} \bb{d}_{\bb{\theta}}(y) \frac{\partial}{\partial y} \Big[\frac{\partial y}{\partial \bb{\kappa}^{\top}} f(y \nvert \bb{\theta}_0,(0,1)^{\top})\Big] \rd y + o_{\hspace{0.3mm}\PP}(1)\bb{1}_2^{\phantom{\top}}\hspace{-1mm}\bb{1}_2^{\top}\\
                &\hspace{-1mm}\stackrel{\phantom{\eqref{eq:asymptotic.normality.vector.d}}}{=} [0]+ \sigma^{-1}\int_{-\infty}^{\infty} \bb{d}_{\bb{\theta}}(y)  \left(\frac{\partial}{\partial y} f(y \nvert \bb{\theta}_0,(0,1)^{\top}),
                \frac{\partial}{\partial y} y f(y \nvert \bb{\theta}_0,(0,1)^{\top})\right) \rd y + o_{\hspace{0.3mm}\PP}(1)\bb{1}_2^{\phantom{\top}}\hspace{-1mm}\bb{1}_2^{\top}\\
                &\hspace{-1mm}\stackrel{\phantom{\eqref{eq:asymptotic.normality.vector.d}}}{=} \sigma^{-1}\int_{-\infty}^{\infty} \bb{d}_{\bb{\theta}}(y)  \left(\frac{\partial}{\partial y} \log f(y \nvert \bb{\theta}_0,(0,1)^{\top}),
                \Big(1 + y\frac{\partial}{\partial y}\log f(y \nvert \bb{\theta}_0,(0,1)^{\top})\Big)\right)f(y \nvert \bb{\theta}_0,(0,1)^{\top}) \rd y \\
                &\quad\stackrel{\phantom{\eqref{eq:asymptotic.normality.vector.d}}}{+} o_{\hspace{0.3mm}\PP}(1)\bb{1}_2^{\phantom{\top}}\hspace{-1mm}\bb{1}_2^{\top}\stackrel{\eqref{eq:d.kappa}}{=}\, \,  - \sigma^{-1}\EE\big[\bb{d}_{\bb{\theta}}(Y) \bb{d}_{\bb{\kappa}}(Y)^{\top}\big] + o_{\hspace{0.3mm}\PP}(1)\bb{1}_2^{\phantom{\top}}\hspace{-1mm}\bb{1}_2^{\top}
                \stackrel{\phantom{\eqref{eq:asymptotic.normality.vector.d}}}{=} -\sigma^{-1} J_{\bb{\theta}\bb{\kappa}} + o_{\hspace{0.3mm}\PP}(1)\bb{1}_2^{\phantom{\top}}\hspace{-1mm}\bb{1}_2^{\top}.
            \end{aligned}
        \end{equation}
        This proves \eqref{eq:prop:MLE.score.asymptotics.explicit.r.prime}.
        Now, we show the asymptotics of $n^{1/2} (\hat{\bb{\kappa}}_n - \bb{\kappa})$.
        By applying Theorem~5.23 in \cite{MR1652247} (we verify the technical conditions of the theorem below) with
        \begin{equation}
            m_{\bb{t}}(x) \leqdef \log f(x \nvert \bb{\theta}_0,\bb{t}), \quad \bb{t}\in \R \times (0,\infty), ~x\in \R,
        \end{equation}
        ($\hat{\bb{\kappa}}_n \in \mathrm{argmax}_{\bb{t}\in \R \times (0,\infty)} \frac{1}{n} \sum_{i=1}^n m_{\bb{t}}(X_i)$, recall Proposition~2.3
        and the fact that $\big.\tfrac{\partial}{\partial \bb{t}} m_{\bb{t}}(x)\big|_{\bb{t} = \bb{\kappa}} = \sigma^{-1} \bb{d}_{\bb{\kappa}}(y)$ by \eqref{eq:d.kappa}
        yields
        \begin{equation}\label{eq:van.der.vaart.thm.5.23.apply}
            \begin{aligned}
            n^{1/2} (\hat{\bb{\kappa}}_n - \bb{\kappa})
            &\stackrel{\phantom{\eqref{eq:asymptotic.normality.vector.d}}}{=} - \EE\big[\sigma^{-1} \bb{d}_{\bb{\kappa}}(Y) \sigma^{-1} \bb{d}_{\bb{\kappa}}(Y)^{\top}\big]^{-1}
            \frac{1}{\sqrt{n}} \sum_{i=1}^n \sigma^{-1} \bb{d}_{\bb{\kappa}}(Y_i) + o_{\hspace{0.3mm}\PP}(1) \bb{1}_2 \\
            &\stackrel{\phantom{\eqref{eq:asymptotic.normality.vector.d}}}{=} \sigma J_{\bb{\kappa}\bb{\kappa}}^{-1} \frac{1}{\sqrt{n}} \sum_{i=1}^n \bb{d}_{\bb{\kappa}}(Y_i) + o_{\hspace{0.3mm}\PP}(1) \bb{1}_2.
            \end{aligned}
        \end{equation}
        This proves \eqref{eq:prop:MLE.score.asymptotics.explicit.MLE}.
        Finally, since $\frac{1}{\sqrt{n}} \sum_{i=1}^n \bb{d}_{\bb{\kappa}}(Y_i)$ is $O_{\PP}(1)$ by Proposition~\ref{prop:asymptotic.normality.vector.d}, Equation \eqref{eq:prop:MLE.score.asymptotics.explicit.score} follows directly from Proposition~\ref{prop:MLE.score.asymptotics.Taylor}, \eqref{eq:prop:MLE.score.asymptotics.explicit.r.prime} and \eqref{eq:prop:MLE.score.asymptotics.explicit.MLE}.

        For the convenience of the reader, we verify below the $6$ conditions of Theorem~5.23 in \cite{MR1652247}, which allowed us to write the first equality in \eqref{eq:van.der.vaart.thm.5.23.apply}:
        \begin{enumerate}
            \item For all $\bb{t}\in \R \times (0,\infty)$, the function $x\mapsto m_{\bb{t}}(x)$ is measurable (this is obvious).
            \item For all $x \neq \mu$ (and thus for almost-all $x$ under the measure on $\R$ induced by the distribution of $X$), the function $\bb{t}\mapsto m_{\bb{t}}(x)$ is differentiable at $\bb{t} = \bb{\kappa}$, and the derivative at that point is
                \begin{equation}
                    \big.\tfrac{\partial}{\partial \bb{t}} m_{\bb{t}}(x)\big|_{\bb{t} = \bb{\kappa}} = \sigma^{-1} \bb{d}_{\bb{\kappa}}(y) = \sigma^{-1} (|y|^{\lambda - 1} \mathrm{sign}(y), |y|^{\lambda} - 1)^{\top}, \quad \text{by \eqref{eq:d.kappa} and \eqref{eq:vector.d}}.
                \end{equation}
            \item The function $\dot m(x) \leqdef \sup_{\bb{t}\in [\mu - 1,\mu + 1] \times [\sigma/2,2\sigma]} \sigma^{-1} (2|y|^{\lambda} + 1)$ is measurable, and satisfies $\EE[\dot m (X)^2] < \infty$ (this is easy to verify because the supremum is attained) and also
                \begin{equation}
                    |m_{\bb{t}_1}(x) - m_{\bb{t}_2}(x)| \leq \dot m(x) \, \|\bb{t}_1 - \bb{t}_2\|_1, \quad \text{for all } x\in \R, ~ \bb{t}\in [\mu - 1,\mu + 1] \times [\sigma/2,2\sigma].
                \end{equation}
                This last equation is just a consequence of the mean value theorem and the fact that, for all $x\in \R$, the function $\bb{t}\mapsto \tfrac{\partial}{\partial \bb{t}} m_{\bb{t}}(x)$ is uniformly continuous on the compact set $[\mu - 1,\mu + 1] \times [\sigma/2,2\sigma]$.
            \item The map $\bb{t}\mapsto \EE[m_{\bb{t}}(X)]$ admits the following second order Taylor expansion at $\bb{t} = \bb{\kappa}$:
                \begin{equation}
                    \EE[m_{\bb{t}}(X)] = \EE[m_{\bb{\kappa}}(X)] + \underbrace{\EE[\sigma^{-1} \bb{d}_{\bb{\kappa}}(Y)^{\top}]}_{=~(0,0)^{\top}} \, (\bb{t} - \bb{\kappa}) + \frac{1}{2} (\bb{t} - \bb{\kappa})^{\top} V_{\bb{\kappa}} \, (\bb{t} - \bb{\kappa}) + o_{\bb{\kappa}}\big(\|\bb{t} - \bb{\kappa}\|_1^2\big),
                \end{equation}
                where the matrix $V_{\bb{\kappa}} \leqdef - \EE[\sigma^{-1} \bb{d}_{\bb{\kappa}}(Y) \sigma^{-1} \bb{d}_{\bb{\kappa}}(Y)^{\top}]$ is finite, non-singular, symmetric (and even diagonal) by the proof of Proposition~\ref{prop:asymptotic.normality.vector.d}.
                Indeed, whenever $\lambda > 1$, the expansion easily holds true because integration by parts shows that $V_{\bb{\kappa}} = \EE\big[\big.\tfrac{\partial^2}{\partial \bb{t} \partial \bb{t}^{\top}} m_{\bb{t}}(X)\big|_{\bb{t} = \bb{\kappa}}\big]$ and $\bb{t}\mapsto \EE\big[\tfrac{\partial^2}{\partial \bb{t} \partial \bb{t}^{\top}} m_{\bb{t}}(X)\big]$ is uniformly continuous on $[\mu - 1,\mu + 1] \times [\sigma/2,2\sigma]$. The only nontrivial case to verify is $\lambda = 1$.
                In that case, we have $m_{\bb{t}}(x) = - \big|\frac{x - t_1}{t_2}\big| - \log(2t_2)$, $V_{\bb{\kappa}} = - \sigma^{-2} I_2$, and therefore
                \begin{align*}
                    &\EE[m_{\bb{t}}(X)] - \EE[m_{\bb{\kappa}}(X)] = \EE\bigg[-\Big|\frac{X - t_1}{t_2}\Big| - \log(2t_2)\bigg] - \EE\big[- |Y| - \log(2\sigma)\big] \\
                    &\quad= \EE\bigg[|Y| - \frac{\sigma}{t_2} \cdot \Big|Y + \frac{\mu - t_1}{\sigma}\Big|\bigg] - \log\Big(1 + \frac{t_2 - \sigma}{\sigma}\Big) \\
                    &\quad= \frac{t_2 - \sigma}{t_2} \int_{-\infty}^{\infty} |y| \frac{e^{-|y|}}{2} \rd y + \frac{\sigma}{t_2} \int_{-\infty}^{\infty} \Big(|y| - \big|y - \tfrac{t_1 - \mu}{\sigma}\big|\Big) \frac{e^{-|y|}}{2} \rd y - \log\Big(1 + \frac{t_2 - \sigma}{\sigma}\Big) \\
                    &\quad= \frac{t_2 - \sigma}{t_2} \cdot 1 + \frac{\sigma}{t_2} \bigg\{1 - \big|\tfrac{t_1 - \mu}{\sigma}\big| - \exp\Big(-\big|\tfrac{t_1 - \mu}{\sigma}\big|\Big)\bigg\} - \Big(\frac{t_2 - \sigma}{\sigma} - \frac{1}{2} \frac{(t_2 - \sigma)^2}{\sigma^2} + O_{\sigma}(|t_2 - \sigma|^3)\Big) \\
                    &\quad= \frac{\sigma}{t_2} \bigg\{-\frac{1}{2} \sigma^{-2} (t_1 - \mu)^2 + O_{\sigma}(|t_1 - \mu|^3)\bigg\} + \sigma^{-2} (t_2 - \sigma)^2 \Big(-\frac{\sigma}{t_2} + \frac{1}{2}\Big) + O_{\sigma}(|t_2 - \sigma|^3) \\
                    &\quad= -\frac{1}{2} \sigma^{-2} (t_1 - \mu)^2 - \frac{1}{2} \sigma^{-2} (t_2 - \sigma)^2 + O_{\sigma}\big(\|\bb{t} - \bb{\kappa}\|_1^3\big) = \frac{1}{2} (\bb{t} - \bb{\kappa})^{\top} V_{\bb{\kappa}} \, (\bb{t} - \bb{\kappa}) + o_{\bb{\kappa}}\big(\|\bb{t} - \bb{\kappa}\|_1^2\big),
                \end{align*}
                where the second integral on the third line was computed using \texttt{Wolfram Mathematica}.
            \item $\EE[m_{\hat{\bb{\kappa}}_n}(X)] = \EE[\sup_{\bb{t}\in \R \times (0,\infty)} m_{\bb{t}}(X)] \geq \sup_{\bb{t}\in \R \times (0,\infty)} \EE[m_{\bb{t}}(X)] - o_{\hspace{0.3mm}\PP}(n^{-1})$ is trivially satisfied.
            \item $\hat{\bb{\kappa}}_n \xrightarrow{\PP} \bb{\kappa}$ since $\hat{\bb{\kappa}}_n \rightarrow \bb{\kappa}$ a.s.\ by Lemma~\ref{lem:convergence.almost.sure.MLE.estimators}.
        \end{enumerate}
        This ends the proof.
    \end{proof}

\subsection{Proof of Proposition~3.6}\label{Proof_ZKnetstar}

    By applying a uniform law of large numbers, we can show the following preliminary result.

    \begin{lemma}\label{lem:K.B.a.s.convergence}
        Under $H_0$ and for $\lambda\geq 1$, we have, as $n\to \infty$,
        \begin{equation}\label{eq:lem:K.B.a.s.convergence.eq.1}
            K_{\lambda}(\Xn) \xrightarrow{\mathrm{a.s.}} \frac{\lambda + \log \lambda + \psi(1/\lambda)}{\lambda}, \qquad S_{\lambda}(\Xn) \xrightarrow{\mathrm{a.s.}} 0,
        \end{equation}
        so that
        \begin{equation}\label{eq:lem:K.B.a.s.convergence.eq.2}
           K^{\mathrm{net}}_{\lambda}(\Xn) \xrightarrow{\mathrm{a.s.}} \frac{\lambda + \log \lambda + \psi(1/\lambda)}{\lambda} > 0.
        \end{equation}
    \end{lemma}

    \begin{proof}[\bf Proof of Lemma~\ref{lem:K.B.a.s.convergence}]
        For $v\in [0,1]$, $t_1\in \R$ and $t_2 > 0$, let $\bb{\kappa}_{n,v}^{\star} \leqdef \bb{\kappa} + v(\hat{\bb{\kappa}}_n - \bb{\kappa})$, $y \leqdef (x - t_1)/t_2$,
        \begin{equation}
            U_5(x,\bb{t}) \leqdef |y|^{\lambda} \log|y| \quad \text{and} \quad U_6(x,\bb{t}) \leqdef |y|^{\lambda} \mathrm{sign}(y), \quad \text{where } \bb{t} \leqdef (t_1,t_2)^{\top}.
        \end{equation}
        By Definition~2.6, note that
        \begin{equation}
            K_{\lambda}(\Xn) = \frac{1}{n} \sum_{i=1}^n U_5(X_i,\bb{\kappa}_{n,v}^{\star}) \quad \text{and} \quad S_{\lambda}(\Xn) = \frac{1}{n} \sum_{i=1}^n U_6(X_i,\bb{\kappa}_{n,v}^{\star}).
        \end{equation}
        By the triangle inequality and Lemma~\ref{lem:standard.ULLN}, we have, for all $(k,\lambda)\in \{5,6\} \times [1,\infty)$,
        \begin{equation}\label{eq:lem:K.B.a.s.convergence.proof.eq.1}
            \begin{aligned}
                &\PP\bigg(\limsup_{n\to\infty} \sup_{v\in [0,1]} \Big|\frac{1}{n} \sum_{i=1}^n (U_k(X_i,\bb{\kappa}_{n,v}^{\star}) - U_k(X_i,\bb{\kappa}))\Big| > 0\bigg) \\
                &\quad\leq 2 \, \PP\bigg(\limsup_{n\to\infty} \sup_{v\in [0,1]} \Big|\frac{1}{n} \sum_{i=1}^n U_k(X_i,\bb{\kappa}_{n,v}^{\star}) - \overline{U}_k(\bb{\kappa})\Big| > 0\bigg) = 0,
            \end{aligned}
        \end{equation}
        where, under $H_0$,
        \begin{equation}\label{eq:lem:K.B.a.s.convergence.proof.eq.2}
            \begin{aligned}
                &\overline{U}_5(\bb{k}) = \EE\left[\Big|\frac{X - \mu}{\sigma}\Big|^{\lambda} \log\Big|\frac{X - \mu}{\sigma}\Big|\right] = \frac{\psi(1 + 1/\lambda) + \log \lambda}{\lambda}, \\
                &\overline{U}_6(\bb{k}) = \EE\left[\Big|\frac{X - \mu}{\sigma}\Big|^{\lambda} \mathrm{sign}\Big(\frac{X - \mu}{\sigma}\Big)\right] = 0,
            \end{aligned}
        \end{equation}
        by \eqref{eq:expectation.ya.log} with $a = \lambda$, and by the symmetry of the density $f(\,\cdot\, |\,\bb{\theta}_0,(0,1)^{\top})$ with respect to $0$ and anti-symmetry of the function $z\mapsto |z|^{\lambda} \mathrm{sign}(z)$, respectively.
        This proves \eqref{eq:lem:K.B.a.s.convergence.eq.1}.
        The limit in \eqref{eq:lem:K.B.a.s.convergence.eq.2} follows by applying the limits from \eqref{eq:lem:K.B.a.s.convergence.eq.1} together in Definition~2.8 for $K^{\mathrm{net}}_{\lambda}(\Xn)$.
        To obtain the positivity on the right-hand side of \eqref{eq:lem:K.B.a.s.convergence.eq.2}, Lemma~2 in \cite{MR162751} shows that $\psi(x) - \log x > - 1/x$ for all $x > 1$, so that for all $\lambda \geq 1$,
        \begin{equation}
            \begin{aligned}
                \frac{\psi(1 + 1/\lambda) + \log \lambda}{\lambda}
                &= \frac{\log(\lambda + 1) + \psi(1 + 1/\lambda) - \log(1 + 1/\lambda)}{\lambda} \\
                &\geq \frac{\log(\lambda + 1) - \frac{1}{1 + 1/\lambda}}{\lambda} \geq \frac{\frac{1}{1/2 + 1/\lambda} - \frac{1}{1 + 1/\lambda}}{\lambda}> 0.
            \end{aligned}
        \end{equation}
        This ends the proof of Lemma~\ref{lem:K.B.a.s.convergence}.
    \end{proof}

    We are now ready to prove Proposition~3.6.
    Assume $H_0$ throughout, and denote
    \begin{equation}\label{eq:ES.EK.VS.VK}
        \begin{aligned}
            &E_S \leqdef 0, \qquad &&E_K\leqdef \frac{\lambda+\log \lambda + \psi(1/\lambda)}{\lambda}, \\
            &V_S\leqdef 1 + \lambda - \frac{\lambda^2}{\Gamma(2-1/\lambda)\Gamma(1/\lambda)}, \qquad &&V_K\leqdef \frac{(1+1/\lambda) \psi_1(1+1/\lambda)-1}{\lambda}.
        \end{aligned}
    \end{equation}
    Using Definition~3.1 and Theorem~3.3, we can write
    \begin{equation}
        \begin{aligned}
            &S_{\lambda}(\Xn) = E_S + n^{-1/2}W_{n,1}, \quad K_{\lambda}(\Xn) = E_K + n^{-1/2} W_{n,2} \quad \text{and} \\
            &K_{\lambda}^{\mathrm{net}}(\Xn) = \max\big\{0, E_K + n^{-1/2} W_{n,2} - (\lambda/2) n^{-1} W_{n,1}^2\big\},
        \end{aligned}
    \end{equation}
    where $W_{n,1}=V_S^{1/2}Z^{*\!}(S_{\lambda})\stackrel{\PP_{H_0}}{\rightsquigarrow} \mathcal{N}(0,V_S)$ and $W_{n,2}=V_K^{1/2}Z^{*\!}(K_{\lambda})\stackrel{\PP_{H_0}}{\rightsquigarrow} \mathcal{N}(0,V_K)$ as $n\to \infty$.
    Then, for any $\omega$ on the event
    \begin{equation}\label{eq:event.A.n.lambda}
        A_{n,\lambda}\leqdef \left\{\omega\in \Omega :
            \begin{array}{l}
                |W_{n,1}(\omega)| \leq n^{1/8}, n^{-1/8} < |W_{n,2}(\omega)| \leq n^{1/8}, ~\text{and} \\
                n^{-1/2} W_{n,2}(\omega) - (\lambda/2) n^{-1} (W_{n,1}(\omega))^2 > -\frac{1}{2} E_K
            \end{array}
            \right\},
    \end{equation}
    we have, as $n\to \infty$,
    \begin{align*}
        \frac{Z^{*\!}(K^{\mathrm{net}}_{\lambda})}{Z^{*\!}(K_{\lambda})} - 1
        &= \left[\frac{n^{1/2}\big((K^{\mathrm{net}}_{\lambda}(\Xn))^{1/4} - E_K^{1/4}\big)}{\big(\frac{1}{16}E_K^{-3/2}V_K\big)^{1/2}} - \frac{n^{1/2}\big(K_{\lambda}(\Xn) - E_K\big)}{V_K^{1/2}}\right] \cdot \frac{V_K^{1/2}}{n^{1/2}\big(K_{\lambda}(\Xn) - E_K\big)} \\
        &=4 E_K^{3/4}\cdot\frac{\big(E_K+n^{-1/2}W_{n,2}-(\lambda/2)n^{-1}W_{n,1}^2\big)^{1/4} - E_K^{1/4} - \frac{1}{4} E_K^{-3/4} n^{-1/2} W_{n,2}} {n^{-1/2} W_{n,2}}.
    \end{align*}
    By a second order Taylor expansion, we know that for all $x > -E_K$,
    \begin{equation}
        \left|(E_K + x)^{1/4} - E_K^{1/4} - \frac{1}{4} E_K^{-3/4} x\right| \leq \frac{|x|^2}{2} \sup_{y\in [E_K + (x \wedge 0), E_K + (x \vee 0)]} \frac{3}{16} y^{-7/4} \leq |x|^2 \frac{3}{32} (E_K + (x \wedge 0))^{-7/4},
    \end{equation}
    where $x \vee 0 \leqdef \max\{x,0\}$ and $x \wedge 0 \leqdef \min\{x,0\}$. Hence, for $x^* = n^{-1/2}W_{n,2}-(\lambda/2)n^{-1}W_{n,1}^2$ and for any $\omega\in A_{n,\lambda}$,
    \begin{equation}
        \left|\frac{Z^{*\!}(K^{\mathrm{net}}_{\lambda})}{Z^{*\!}(K_{\lambda})} - 1\right| \leq 4 E_K^{3/4}\cdot \frac{|x^*|^2 \frac{3}{32} (E_K + (x^* \wedge 0))^{-7/4} + \frac{1}{4} E_K^{-3/4} (\lambda/2) n^{-1} W_{n,1}^2}{n^{-1/2} |W_{n,2}|} \leq C_{\lambda} n^{-1/8},
    \end{equation}
    for some constant $C_{\lambda} > 0$ that depends only on $\lambda$.
    Since $\PP_{H_0}(A_{n,\lambda})\to 1$ as $n\to \infty$ by Theorem~3.3, we conclude that
    \begin{equation}
        \frac{Z^{*\!}(K^{\mathrm{net}}_{\lambda})}{Z^{*\!}(K_{\lambda})}\stackrel{\PP_{H_0}}{\longrightarrow} 1.
    \end{equation}
    The last part of Proposition~3.6 follows directly using Slutsky's theorem and Theorem~3.3.

    We give here another proof of the asymptotic distribution of $(Z^*(S_{\lambda}) ~Z^*(K_{\lambda}^{\mathrm{net}}))^{\top}$ in the statement of Proposition~3.6.
        It suffices to prove that, as $n\to \infty$,
        \begin{equation}\label{eq:prop:law.net.kurtosis.to.prove.more.general}
            n^{1/2} \left[
            \begin{pmatrix}
                S_{\lambda}(\Xn) \\
                \big(K^{\mathrm{net}}_{\lambda}(\Xn)\big)^{1/4}
            \end{pmatrix}
            -
            \begin{pmatrix}
                E_S \\
                E_K^{1/4}
            \end{pmatrix}
            \right] \stackrel{\PP_{H_0}}{\scalebox{2}[1.2]{$\rightsquigarrow$}} \mathcal{N}_2
            \left(\bb{0},
                \begin{pmatrix}
                    V_S & 0 \\[1mm]
                    0 & \frac{1}{16} E_K^{-3/2} V_K
                \end{pmatrix}
            \right).
        \end{equation}
        Consider the vector-valued function
        \begin{equation}
            \bb{g}(s,k) =
            \begin{pmatrix}
                s \\
                (\max\{0,k - (\lambda/2) s^2\})^{1/4}
            \end{pmatrix}.
        \end{equation}
        The positivity of the right-hand side in \eqref{eq:lem:K.B.a.s.convergence.eq.2} implies that $\bb{g}$ is differentiable in a small open set that contains the point $(s,k) = (E_S,E_K)$.
        At that point, we have $\bb{g}(E_S,E_K) = (0, E_K^{1/4})$ and
        \begin{equation}
            \begin{pmatrix}
                \frac{\partial}{\partial s} \bb{g}(E_S,E_K)^{\top} \\[1mm]
                \frac{\partial}{\partial k} \bb{g}(E_S,E_K)^{\top}
            \end{pmatrix}
            =
            \begin{pmatrix}
                1 & \frac{-\lambda E_S}{4} \Big(E_K - \frac{\lambda}{2} E_S^2\Big)^{-3/4} \\
                0 & \frac{1}{4} \Big(E_K - \frac{\lambda}{2} E_S^2\Big)^{-3/4}
            \end{pmatrix}
            =
            \begin{pmatrix}
                1 & 0 \\
                0 & \frac{E_K^{-3/4}}{4}
            \end{pmatrix}.
        \end{equation}
        Now, by Equation~\eqref{eqn:law.B.K.H0} and the delta method, we get
        \begin{equation}
            n^{1/2} \left[\bb{g}(S_{\lambda}(\Xn),K_{\lambda}(\Xn)) -
                \begin{pmatrix}
                    0 \\
                    E_K^{1/4}
                \end{pmatrix}\right]
            \stackrel{\PP_{H_0}}{\scalebox{2}[1.2]{$\rightsquigarrow$}} \mathcal{N}_2
            \left(\bb{0},
                \begin{pmatrix}
                    1 & 0 \\
                    0 & \frac{E_K^{-3/4}}{4}
                \end{pmatrix}
                \begin{pmatrix}
                    V_S & 0 \\[1mm]
                    0 & V_K
                \end{pmatrix}
                \begin{pmatrix}
                    1 & 0 \\
                    0 & \frac{E_K^{-3/4}}{4}
                \end{pmatrix}
            \right),
        \end{equation}
        which is exactly the statement \eqref{eq:prop:law.net.kurtosis.to.prove.more.general}.

\subsection{Proof of Proposition~3.8}\label{Proof_ZKnet}

    Assume $H_0$ throughout.
    We have
    \begin{equation}
        \frac{Z(S_{\lambda})}{Z^{*\!}(S_{\lambda})} \stackrel{\mathrm{a.s.}}{=} \Big(1 + \frac{\cone}{n^{\hspace{-0.3mm}\aone}}\Big)^{-1/2} \stackrel{n\to \infty}{\longrightarrow} 1,
    \end{equation}
    since $\aone>0$ for all values of $\lambda$ in Table~3.1. Furthermore, we have
    \begin{equation}
      Z(K^{\mathrm{net}}_{\lambda})= \frac{Z^{*\!}(K^{\mathrm{net}}_{\lambda})}{\left(1 +
        \frac{\cthree}{n^{\hspace{-0.3mm}\athree}}+
        \frac{\cfour}{n^{\hspace{-0.3mm}\afour}}\right)^{1/2}}
     + \frac{n^{1/2}\left(\frac{\lambda+\log \lambda +
        \psi(1/\lambda)}{\lambda}\right)^{1/4}\big(1-(1+\frac{\ctwo}{n^{\hspace{-0.3mm}\atwo}})\big)}{\Big[\frac{1}{16}\left(\frac{\lambda+\log
        \lambda + \psi(1/\lambda)}{\lambda}\right)^{-3/2} \frac{(1+1/\lambda) \psi_1(1+1/\lambda)-1}{\lambda}\left(1 +
        \frac{\cthree}{n^{\hspace{-0.3mm}\athree}}+
        \frac{\cfour}{n^{\hspace{-0.3mm}\afour}}\right)\Big]^{1/2}}.
    \end{equation}
    Since $\athree>0$ and $\afour>0$, and since $\atwo \ge 0.86$ and $|\ctwo|<1$ for all values of $\lambda$ in Table~3.1, we have
    \begin{equation}
        1 +  \frac{\cthree}{n^{\hspace{-0.3mm}\athree}}+
        \frac{\cfour}{n^{\hspace{-0.3mm}\afour}} \stackrel{n\to \infty}{\longrightarrow} 1 \quad \text{ and } \quad       n^{1/2}\Big|1-\big(1+\frac{\ctwo}{n^{\hspace{-0.3mm}\atwo}}\big)\Big|=|\ctwo|\cdot n^{1/2-\hspace{-0.3mm}\atwo}<n^{-0.36} \stackrel{n\to \infty}{\longrightarrow} 0.
    \end{equation}
    Also, note that $Z^{*\!}(K^{\mathrm{net}}_{\lambda})\stackrel{\mathcal{D}}{\longrightarrow} \mathcal{N}(0,1)$ (see Proposition~3.6) and thus $n^{-0.36}/Z^{*\!}(K^{\mathrm{net}}_{\lambda})\stackrel{\PP_{H_0}}{\longrightarrow} 0$ by Slutsky's theorem.
    Therefore, we have
    \begin{equation}
        \frac{Z(K^{\mathrm{net}}_{\lambda})}{Z^{*\!}(K^{\mathrm{net}}_{\lambda})}\stackrel{\PP_{H_0}}{\longrightarrow}1, \quad \text{as } n\rightarrow\infty.
    \end{equation}
    Finally, the last convergence in distribution results follow directly from Proposition~3.6 and Slutsky's theorem.

\subsection{Proof of Theorem~3.10}\label{sec:proof.asymp.dist.H1}

    The following proposition will be a crucial tool to prove the weak convergence of our modified score statistic under local alternatives.
    It is a consequence of the concept of {\it contiguity}, see, e.g., Section 6.2 in \cite{MR1652247}.

    \begin{proposition}\label{prop:iff.convergence.H0.H1n}
        For any statistics $\bb{T}_n \leqdef \bb{T}_n(X_1,X_2,\dots,X_n;\bb{\kappa})$ taking values in $\R^d$,
        \begin{equation}\label{eq:prop:iff.convergence.H0.H1n.eq}
            \bb{T}_n \xrightarrow{\PP_{H_0}} 0 \quad \text{if and only if} \quad \bb{T}_n \xrightarrow{\PP_{H_{\hspace{-0.2mm}1\hspace{-0.1mm},\hspace{-0.2mm}n}}} 0,
        \end{equation}
        as $n\to \infty$.
    \end{proposition}

    As an immediate consequence, we obtain the same decomposition under $H_{\hspace{-0.2mm}1\hspace{-0.1mm},\hspace{-0.2mm}n}$ that we found for the modified score statistic under $H_0$ in Proposition~\ref{prop:MLE.score.asymptotics.explicit}.

    \begin{corollary}\label{cor:score.asymptotics.explicit.score.H1}
        Let $\bb{\delta}\in \R^2 \backslash \{\bb{0}\}$. Then, as $n\to \infty$,
        \begin{equation}\label{eq:cor:score.asymptotics.explicit.score.H1}
            n^{1/2} \bb{r}_n(\hat{\bb{\kappa}}_n)
            =
            \left(I_2 \, ; \, - J_{\bb{\theta}\bb{\kappa}} J_{\bb{\kappa}\bb{\kappa}}^{-1}\right) \frac{1}{\sqrt{n}} \sum_{i=1}^n
            \begin{pmatrix}
                \bb{d}_{\bb{\theta}}(Y_i) \\[1mm]
                \bb{d}_{\bb{\kappa}}(Y_i)
            \end{pmatrix}
            + o_{\hspace{0.3mm}\PP_{H_{\hspace{-0.2mm}1\hspace{-0.1mm},\hspace{-0.2mm}n}}}\hspace{-1mm}(1)\bb{1}_2.
        \end{equation}
    \end{corollary}

    We now use Le Cam's third lemma to prove the analogue of Proposition~\ref{prop:asymptotic.normality.vector.d} under $H_{\hspace{-0.2mm}1\hspace{-0.1mm},\hspace{-0.2mm}n}$.
    Our aim is to obtain the asymptotic distribution of the right-hand side of \eqref{eq:cor:score.asymptotics.explicit.score.H1}.

    \begin{proposition}\label{prop:asymptotic.normality.vector.d.H1}
        Let $\bb{\delta}\in \R^2 \backslash \{\bb{0}\}$. Then, as $n\to\infty$,
        \begin{equation}\label{eq:asymptotic.normality.vector.d.H1}
            \frac{1}{\sqrt{n}} \sum_{i=1}^n
            \begin{pmatrix}
                \bb{d}_{\bb{\theta}}(Y_i) \\[1mm]
                \bb{d}_{\bb{\kappa}}(Y_i)
            \end{pmatrix}
            \stackrel{\PP_{H_{\hspace{-0.2mm}1\hspace{-0.1mm},\hspace{-0.2mm}n}}}{\scalebox{2}[1.2]{$\rightsquigarrow$}}
            \mathcal{N}_4\left(
            \begin{pmatrix}
                J_{\bb{\theta}\bb{\theta}} \bb{\delta} \\[2mm]
                J_{\bb{\theta}\bb{\kappa}}^{\top} \bb{\delta}
            \end{pmatrix}
            ,
            J \leqdef
            \begin{pmatrix}
                J_{\bb{\theta}\bb{\theta}} &J_{\bb{\theta}\bb{\kappa}} \\[2mm]
                J_{\bb{\theta}\bb{\kappa}}^{\top} &J_{\bb{\kappa}\bb{\kappa}}
            \end{pmatrix}
            \right),
        \end{equation}
        where $J$ is given in detail in Proposition~\ref{prop:asymptotic.normality.vector.d}.
    \end{proposition}
    By combining Corollary~\ref{cor:score.asymptotics.explicit.score.H1} and Proposition~\ref{prop:asymptotic.normality.vector.d.H1}, we see that
    \begin{equation}\label{eq:asymp.r.n.under.H1}
        n^{1/2} \bb{r}_n(\hat{\bb{\kappa}}_n)
        \stackrel{\PP_{H_{\hspace{-0.2mm}1\hspace{-0.1mm},\hspace{-0.2mm}n}}}{\scalebox{2}[1.2]{$\rightsquigarrow$}}
        \mathcal{N}_2\left((J_{\bb{\theta}\bb{\theta}} - J_{\bb{\theta} \bb{\kappa}} J_{\bb{\kappa} \bb{\kappa}}^{-1} J_{\bb{\theta} \bb{\kappa}}^{\top}) \bb{\delta}, \Sigma\right),
    \end{equation}
    where the expressions for $J_{\bb{\theta} \bb{\theta}}$, $J_{\bb{\theta} \bb{\kappa}}$ and $J_{\bb{\kappa} \bb{\kappa}}$ are found in \eqref{eq:asymptotic.normality.vector.d.matrix}, and the covariance matrix $\Sigma$ was previously calculated in \eqref{eq:calculation.Sigma}.
    Given \eqref{eq:vector.d} and \eqref{eq:score.statistic.unknown.2}, we deduce from \eqref{eq:asymp.r.n.under.H1} that, as $n\to \infty$,
    \begin{equation}\label{eqn:law.B.K.H1}
        n^{1/2}
        \begin{pmatrix}
            S_{\lambda}(\Xn)\\[1mm]
            K_{\lambda}(\Xn) - \frac{\lambda+\log \lambda + \psi(1/\lambda)}{\lambda}
        \end{pmatrix}
        \stackrel{\PP_{H_{\hspace{-0.2mm}1\hspace{-0.1mm},\hspace{-0.2mm}n}}}{\scalebox{2}[1.2]{$\rightsquigarrow$}}
        \mathcal{N}_2
        \left(
        \begin{pmatrix}
            -\delta_1 \Vone^{1/2} \\[1mm]
            -\delta_2 \Vtwo^{1/2}
        \end{pmatrix},
        \begin{pmatrix}
            1 & 0 \\[1mm]
            0 & 1
        \end{pmatrix}
        \right),
    \end{equation}
    where
    \begin{equation}
        \begin{aligned}
            V_{1,\lambda}
            &\leqdef \left\{\frac{1}{2} \cdot \left[\Sigma^{-1/2} (J_{\bb{\theta}\bb{\theta}} - J_{\bb{\theta} \bb{\kappa}} J_{\bb{\kappa} \bb{\kappa}}^{-1} J_{\bb{\theta} \bb{\kappa}}^{\top})\right]_{11}\right\}^2 \\
            &= \left\{\frac{1}{2} \cdot \frac{4(1 + \lambda) - \left(-\frac{2\lambda^{2-1/\lambda}}{\Gamma(1/\lambda)}\right) \cdot \frac{\Gamma(1/\lambda)}{\lambda^{2-2/\lambda}\Gamma(2-1/\lambda)} \cdot \left(-\frac{2\lambda^{2-1/\lambda}}{\Gamma(1/\lambda)}\right)}{\sqrt{1 + \lambda - \frac{\lambda^2}{\Gamma(2-1/\lambda) \Gamma(1/\lambda)}}}\right\}^2 = 4 (1 + \lambda) - \frac{4 \lambda^2}{\Gamma(2-1/\lambda) \Gamma(1/\lambda)}, \\
            V_{2,\lambda}
            &\leqdef \left\{\lambda \cdot \left[\Sigma^{-1/2} (J_{\bb{\theta}\bb{\theta}} - J_{\bb{\theta} \bb{\kappa}} J_{\bb{\kappa} \bb{\kappa}}^{-1} J_{\bb{\theta} \bb{\kappa}}^{\top})\right]_{22}\right\}^2 \\
            &= \left\{\lambda \cdot \frac{\frac{(1+1/\lambda)\psi_1(1+1/\lambda)+\phi^2-1}{\lambda^3} - \left(-\frac{\phi}{\lambda}\right) \cdot \frac{1}{\lambda} \cdot \left(-\frac{\phi}{\lambda}\right)}{\sqrt{\frac{(1+1/\lambda) \psi_1(1+1/\lambda)-1}{\lambda}}}\right\}^2 = \frac{(1+1/\lambda) \psi_1(1+1/\lambda)-1}{\lambda^3}.
        \end{aligned}
    \end{equation}
    Assuming that we have proofs for Propositions~\ref{prop:iff.convergence.H0.H1n}~and~\ref{prop:asymptotic.normality.vector.d.H1} (see Section~\ref{sec:main.result.2.proofs} below), this ends the proof of Theorem~3.10.

\subsection{Proofs of Propositions~\ref{prop:iff.convergence.H0.H1n}~and~\ref{prop:asymptotic.normality.vector.d.H1} to complete the proof of Theorem~3.10}\label{sec:main.result.2.proofs}

    In order to establish our results under the local alternatives $H_{\hspace{-0.2mm}1\hspace{-0.1mm},\hspace{-0.2mm}n}$, we use Le Cam's first and third lemma (see Lemma~6.4 and Example~6.7 of \cite{MR1652247}).
    The proof structure is inspired by the one presented in Section~4 of \cite{MR2442221}.

    \begin{lemma}[Le Cam's first lemma]\label{lem:Le Cam.first.lemma}
        Let $(P_n, n\in \N)$ and $(Q_n, n\in \N)$ be sequences of probability measures on the measurable spaces $(\Omega_n,\mathcal{A}_n)$.
        Then, the following statements are equivalent:
        \begin{enumerate}[\quad(i)]
            \item $Q_n \lhd P_n$, i.e., $(Q_n, n\in \N)$ is contiguous with respect to $(P_n, n\in \N)$.\vspace{-0.5mm}
            \item If $\frac{\rd P_n}{\rd Q_n} \stackrel{Q_n}{\scalebox{2}[1.2]{$\rightsquigarrow$}} U$ along a subsequence, then $\PP(U > 0) = 1$. \vspace{-1mm}
            \item If $\frac{\rd Q_n}{\rd P_n} \stackrel{P_n}{\scalebox{2}[1.2]{$\rightsquigarrow$}} V$ along a subsequence, then $\EE[V] = 1$. \vspace{-1mm}
            \item For any statistics $\bb{T}_n : \Omega_n \to \R^k$: If $\bb{T}_n \xrightarrow{P_n} 0$, then $\bb{T}_n \xrightarrow{Q_n} 0$. \vspace{-1mm}
        \end{enumerate}
    \end{lemma}

    \begin{lemma}[Le Cam's third lemma]\label{lem:Le Cam.third.lemma}
        Let $(P_n, n\in \N)$ and $(Q_n, n\in \N)$ be sequences of probability measures on the measurable spaces $(\Omega_n,\mathcal{A}_n)$, and let $\bb{W}_{\hspace{-0.5mm}n} : \Omega_n \to \R^k$ be a sequence of random vectors.
        Suppose that $Q_n \lhd P_n$ and
        \begin{equation}
            \begin{pmatrix}
                \bb{W}_{\hspace{-0.5mm}n} \\[1mm]
                \log \frac{\rd Q_n}{\rd P_n}
            \end{pmatrix}
            \stackrel{P_n}{\scalebox{2}[1.2]{$\rightsquigarrow$}} \mathcal{N}_{k+1}
            \left(
            \begin{pmatrix}
                m \\[1mm]
                -\frac{1}{2} s^2
            \end{pmatrix}
            ,
            \begin{pmatrix}
                M &\tau \\[1mm]
                \tau^{\top} &s^2
            \end{pmatrix}
            \right),
        \end{equation}
        where $M\in \R^{k\times k}$ is positive definite, $m,\tau\in \R^k$ and $s^2 > 0$, then
        \begin{equation}
            \bb{W}_{\hspace{-0.5mm}n} \stackrel{Q_n}{\scalebox{2}[1.2]{$\rightsquigarrow$}} \mathcal{N}_k(m + \tau, M).
        \end{equation}
    \end{lemma}

    \begin{proof}[\bf Proof of Proposition~\ref{prop:iff.convergence.H0.H1n}]
        As suggested by a referee, this result can be proved using Lemma~7.6 of \cite{MR1652247} to establish the differentiability in quadratic mean of our density function under $H_0$. Using Theorem~2 of \cite{MR1652247}, this would imply a second order Taylor expansion akin to \eqref{eq:second.order.taylor.ratio.densities.H0} along with the convergence of the first and second order terms as in \eqref{eq:referee.first.order} and \eqref{eq:prop:iff.convergence.H0.H1n.eq.main}. This would then yield \eqref{eq:prop:iff.convergence.H0.H1n.convergence.ratio.measures} and the same argument using Le Cam's first lemma (see below \eqref{eq:prop:iff.convergence.H0.H1n.convergence.ratio.measures}) would finally give us the contiguity $\PP_{H_{\hspace{-0.1mm}0\hspace{-0.1mm},\hspace{-0.2mm}n}} \hspace{-0.5mm}\lhd \rhd \hspace{0.5mm} \PP_{H_{\hspace{-0.2mm}1\hspace{-0.1mm},\hspace{-0.2mm}n}}$. This proof is shorter but has the downside of involving the concept of differentiability in quadratic mean, which may slightly obscure the relation between the statement of Proposition~\ref{prop:iff.convergence.H0.H1n} and the aforementioned contiguity for some readers.

        Here is another straightforward (albeit lengthier) approach to the proof.
        We want to use Le Cam's first lemma.
        Assume that our vector of observations is the identity function
        \begin{equation}
            \bb{X} \leqdef (X_1,X_2,\dots,X_n) \leqdef \mathrm{Id} : (\Omega_n \leqdef \R^n, \mathcal{A}_n \leqdef \mathcal{L}(\R^n),\lambda) \longrightarrow (\R^n, \mathcal{B}(\R^n),\lambda),
        \end{equation}
        where $\mathcal{L}(\R^n)$ denotes the completion of the Borel $\sigma$-algebra $\mathcal{B}(\R^n)$, and where $\lambda$ denotes the Lebesgue measure.
        On $(\Omega_n,\mathcal{A}_n)$, define the probability measures
        \begin{equation}\label{eq:prop:iff.convergence.H0.H1n.def.Pn.Qn}
            \begin{aligned}
                &\PP_{H_{\hspace{-0.1mm}0\hspace{-0.1mm},\hspace{-0.2mm}n}}(A) \leqdef \int_A \, \prod_{i=1}^n f(X_i(\omega) \nvert \bb{\theta}_0,\bb{\kappa}) \, \rd \lambda(\omega), \quad A\in \mathcal{A}_n, \\
                &\PP_{H_{\hspace{-0.2mm}1\hspace{-0.1mm},\hspace{-0.2mm}n}}(A) \leqdef \int_A \, \prod_{i=1}^n f(X_i(\omega) \nvert \bb{\theta}_n,\bb{\kappa}) \, \rd \lambda(\omega), \quad A\in \mathcal{A}_n,
            \end{aligned}
        \end{equation}
        where $\bb{\theta}_n \leqdef \bb{\theta}_0 + (1 + o(1)) n^{-1/2} \bb{\delta}$ and $H_{\hspace{-0.1mm}0\hspace{-0.1mm},\hspace{-0.2mm}n} \leqdef H_0$.
        By construction, the law of $\bb{X}$ under $\PP_{H_{\hspace{-0.1mm}0\hspace{-0.1mm},\hspace{-0.2mm}n}}$ corresponds to the null hypothesis $H_0$ and the law $\bb{X}$ under $\PP_{H_{\hspace{-0.2mm}1\hspace{-0.1mm},\hspace{-0.2mm}n}}$ corresponds the alternative hypothesis $H_{\hspace{-0.2mm}1\hspace{-0.1mm},\hspace{-0.2mm}n}$.
        Since $f$ is positive on $\R$, the measures $\PP_{H_{\hspace{-0.1mm}0\hspace{-0.1mm},\hspace{-0.2mm}n}}$, $\PP_{H_{\hspace{-0.2mm}1\hspace{-0.1mm},\hspace{-0.2mm}n}}$ and $\lambda$ are equivalent on $(\Omega_n,\mathcal{A}_n)$.
        From \eqref{eq:prop:iff.convergence.H0.H1n.def.Pn.Qn}, we deduce that
        \begin{equation}
            \begin{aligned}
                \frac{\rd \PP_{H_{\hspace{-0.2mm}1\hspace{-0.1mm},\hspace{-0.2mm}n}}}{\rd \PP_{H_{\hspace{-0.1mm}0\hspace{-0.1mm},\hspace{-0.2mm}n}}} = \frac{\rd \PP_{H_{\hspace{-0.2mm}1\hspace{-0.1mm},\hspace{-0.2mm}n}} / \rd \lambda}{\rd \PP_{H_{\hspace{-0.1mm}0\hspace{-0.1mm},\hspace{-0.2mm}n}} / \rd \lambda} = \frac{\prod_{i=1}^n f(X_i \nvert \bb{\theta}_n,\bb{\kappa})}{\prod_{i=1}^n f(X_i \nvert \bb{\theta}_0,\bb{\kappa})} = \frac{\prod_{i=1}^n f(Y_i \nvert \bb{\theta}_n,(0,1)^{\top})}{\prod_{i=1}^n f(Y_i \nvert \bb{\theta}_0,(0,1)^{\top})},
            \end{aligned}
        \end{equation}
        where $Y_i \leqdef \sigma^{-1}(X_i - \mu)$.

        Using a second-order Taylor expansion around $\bb{\theta}_0$, we have, under $H_0 : X_i \sim \mathrm{APD}_{\lambda}(\bb{\theta}_0,\bb{\kappa})$,
        \begin{equation}\label{eq:second.order.taylor.ratio.densities.H0}
            \begin{aligned}
            &\log\left(\frac{\rd \PP_{H_{\hspace{-0.2mm}1\hspace{-0.1mm},\hspace{-0.2mm}n}}}{\rd \PP_{H_{\hspace{-0.1mm}0\hspace{-0.1mm},\hspace{-0.2mm}n}}}\right)
            = \sum_{i=1}^n (\log f(Y_i \nvert \bb{\theta}_n,(0,1)^{\top}) - \log f(Y_i \nvert \bb{\theta}_0,(0,1)^{\top})) \\
            &= (1 + o(1)) \, \bb{\delta}^{\top} \frac{1}{\sqrt{n}} \sum_{i=1}^n \bb{d}_{\bb{\theta}}(Y_i) + (1 + o(1))^2 \, \bb{\delta}^{\top} \hspace{-1mm}\int_0^1 \int_0^1 v \frac{1}{n} \sum_{i=1}^n \frac{\partial^2}{\partial \bb{\theta}^2} \log f(Y_i \nvert \bb{t}_{n,u,v},(0,1)^{\top}) \rd u \rd v \, \bb{\delta},
            \end{aligned}
        \end{equation}
        where $\bb{t}_{n,u,v} \leqdef \bb{\theta}_0 + uv (\bb{\theta}_n - \bb{\theta}_0)$.
        From the convergence of the first two components in \eqref{eq:asymptotic.normality.vector.d}, we know that, as $n\to\infty$,
        \begin{equation}\label{eq:referee.first.order}
            (1 + o(1))\, \bb{\delta}^{\top} \frac{1}{\sqrt{n}} \sum_{i=1}^n \bb{d}_{\bb{\theta}}(Y_i) \stackrel{\PP_{H_{\hspace{-0.1mm}0\hspace{-0.1mm},\hspace{-0.2mm}n}}}{\scalebox{2}[1.2]{$\rightsquigarrow$}} \mathcal{N}(0,\bb{\delta}^{\top} \hspace{-0.5mm}J_{\bb{\theta}\bb{\theta}} \, \bb{\delta}).
        \end{equation}
        For the second term on the right-hand side of \eqref{eq:second.order.taylor.ratio.densities.H0}, we want to apply a standard uniform law of large numbers (Lemma~\ref{lem:standard.ULLN}).
        From the expression of $f(y \nvert \bb{t},(0,1)^{\top})$ in (1), we see that for each $(j,k)\in \{1,2\}^2$, the function
        $U_{j,k}(y,\bb{t}) \leqdef \frac{\partial^2}{\partial \theta_j \partial \theta_k} \log f(y \nvert \bb{\theta},(0,1)^{\top})|_{\bb{\theta}=\bb{t}}$
        satisfies:
        \begin{description}
            \item[({\rm C.1})] For all $y\in \R$, $\bb{t}\mapsto U_{j,k}(y,\bb{t})$ is continuous on the compact $\mathcal{C} \leqdef [\frac{1}{4},\frac{3}{4}] \times [\frac{\lambda}{2},\frac{3\lambda}{2}]$;
            \item[({\rm C.2})] There exists a finite polynomial $K : \R\to \R$ such that $|U_{j,k}(y,\bb{t})| \leq K(|y|)$ for all $(y,\bb{t})\in \R \times \mathcal{C}$ (which implies that $K(|y|)$ is integrable under $f(y \nvert \bb{\theta}_0,(0,1)^{\top}) \rd y$).
        \end{description}
        Take $N\in \N$ large enough that $\bb{\theta}_n\in \mathcal{C}$ for all $n \geq N$.
        By Jensen's inequality and Lemma~\ref{lem:standard.ULLN} (under $H_0$), we deduce that
        \begin{equation}\label{eq:prop:iff.convergence.H0.H1n.eq.main}
            \begin{aligned}
                &\bigg|\int_0^1 \int_0^1 v \frac{1}{n} \sum_{i=1}^n U_{j,k}(Y_i,\bb{t}_{n,u,v}) \rd u \rd v - \int_0^1 \int_0^1 v \frac{1}{n} \sum_{i=1}^n \overline{U}_{j,k}(\bb{\theta}_0) \rd u \rd v\bigg| \\
                &\quad\leq \int_0^1 \int_0^1 v \frac{1}{n} \sum_{i=1}^n \left|U_{j,k}(Y_i,\bb{t}_{n,u,v}) - \overline{U}_{j,k}(\bb{\theta}_0)\right| \rd u \rd v \\
                &\quad\leq \frac{1}{2} \sup_{\bb{t}\in B_{\|\bb{\theta}_n - \bb{\theta}_0\|_2}[\bb{\theta}_0]} \frac{1}{n} \sum_{i=1}^n \left|U_{j,k}(Y_i,\bb{t}) - \overline{U}_{j,k}(\bb{\theta}_0)\right| \xrightarrow{\PP_{H_{\hspace{-0.1mm}0\hspace{-0.1mm},\hspace{-0.2mm}n}}} 0.
            \end{aligned}
        \end{equation}
        By definition of the matrix $J$ in \eqref{eq:asymptotic.normality.vector.d}, note that $\overline{U}_{j,k}(\bb{\theta_0}) = -J_{\theta_j\theta_k}$ (this can be seen by integrating by parts).
        Hence, \eqref{eq:prop:iff.convergence.H0.H1n.eq.main} shows that the second term on the right-hand side of \eqref{eq:second.order.taylor.ratio.densities.H0} is equal to $-\frac{1}{2} \bb{\delta}^{\top} \hspace{-0.5mm}J_{\bb{\theta}\bb{\theta}}\, \bb{\delta} + o_{\hspace{0.3mm}\PP_{H_{\hspace{-0.1mm}0\hspace{-0.1mm},\hspace{-0.2mm}n}}}(1)$.
        We deduce that
        \begin{equation}\label{eq:prop:iff.convergence.H0.H1n.convergence.ratio.measures}
            \log\left(\frac{\rd \PP_{H_{\hspace{-0.2mm}1\hspace{-0.1mm},\hspace{-0.2mm}n}}}{\rd \PP_{H_{\hspace{-0.1mm}0\hspace{-0.1mm},\hspace{-0.2mm}n}}}\right) \stackrel{\PP_{H_{\hspace{-0.1mm}0\hspace{-0.1mm},\hspace{-0.2mm}n}}}{\scalebox{2}[1.2]{$\rightsquigarrow$}} \mathcal{N}\left(-\frac{1}{2} \bb{\delta}^{\top} \hspace{-0.5mm}J_{\bb{\theta}\bb{\theta}}\, \bb{\delta}, \bb{\delta}^{\top}\hspace{-0.5mm}J_{\bb{\theta}\bb{\theta}}\, \bb{\delta}\right).
        \end{equation}
        Take any random variable $V > 0$ such that $\log(V)\stackrel{\PP_{H_{\hspace{-0.1mm}0\hspace{-0.1mm},\hspace{-0.2mm}n}}}{\sim} \mathcal{N}(-\frac{1}{2} \bb{\delta}^{\top} \hspace{-0.5mm}J_{\bb{\theta}\bb{\theta}}\, \bb{\delta}, \bb{\delta}^{\top} \hspace{-0.5mm}J_{\bb{\theta}\bb{\theta}}\, \bb{\delta})$.
        The continuous mapping theorem and \eqref{eq:prop:iff.convergence.H0.H1n.convergence.ratio.measures} imply that
        \begin{equation}
            \frac{\rd \PP_{H_{\hspace{-0.2mm}1\hspace{-0.1mm},\hspace{-0.2mm}n}}}{\rd \PP_{H_{\hspace{-0.1mm}0\hspace{-0.1mm},\hspace{-0.2mm}n}}} \stackrel{\PP_{H_{\hspace{-0.1mm}0\hspace{-0.1mm},\hspace{-0.2mm}n}}}{\scalebox{2}[1.2]{$\rightsquigarrow$}} V.
        \end{equation}
        By the definition of $V$, we have $\EE_{H_{\hspace{-0.1mm}0\hspace{-0.1mm},\hspace{-0.2mm}n}}[V] = 1$. This shows $(iii)$ in Lemma~\ref{lem:Le Cam.first.lemma} with $P_n = \PP_{H_{\hspace{-0.1mm}0\hspace{-0.1mm},\hspace{-0.2mm}n}}$ and $Q_n = \PP_{H_{\hspace{-0.2mm}1\hspace{-0.1mm},\hspace{-0.2mm}n}}$, which implies $\PP_{H_{\hspace{-0.2mm}1\hspace{-0.1mm},\hspace{-0.2mm}n}} \hspace{-0.5mm}\lhd \hspace{0.5mm} \PP_{H_{\hspace{-0.1mm}0\hspace{-0.1mm},\hspace{-0.2mm}n}}$ by $(i)$.
        Define $U \leqdef V$ and note that $\PP_{H_{\hspace{-0.1mm}0\hspace{-0.1mm},\hspace{-0.2mm}n}}(U > 0) = 1$ by definition of $V$.
        This shows $(ii)$ in Lemma~\ref{lem:Le Cam.first.lemma} where the roles of $P_n$ and $Q_n$ have been interchanged, which implies $\PP_{H_{\hspace{-0.1mm}0\hspace{-0.1mm},\hspace{-0.2mm}n}} \hspace{-0.5mm}\lhd \hspace{0.5mm} \PP_{H_{\hspace{-0.2mm}1\hspace{-0.1mm},\hspace{-0.2mm}n}}$ by $(i)$.
        We conclude that the sequences $(\PP_{H_{\hspace{-0.1mm}0\hspace{-0.1mm},\hspace{-0.2mm}n}}, n\in \N)$ and $(\PP_{H_{\hspace{-0.2mm}1\hspace{-0.1mm},\hspace{-0.2mm}n}}, n\in \N)$ are mutually contiguous, which we denote by $\PP_{H_{\hspace{-0.1mm}0\hspace{-0.1mm},\hspace{-0.2mm}n}} \hspace{-0.5mm}\lhd \rhd \hspace{0.5mm} \PP_{H_{\hspace{-0.2mm}1\hspace{-0.1mm},\hspace{-0.2mm}n}}$.
        The conclusion follows from $(iv)$.
    \end{proof}

    \begin{proof}[\bf Proof of Proposition~\ref{prop:asymptotic.normality.vector.d.H1}]
        From the expressions that we found for the two terms on the right-hand side of \eqref{eq:second.order.taylor.ratio.densities.H0} in the proof of Proposition~\ref{prop:iff.convergence.H0.H1n}, we have
        \begin{equation}
            \begin{pmatrix}
                \frac{1}{\sqrt{n}} \sum_{i=1}^n \bb{d}_{\bb{\theta}}(Y_i) \\[2mm]
                \frac{1}{\sqrt{n}} \sum_{i=1}^n \bb{d}_{\bb{\kappa}}(Y_i) \\[2mm]
                \log\left(\frac{\rd \PP_{H_{\hspace{-0.2mm}1\hspace{-0.1mm},\hspace{-0.2mm}n}}}{\rd \PP_{H_{\hspace{-0.1mm}0\hspace{-0.1mm},\hspace{-0.2mm}n}}}\right)
            \end{pmatrix}
            =
            \begin{pmatrix}
                \bb{0}_2 \\[2mm]
                \bb{0}_2 \\[2mm]
                -\frac{1}{2} \bb{\delta}^{\top} \hspace{-0.5mm}J_{\bb{\theta}\bb{\theta}}\, \bb{\delta} + o_{\hspace{0.3mm}\PP_{H_{\hspace{-0.1mm}0\hspace{-0.1mm},\hspace{-0.2mm}n}}}(1)
            \end{pmatrix}
            +
            \begin{pmatrix}
                \frac{1}{\sqrt{n}} \sum_{i=1}^n \bb{d}_{\bb{\theta}}(Y_i) \\[2mm]
                \frac{1}{\sqrt{n}} \sum_{i=1}^n \bb{d}_{\bb{\kappa}}(Y_i) \\[2mm]
                (1 + o(1)) \, \bb{\delta}^{\top} \frac{1}{\sqrt{n}} \sum_{i=1}^n \bb{d}_{\bb{\theta}}(Y_i)
            \end{pmatrix},
        \end{equation}
        where $\bb{0}_2 \leqdef (0,0)^{\top}$.
        By the central limit theorem (see the definition of $J$ in Proposition~\ref{prop:asymptotic.normality.vector.d}), we obtain that, under $H_0$,
        \begin{equation}
            \begin{pmatrix}
                \frac{1}{\sqrt{n}} \sum_{i=1}^n \bb{d}_{\bb{\theta}}(Y_i) \\[2mm]
                \frac{1}{\sqrt{n}} \sum_{i=1}^n \bb{d}_{\bb{\kappa}}(Y_i) \\[2mm]
                \log\left(\frac{\rd \PP_{H_{\hspace{-0.2mm}1\hspace{-0.1mm},\hspace{-0.2mm}n}}}{\rd \PP_{H_{\hspace{-0.1mm}0\hspace{-0.1mm},\hspace{-0.2mm}n}}}\right)
            \end{pmatrix}
            \stackrel{\PP_{H_0}}{\scalebox{2}[1.2]{$\rightsquigarrow$}}
            \mathcal{N}_5\left(
            \begin{pmatrix}
                \bb{0}_2 \\[2mm]
                \bb{0}_2 \\[2mm]
                -\frac{1}{2} \bb{\delta}^{\top} \hspace{-0.5mm}J_{\bb{\theta}\bb{\theta}}\, \bb{\delta}
            \end{pmatrix}
            ,
            \begin{pmatrix}
                J_{\bb{\theta}\bb{\theta}} &J_{\bb{\theta}\bb{\kappa}} &J_{\bb{\theta}\bb{\theta}} \bb{\delta} \\[2mm]
                J_{\bb{\theta}\bb{\kappa}}^{\top} &J_{\bb{\kappa}\bb{\kappa}} &J_{\bb{\theta}\bb{\kappa}}^{\top} \bb{\delta} \\[2mm]
                \bb{\delta}^{\top} J_{\bb{\theta}\bb{\theta}} &\bb{\delta}^{\top} J_{\bb{\theta}\bb{\kappa}} &\bb{\delta}^{\top} J_{\bb{\theta}\bb{\theta}} \bb{\delta}
            \end{pmatrix}
            \right).
        \end{equation}
        Then, by Le Cam's third lemma,
        \begin{equation}
            \begin{pmatrix}
                \frac{1}{\sqrt{n}} \sum_{i=1}^n \bb{d}_{\bb{\theta}}(Y_i) \\[2mm]
                \frac{1}{\sqrt{n}} \sum_{i=1}^n \bb{d}_{\bb{\kappa}}(Y_i)
            \end{pmatrix}
            \stackrel{\PP_{H_{\hspace{-0.2mm}1\hspace{-0.1mm},\hspace{-0.2mm}n}}}{\scalebox{2}[1.2]{$\rightsquigarrow$}}
            \mathcal{N}_4\left(
            \begin{pmatrix}
                J_{\bb{\theta}\bb{\theta}} \bb{\delta} \\[2mm]
                J_{\bb{\theta}\bb{\kappa}}^{\top} \bb{\delta}
            \end{pmatrix}
            ,
            \begin{pmatrix}
                J_{\bb{\theta}\bb{\theta}} &J_{\bb{\theta}\bb{\kappa}} \\[2mm]
                J_{\bb{\theta}\bb{\kappa}}^{\top} &J_{\bb{\kappa}\bb{\kappa}}
            \end{pmatrix}
            \right).
        \end{equation}
        This ends the proof.
    \end{proof}

\subsection{Proof of Corollary~3.11}\label{Proof_local}

    First, we apply Proposition~\ref{prop:iff.convergence.H0.H1n} to obtain a version of Proposition~\ref{lem:K.B.a.s.convergence} that is valid under $H_{\hspace{-0.2mm}1\hspace{-0.1mm},\hspace{-0.2mm}n}$.
    \begin{lemma}
        For $\lambda\geq 1$, we have, as $n\to \infty$,
        \begin{equation}
            K_{\lambda}(\Xn) \xrightarrow{\PP_{H_{\hspace{-0.2mm}1\hspace{-0.1mm},\hspace{-0.2mm}n}}} \frac{\lambda + \log \lambda + \psi(1/\lambda)}{\lambda}, \qquad S_{\lambda}(\Xn) \xrightarrow{\PP_{H_{\hspace{-0.2mm}1\hspace{-0.1mm},\hspace{-0.2mm}n}}} 0,
        \end{equation}
        so that
        \begin{equation}
           K^{\mathrm{net}}_{\lambda}(\Xn) \xrightarrow{\PP_{H_{\hspace{-0.2mm}1\hspace{-0.1mm},\hspace{-0.2mm}n}}} \frac{\lambda + \log \lambda + \psi(1/\lambda)}{\lambda} > 0.
        \end{equation}
    \end{lemma}
    Then, we can apply the same second order Taylor expansion idea we used in the proof of Proposition~\ref{Proof_ZKnetstar} (the proof is virtually identical so we omit the details) to prove that
    \begin{equation}
        \frac{Z^{*\!}(K^{\mathrm{net}}_{\lambda})}{Z^{*\!}(K_{\lambda})}\stackrel{\PP_{H_{\hspace{-0.2mm}1\hspace{-0.1mm},\hspace{-0.2mm}n}}}{\longrightarrow} 1, \quad \text{as } n\to \infty.
    \end{equation}
    Using the proof of Proposition~3.8 in Section~\ref{Proof_ZKnet}, in conjunction again with Proposition~\ref{prop:iff.convergence.H0.H1n}, we also have
    \begin{equation}
        \frac{Z(S_{\lambda})}{Z^{*\!}(S_{\lambda})}\stackrel{\mathrm{a.s.}}{\longrightarrow}1\quad \text{ and } \quad
        \frac{Z(K^{\mathrm{net}}_{\lambda})}{Z^{*\!}(K^{\mathrm{net}}_{\lambda})}\stackrel{\PP_{H_{\hspace{-0.2mm}1\hspace{-0.1mm},\hspace{-0.2mm}n}}}{\longrightarrow}1, \quad \text{as } n\to \infty.
    \end{equation}
    The last part of Corollary~3.11 follows directly using Slutsky's theorem and Theorem~3.10.

\end{appendices}

\section*{References}


\end{document}